\documentclass[10pt,letterpaper,twoside]{article}

\usepackage[letterpaper,left=1in,right=1in,top=1.5in,bottom=1.5in]{geometry}

\title{Cyclotomic synthetic spectra}
\author{Benjamin Antieau and Noah Riggenbach}
\date{\today}

\usepackage[pdfstartview=FitH,
            pdfauthor={},
            pdftitle={Cyclotomic synthetic spectra},
            colorlinks,
            linkcolor=reference,
            citecolor=citation,
            urlcolor=e-mail,
            backref]{hyperref}

\usepackage{amsmath}
\usepackage{amscd}
\usepackage{amsbsy}
\usepackage{amssymb}
\usepackage{verbatim}
\usepackage{eufrak}
\usepackage{eucal}
\usepackage{microtype}
\usepackage{hyperref}
\usepackage{mathrsfs}
\usepackage{amsthm}
\usepackage{stmaryrd}
\usepackage{xy}
\input xy
\xyoption{all}
\usepackage{tikz}
\usetikzlibrary{matrix,arrows}
\usepackage{tikz-cd}
\usepackage{dsfont}
\usepackage[T1]{fontenc}
\usepackage{enumitem}
\setlist{noitemsep}

\usepackage{fancyhdr}
\newcommand{\stackspace}{2.5}
\newcommand{\stack}[2][1cm]{\;\tikz[baseline, yshift=.65ex]%
    {\foreach \k [evaluate=\k as \r using (.5*#2+.5-\k)*\stackspace] in {1,...,#2}{%
    \ifodd\k{\draw[->](0,\r pt)--(#1,\r pt);}%
    \else{\draw[<-](0,\r pt)--(#1,\r pt);}\fi
    }}\;}

\usepackage[bbgreekl]{mathbbol}
\usepackage{amsfonts}
\DeclareSymbolFontAlphabet{\mathbb}{AMSb} 
\DeclareSymbolFontAlphabet{\mathbbl}{bbold}

\usepackage{yhmath}

\newcommand{\prism}{{ \mathbbl{\Delta}}}

\usepackage{color}
\definecolor{todo}{rgb}{1,0,0}
\definecolor{conditional}{rgb}{0,1,0}
\definecolor{e-mail}{rgb}{0,.40,.80}
\definecolor{reference}{rgb}{.20,.60,.22}
\definecolor{mrnumber}{rgb}{.80,.40,0}
\definecolor{citation}{rgb}{0,.40,.80}

\renewcommand{\bf}{\bfseries}

\let\oldmarginpar\marginpar
\renewcommand\marginpar[1]{\-\oldmarginpar[\raggedleft\footnotesize #1]%
{\raggedright\footnotesize #1}}

\newcommand{\Cscr}{\mathcal{C}}

\newcommand{\B}{\mathrm{B}}
\newcommand{\C}{\mathrm{C}}
\renewcommand{\d}{\mathrm{d}}
\newcommand{\D}{\mathrm{D}}
\newcommand{\E}{\mathrm{E}}
\newcommand{\F}{\mathrm{F}}
\newcommand{\G}{\mathrm{G}}
\renewcommand{\H}{\mathrm{H}}
\newcommand{\h}{\mathrm{h}}

\newcommand{\K}{\mathrm{K}}
\renewcommand{\L}{\mathrm{L}}

\newcommand{\N}{\mathrm{N}}
\renewcommand{\P}{\mathrm{P}}

\newcommand{\T}{\mathrm{T}}
\renewcommand{\t}{\mathrm{t}}

\newcommand{\W}{\mathrm{W}}

\newcommand{\bC}{\mathbf{C}}
\newcommand{\bD}{\mathbf{D}}
\newcommand{\bE}{\mathbf{E}}
\newcommand{\bF}{\mathbf{F}}
\newcommand{\bG}{\mathbf{G}}

\newcommand{\bN}{\mathbf{N}}

\newcommand{\bQ}{\mathbf{Q}}

\newcommand{\bS}{\mathbf{S}}
\newcommand{\bT}{\mathbf{T}}

\newcommand{\bZ}{\mathbf{Z}}

\newcommand{\Sp}{\mathrm{Sp}}

\newcommand{\op}{\mathrm{op}}
\newcommand{\cofib}{\mathrm{cofib}}
\newcommand{\fib}{\mathrm{fib}}

\newcommand{\proj}{\mathrm{proj}}

\newcommand{\GrMod}{\mathrm{GrMod}}

\newcommand{\Mod}{\mathrm{Mod}}
\newcommand{\coMod}{\mathrm{coMod}}

\newcommand{\CAlg}{\mathrm{CAlg}}

\newcommand{\Gr}{\mathrm{Gr}}
\newcommand{\gr}{\mathrm{gr}}
\newcommand{\ins}{\mathrm{ins}}
\newcommand{\fil}{\mathrm{fil}}
\newcommand{\Fil}{\mathrm{Fil}}
\newcommand{\FD}{\mathrm{FD}}

\newcommand{\cyc}{\mathrm{cyc}}
\newcommand{\Nm}{\mathrm{Nm}}
\newcommand{\CQSyn}{\mathrm{CQSyn}}
\newcommand{\Fr}{\mathrm{Fr}}
\newcommand{\Eq}{\mathrm{Eq}}
\newcommand{\triv}{\mathrm{triv}}
\newcommand{\lev}{\mathrm{lev}}
\renewcommand{\part}{\mathrm{part}}
\newcommand{\mot}{\mathrm{mot}}
\newcommand{\Fgeqs}{\F^{\geq\star}}

\newcommand{\Syn}{\mathrm{Syn}}
\newcommand{\SynSp}{\mathrm{SynSp}}
\newcommand{\Pstragowski}{Pstr\k{a}gowski}
\newcommand{\CycSyn}{\mathrm{CycSyn}}
\newcommand{\syncart}[1][]{\mathrm{CartSyn}_{{#1}}}
\newcommand{\CycSp}{\mathrm{CycSp}}
\newcommand{\syn}{\mathrm{syn}}
\newcommand{\ev}{\mathrm{ev}}

\newcommand{\can}{\mathrm{can}}

\newcommand{\heart}{\heartsuit}

\newcommand{\id}{\mathrm{id}}

\renewcommand{\geq}{\geqslant}
\renewcommand{\leq}{\leqslant}

\newcommand{\HC}{\mathrm{HC}}
\newcommand{\THH}{\mathrm{THH}}

\newcommand{\HH}{\mathrm{HH}}

\newcommand{\TP}{\mathrm{TP}}
\newcommand{\TC}{\mathrm{TC}}

\newcommand{\TR}{\mathrm{TR}}

\newcommand{\BMS}{\mathrm{BMS}}
\newcommand{\BLM}{\mathrm{BLM}}
\newcommand{\HRW}{\mathrm{HRW}}

\newcommand{\QSyn}{\mathrm{QSyn}}

\newcommand{\MU}{\mathrm{MU}}
 
\newcommand{\Map}{\mathrm{Map}}
\newcommand{\bMap}{\mathbf{Map}}

\newcommand{\Fun}{\mathrm{Fun}}

\newcommand{\Gm}{\bG_{m}}

\DeclareMathOperator*{\colim}{colim}

\DeclareMathOperator*{\Tot}{Tot}
\newcommand{\LEq}{\mathrm{LEq}}

\DeclareMathOperator{\Spf}{Spf}

\newcommand{\we}{\simeq}
\newcommand{\iso}{\cong}

\theoremstyle{plain}
\newtheorem{theorem}{Theorem}[section]
\newtheorem*{theorem*}{Theorem}
\newtheorem{lemma}[theorem]{Lemma}

\newtheorem{proposition}[theorem]{Proposition}

\newtheorem{conjecture}[theorem]{Conjecture}
\newtheorem{corollary}[theorem]{Corollary}
\newtheorem*{corollary*}{Corollary}

\theoremstyle{plain}
\newcounter{zaehler}

\newtheorem{introthm}[zaehler]{Theorem}
\newtheorem{introcor}[zaehler]{Corollary}

\theoremstyle{plain}

\theoremstyle{definition}

\newtheoremstyle{named}{}{}{\itshape}{}{\bfseries}{.}{.5em}{#1 \thmnote{#3}}
\theoremstyle{named}

\theoremstyle{definition}
\newtheorem{definition}[theorem]{Definition}
\newtheorem{warning}[theorem]{Warning}
\newtheorem{variant}[theorem]{Variant}

\newtheorem{notation}[theorem]{Notation}

\newtheorem{example}[theorem]{Example}
\newtheorem*{example*}{Example}

\newtheorem{question}[theorem]{Question}
\newtheorem*{question*}{Question}
\newtheorem{construction}[theorem]{Construction}

\newtheorem{remark}[theorem]{Remark}

\begin{document}

\maketitle
\begin{abstract}
    \noindent
    We define an $\infty$-category $\CycSyn$ of $p$-typical cyclotomic synthetic spectra and prove
    that the motivic filtration on $\THH(R;\bZ_p)$, defined by Bhatt, Morrow,
    and Scholze when $R$ is quasisyntomic and by Hahn, Raksit, and Wilson in the chromatically
    quasisyntomic case, naturally admits the structure of a
    $p$-typical cyclotomic synthetic spectrum. As a consequence, we obtain new bounds on the
    syntomic cohomology of connective chromatically quasisyntomic $\bE_\infty$-ring spectra.
\end{abstract}
\tableofcontents

\section{Introduction}\label{sec:intro}

If $R$ is a ring or a ring spectrum, then its topological Hochschild homology
$\THH(R)$ naturally admits the structure of a cyclotomic spectrum~\cite{nikolaus-scholze}. This
structure allows one to define $\TC(R)$, the topological cyclic homology of
$R$, which is the target of the trace map $\K(R)\rightarrow\TC(R)$ forming the
basis of many recent computations in algebraic $K$-theory.

If $R$ is commutative and $p$-quasisyntomic, then Bhatt, Morrow, and Scholze
introduced a motivic filtration $\F^{\geq\star}_\mot\THH(R;\bZ_p)$ on the $p$-completed
topological Hochschild homology of $R$ as well as motivic filtrations on
$\TC^-(R;\bZ_p)$, $\TP(R;\bZ_p)$, and $\TC(R;\bZ_p)$. The associated graded
pieces of these filtrations can be used to define prismatic and syntomic
cohomology theories. The existence of these filtrations and their connection to
$\K$-theory provides additional techniques for computing $\TC(R;\bZ_p)$ and
$\K$-groups. However, this motivic filtration is not a filtration in cyclotomic spectra.

A related motivic filtration has been constructed in~\cite{hrw} by Hahn, Raksit, and Wilson on $\THH(R;\bZ_p)$
for certain $\bE_\infty$-ring spectra, which they call chromatically $p$-quasisyntomic.
For the purpose of this introduction, we will also write $\F^{\geq\star}_\mot\THH(R;\bZ_p)$.
Integral versions of these filtrations exist thanks to~\cite{bhatt-lurie-apc,hrw,morin}.

This paper gives an explanation of what is $\F^{\geq\star}_\mot\THH(R;\bZ_p)$.
To explain our answer, we recall that an algebraic analogue of
motivically-filtered $\THH$ is given by $\HH_\fil(R/\bZ)$, which is the
Hochschild homology of the commutative ring $R$ together with its
HKR-filtration. This is a filtered spectrum with $S^1$-action, but in fact
admits more structure: it is a filtered spectrum with $\bT_\fil$-action, where
$\bT_\fil=\tau_{\geq\star}\bZ[S^1]$. The filtered circle was introduced in this
form in Raksit~\cite{raksit} and in a closely related form in work of
Moulinos--Robalo--To\"en~\cite{moulinos-robalo-toen}.

A minimum requirement for $\F^{\geq\star}_\mot\THH(R;\bZ_p)$ is that it should
be a filtered spectrum with filtered circle action and where the cyclotomic
Frobenius $\varphi_p\colon\THH(R;\bZ_p)\rightarrow\THH(R;\bZ_p)^{\t C_p}$
respects the filtrations and filtered circle-actions.
It is well-known that there cannot be a lift of $\bT_\fil$ to the sphere
spectrum (see Remark~\ref{rem:no_spherical_Tfil}). However,
$\F^{\geq\star}_\mot\THH(R;\bZ_p)$ is naturally a synthetic spectrum and,
following a suggestion of Raksit, there is a lift of $\bT_\fil$ to a synthetic
spectrum $\bT_\ev$. This lift is also studied by Hedenlund and Moulinos in~\cite{hedenlund-moulinos}.

We introduce the stable $\infty$-category $\SynSp_{\bT_\ev}$ of synthetic
spectra with $\bT_\ev$-action and construct synthetic analogues of homotopy orbits, fixed
points, and Tate for $C_n$ and $S^1$. This allows us to introduce the
$\infty$-category $\CycSyn$ of $p$-typical cyclotomic synthetic spectra as the $\infty$-category of
pairs $(M,\varphi_p)$ where $M$ is a synthetic spectrum with $\bT_\ev$-action
and $\varphi_p\colon M\rightarrow M^{\t C_{p,\ev}}$ is a $\bT_\ev$-equivariant map,
where $M^{\t C_{p,\ev}}$ is viewed as a synthetic spectrum with $\bT_\ev$-action via
restriction of scalars along the multiplication-by-$p$ map
$\bT_\ev\rightarrow\bT_\ev$. We do not discuss the more general notion of cyclotomic synthetic
spectrum with Frobenii for all primes, so we will call $p$-typical cyclotomic synthetic spectra simply
cyclotomic synthetic spectra.

With this definition of cyclotomic synthetic spectra, we establish the following result.

\begin{introthm}
    If $R$ is a $p$-quasisyntomic commutative ring, then
    the motivic filtration $\F^{\geq\star}_\mot\THH(R;\bZ_p)$ naturally admits
    the structure of an $\bE_\infty$-algebra in $\CycSyn$. If $R$
    is chromatically $p$-quasisyntomic, then $\F^{\geq\star}_\mot\THH(R;\bZ_p)$ naturally admits
    the structure of an $\bE_\infty$-algebra in $\CycSyn$.
\end{introthm}

There is a synthetic analogue of $\TC$, which is an exact functor
$\CycSyn\rightarrow\SynSp$ and we find that
$\TC(\F^{\geq\star}_\mot\THH(R;\bZ_p))\we\F^{\geq\star}_\mot\TC(R;\bZ_p)$.
Thus, having constructed the cyclotomic synthetic structure on
$\F^{\geq\star}_\mot\THH(R;\bZ_p)$, we recover in addition the motivic
filtration on $\TC(R;\bZ_p)$. Similar results hold for $\TC^-(R;\bZ_p)$ and
$\TP(R;\bZ_p)$ and there are non-$p$-completed results as well in the case that $R$ is integrally
quasisyntomic. See Theorems~\ref{thm:even_comparison} and~\ref{thm:bms_comparison} for details.

Synthetic spectra admit many natural $t$-structures $\C$. To each one which satisfies a reasonable
list of conditions we associate a $t$-structure on $\CycSyn$, generalizing the work of~\cite{an1}.
We will write $\pi_*^{\cyc,\C}$ for the homotopy objects with respect to the $t$-structure induced
by $\C$. Perhaps the most interesting is the
Postnikov $t$-structure $\CycSyn_{\geq 0}^\P$, in which case we will write $\pi_*^{\cyc,\P}$. To analyze this $t$-structure, we give a theory of
Cartier modules in synthetic spectra.

In~\cite[Def.~3.34]{an1}, a notion of Cartier complex was given, consisting of graded abelian group $M^*$
together with operations $\d,\eta\colon M^*\rightarrow M^{*+1}$ and $F,V\colon M^*\rightarrow
M^{*}$ satisfying various relations which hold in de Rham--Witt complexes. Below, we introduce
the notion of an $\eta$-deformed ($p$-typical) Cartier complex, which agrees with the notion of a Cartier complex
except that we do not require that $\eta^4=0$. (These have also been studied by Krause and Nikolaus in
unpublished work.) There is an abelian category $\mathrm{DCart}^\eta$
of $\eta$-deformed Cartier complexes as well as a full subcategory $\mathrm{DCart}^{\eta,\wedge}$
of derived $V$-complete $\eta$-deformed Cartier complexes.

\begin{introthm}\label{thm:intro_heart}
    There is an equivalence $\CycSyn^{\P\heart}\we\mathrm{DCart}^{\eta,\wedge}$.
\end{introthm}

\begin{introcor}\label{cor:intro_drw}
    Suppose that $R$ is a smooth commutative $\bF_p$-algebra.
    Under the equivalence of Theorem~\ref{thm:intro_heart},
    $\pi_0^{\cyc,\P}(\F^{\geq\star}_\mot\THH(R))\iso\W\Omega^{\bullet}_R$, the de Rham--Witt
    complex of $R$.
\end{introcor}

The corollary itself has some interesting consequences, such as the fact that
$$\W\Omega^\bullet_R\widehat{\boxtimes}_{\bZ_p}\W\Omega^\bullet_S\iso\W\Omega^\bullet_{R\otimes_{\bF_p}S},$$
where the tensor product is the $V$-completed tensor product on $\mathrm{DCart}^{\eta,\wedge}$
induced from the compatibility of the Postnikov cyclotomic synthetic $t$-structure with the
symmetric monoidal structure. This is a generalization of the
isomorphism $W(R)\widehat{\boxtimes}W(S)\iso W(R\otimes S)$ proved in~\cite{an1}.

\begin{introcor}
    If $S$ is a smooth commutative $\bF_p$-algebra, then
    $\tau_{\geq\star}\TR(S)$ admits the structure of an $\bE_\infty$-algebra
    with $\bT_\fil$-action.
\end{introcor}

It is natural to wonder if $\TR(S)$ in fact admits the structure of an
filtered animated commutative $\bZ_p$-algebra with $\bT_\fil$-action in the
sense of~\cite{raksit}. It has been announced by Nikolaus that
$\TR(S)$ is an animated commutative ring, so we are left to wonder about
the compatibility of this structure with the filtration and the circle
action. This will be clarified by forthcoming work of Bachmann--Burklund
who will show that $\F^{\geq\star}_\mot\THH(R;\bZ_p)$ admits a universal
property in the context of $\bG_m$-equivariant normed motivic algebras. The
connection here is that the $\infty$-category of (even) synthetic spectra
is equivalent (after $p$-completion) to the $\infty$-category of
$p$-complete cellular motivic spectra over $\bC$. Under this equivalence,
our synthetic circle $\bT_\ev$ is corresponds to $\Gm$. There is
a notion of normed motivic algebra due to
Bachmann--Hoyois~\cite{bachmann-hoyois} which
deforms the notion of derived commutative ring and
$\F^{\geq\star}_\mot\THH(R;\bZ_p)$ is the universal $p$-complete
normed motivic algebra with a $\Gm$-action attached to $R$.

Using the connection between crystals and de Rham--Witt connections established
by Bloch~\cite{Bloch_Crystals}, we use Corollary~\ref{cor:intro_drw} to prove the following
result.

\begin{introthm}~\label{thm: thmA}
    Let $k$ be a perfect $\bF_p$-algebra and let $A$ be a smooth $k$-algebra.
    There is a fully faithful embedding
    \[\mathrm{FBT}_A^{\op}\hookrightarrow{} \CycSyn_{\F^{\geq
    *}_{\ev}\THH(A;\bZ_p)}\] from the opposite category of formal $p$-divisible
    groups over $A$ to the $\infty$-category of $\F^{\geq *}_{\ev}\THH(A;\bZ_p)$-modules in $\CycSyn$.
\end{introthm}

We were motivated by the following question for which we have currently no more
than ad hoc constructions.

\begin{question}
    Is there a deformation of the functor of Theorem~\ref{thm: thmA} to flat
    $\F^{\geq\star}_\mot\THH(R;\bZ_p)$-modules in $\CycSyn$?
\end{question}

There is also a functor defined on a certain subcategory of
$\F^{\geq\star}_\mot\THH(R;\bZ_p)$-modules in cyclotomic synthetic spectra to
quasicoherent sheaves on $(\Spf R)^{\syn}$, the syntomification of $R$ in the sense of Bhatt, Lurie, and
Drinfeld~\cite{bhatt-fgauges,drinfeld-prismatization}. This source of prismatic $F$-gauges was one of the original
motivations for this work and was discussed at length with Jacob Lurie who has
discovered a compatible $t$-structure on $F$-gauges.

We conclude the paper by giving a synthetic analogue of the Beilinson fiber
square as established in~\cite{ammn} as well as a synthetic analogue of a result connecting the
image of $j$ spectrum to $\THH(\bZ_p)$ due to Devalapurkar and Raksit~\cite{devalapurkar-raksit}.
The former allows for an extension of some of the main results
of~\cite{ammn} to the case of chromatically $p$-quasisyntomic ring spectra.

\begin{introthm}\label{introthm:syntomic_bounds}
    Let $R$ be a connective chromatically $p$-quasisyntomic $\bE_\infty$-ring spectrum. Then,
    $\gr^i_\mot\TC(R;\bZ_p)\in\D(\bZ_p)_{[i-1,2i]}$ for all $i\geq 0$.
\end{introthm}

As a special case, when $R$ is discrete, Theorem~\ref{introthm:syntomic_bounds} recovers
Theorem~5.1(1) of~\cite{ammn}.

\subsection{Acknowledgments}
We thank Sanath Devalapurkar, Paul Georss, Jeremy Hahn, Alice Hedenlund, Achim Krause, Jacob Lurie, Deven
Manam, Tasos Moulinos, Thomas Nikolaus, Arpon Raksit, and Dylan Wilson for many useful
conversations. In particular, Raksit originally suggested to us the definition
of the synthetic circle and the definition of filtered
$C_n$-fixed points of a synthetic spectrum with $\bT_\ev$-action was suggested to BA by
Krause and Nikolaus at Oberwolfach in 2022. Special thanks go to Alice Hedenlund and Tasos Moulinos
for comments on a draft and for correcting an important misconception.

We were supported by NSF grants DMS-2120005,
DMS-2102010, and DMS-2152235, by Simons Fellowships 666565 and 00005925, and by the Simons
Collaboration on Perfection.
BA additionally thanks the Institute for Advanced Study for its hospitality
and support via NSF grant DMS-1926686 as well as the Max Planck Institute for Mathematics in Bonn,
where this work was completed.

\section{The synthetic circle}\label{sec:even_circle}

In this section,
we recall some facts about synthetic spectra~\cite{gikr,pstragowski-synthetic} and introduce the synthetic analogue of the circle, $\bT_\ev$, which is
constructed after a suggestion of Raksit using the even filtration~\cite{hrw,pstragowski-even} and
is a deformation to synthetic spectra of the filtered circle of~\cite{moulinos-robalo-toen,raksit}.
See also the work of Hedenlund--Moulinos~\cite{hedenlund-moulinos} which has substantial overlap with the material in
this section.

\subsection{Synthetic spectra and $t$-structures}

Synthetic spectra were introduced in the context of $K(n)$-local homotopy
theory by Hopkins and Lurie in~\cite{hopkins-lurie-brauer} and generalized by
\Pstragowski{} in~\cite{pstragowski-synthetic}. We take an alternative approach via filtered spectra due
to Gheorghe, Isaksen, Krause, and Ricka~\cite{gikr}. We touch upon the
equivalence of these approaches below in Remark~\ref{rem:history}.

\begin{definition}[Filtered spectra]
    Let $\bZ^\op$ be the poset of decreasing integers.
    We let $\Sp=\D(\bS)$ denote the $\infty$-category of spectra, and we let
    $\F\Sp=\FD(\bS)=\Fun(\bZ^{\op},\D(\bS))$ be the $\infty$-category of (decreasing) filtered spectra.
    Both $\infty$-categories are symmetric monoidal, $\D(\bS)$ with the usual tensor (or smash) product
    $(-)\otimes_{\bS}(-)$ of spectra and $\FD(\bS)$ with the Day convolution symmetric monoidal
    structure, induced from the symmetric monoidal structure on $\D(\bS)$ and the symmetric monoidal
    structure on $\bZ^\op$ arising from addition.

    If $\Fgeqs M$ and $\Fgeqs N$ are two filtered
    spectra, then the tensor product $\Fgeqs O=(\Fgeqs M)\otimes_\bS(\Fgeqs N)$
    is given by $\F^{\geq n} O\we\colim_{i+j\geq
    n}\F^{\geq i}M\otimes_\bS\F^{\geq j}N$. For more details, see~\cite{bms2,
    gwilliam-pavlov,moulinos,raksit}.

    Given $i\in\bZ$, we let $\ins^i\colon\D(\bS)\rightarrow\FD(\bS)$ denote the
    left Kan extension along the inclusion $\{i\}\hookrightarrow\bZ^\op$. As
    the Yoneda functor $\bZ^\op\rightarrow\FD(\bS)$ is symmetric monoidal,
    $\ins^0\bS$ is the unit for the symmetric monoidal structure on $\FD(\bS)$.
    We will typically write $\bS$ for $\ins^0\bS$.

    There is also the constant filtration functor
    $c\colon\D(\bZ)\rightarrow\FD(\bS)$ obtained by pullback along
    $\bZ^\op\rightarrow\ast$. Finally, there is the colimit, or underlying
    spectrum functor, $$\F^{-\infty}\colon\FD(\bS)\rightarrow\D(\bS)$$ given by
    $\F^{\geq -\infty}(\F^{\geq\star}M)=\colim_{i\rightarrow-\infty}\F^{\geq i}M$.
    We will write this more succinctly as $\F^{\geq -\infty}M$ or $|\F^{\geq\star}M|$.

    Finally, say that $\F^{\geq\star}M$ is complete if $\lim_i\F^{\geq i}M\we
    0$. The full subcategory of complete filtered spectra is denoted by
    $\widehat{\FD}(\bS)\subseteq\FD(\bS)$. It is a symmetric monoidal
    localization of the $\infty$-category of filtered spectra and we denote the
    completed tensor product by $(-)\widehat{\otimes}(-)$.
\end{definition}

There are several $t$-structures on the $\infty$-category of filtered spectra
of interest to us.

\begin{construction}[The neutral $t$-structure]
    As $\FD(\bS)=\Fun(\bZ^\op,\D(\bS))$ is a functor category, it inherits a
    $t$-structure from the standard $t$-structure on $\D(\bS)$, called the
    neutral $t$-structure, in which a
    filtered spectrum $\F^{\geq\star} M$ is connective (or coconnective) in the canonical $t$-structure
    if and only if $\F^{\geq i} M\in\D(\bS)_{\geq 0}$ (or in $\D(\bS)_{\leq
    0}$) for all $i$. This $t$-structure is accessible and compatible with
    filtered colimits and is moreover compatible with the symmetric
    monoidal structure on $\FD(\bS)$ and the functor taking $\Fgeqs M$ to the filtered abelian
    group $i\mapsto \pi_0\F^{\geq i}M$ induces a symmetric monoidal equivalence
    $$\pi_0\F^{\geq\star}\colon\FD(\bS)^{\N\heart}\we\F\Mod_{\bZ}^{\D\otimes},$$
    where $\F\Mod_\bZ^{\D\otimes}$ is the abelian category of filtered abelian
    groups with the Day convolution symmetric monoidal
    structure.\footnote{We do not require
    the transition maps $\F^{i+1}M\rightarrow\F^iM$ in a filtered abelian group
    to be injective.}
    We write $(\FD(\bS)_{\geq 0}^\N,\FD(\bS)_{\leq 0}^\N)$ for this
    $t$-structure.
\end{construction}

\begin{construction}[The Postnikov $t$-structure]
    Let $\FD(\bS)_{\geq 0}^\P\subseteq\FD(\bS)$ be the full subcategory consisting of filtered spectra
    $\F^\star M$ such that $\F^iM\in\D(\bS)_{\geq i}$ for all $i\in\bZ$. Similarly, let
    $\FD(\bS)_{\leq 0}^\P\subseteq\FD(\bS)$ consist of those $\F^\star M$ where $\F^iM\in\D(\bS)_{\leq
    i}$ for all $i\in\bZ$. The pair $(\FD(\bS)^\P_{\geq 0},\FD(\bS)^\P_{\leq 0})$ defines a $t$-structure
    on $\FD(\bS)$, the {\bf Postnikov $t$-structure}. Again, this $t$-structure
    is accessible, compatible with filtered colimits, and compatible with the
    symmetric monoidal structure on $\FD(\bS)$. The functor
    $\pi_*\F^*$, which takes a filtered spectrum $\F^\star M$ to the graded abelian group
    $i\mapsto\pi_i\F^iM$, induces a symmetric monoidal equivalence
    \begin{equation}\label{eq:postnikov_heart}
        \pi_*\F^*\colon\FD(\bS)^{\P\heart}\we\Gr\Mod_\bZ^{\K\otimes},
    \end{equation}
    where $\Gr\Mod_\bZ^{\K\heart}$ denotes the abelian category of graded
    abelian groups with the Koszul symmetric monoidal structure.
\end{construction}

In order to define synthetic spectra, we first define the synthetic sphere
spectrum as a filtered $\bE_\infty$-ring. To do, we give a way of building some
non-trivial filtered $\bE_\infty$-rings.

\begin{lemma}\label{lem:lax_monoidal}
    The functor $\tau_{\geq \star}\colon\D(\bS)\rightarrow\FD(\bS)$ admits a
    natural lax symmetric monoidal structure such that the induced symmetric
    monoidal structure on the composition
    $\id_{\D(\bS)}\we\F^{-\infty}\circ\tau_{\geq\star}$ is the canonical one.
\end{lemma}

\begin{proof}
    The functor $\tau_{\geq\star}\colon\D(\bS)\rightarrow\FD(\bS)$ is obtained as a composition
    $$\D(\bS)\xrightarrow{c}\FD(\bS)\xrightarrow{\tau_{\geq 0}^\P}\FD(\bS)^{\P}_{\geq
    0}\hookrightarrow\FD(\bS),$$
    where $c\colon\D(\bS)\rightarrow\FD(\bZ)$ is the constant filtration functor.
    The truncation $\tau_{\geq 0}^\P$ and the inclusion $\FD(\bS)^{\P}_{\geq
    0}\hookrightarrow\FD(\bS)$ are lax symmetric monoidal and symmetric monoidal, respectively, as the Postnikov $t$-structure is compatible
    with the symmetric monoidal structure. Finally, the constant filtration
    functor is lax symmetric monoidal; in fact it is a fully faithful symmetric
    monoidal functor onto the category of $c(\bS)$-modules in $\FD(\bS)$.
\end{proof}

\begin{corollary}\label{cor:double_lax}
    The functor $\tau_{\geq 2\star}\colon\D(\bS)\rightarrow\FD(\bS)$ admits a lax symmetric
    monoidal structure such that the induced symmetric
    monoidal structure on the composition
    $\id_{\D(\bS)}\we\F^{-\infty}\circ\tau_{\geq 2\star}$ is the canonical one.
\end{corollary}

\begin{proof}
    This follows from Lemma~\ref{lem:lax_monoidal} and the fact that the multiplication-by-$2$
    functor $\bZ^\op\xrightarrow{2}\bZ^\op$ is symmetric monoidal, so
    restriction along it induces a symmetric monoidal functor $\FD(\bS)\rightarrow\FD(\bS)$ sending
    $\F^\star M$ to $\F^{2\star}M$.
\end{proof}

\begin{corollary}\label{cor:double_speed_ring}
    If $A$ is an $\bE_\infty$-ring spectrum, then $\tau_{\geq 2\star}A$ is naturally an $\bE_\infty$-algebra
    in filtered spectra.
\end{corollary}

\begin{definition}[The synthetic sphere spectrum]
    Let $\MU^\bullet$ be the \v{C}ech complex associated to $\bS\rightarrow\MU$, viewed as a
    cosimplicial $\bE_\infty$-ring, so that $\MU^0\we\MU$,
    $\MU^1\we\MU\otimes_{\bS}\MU$, and so on. We can apply $\tau_{\geq2\star}$ pointwise to obtain a
    cosimplicial $\bE_\infty$-algebra $\tau_{\geq 2\star}\MU^\bullet$ in filtered spectra using
    Corollary~\ref{cor:double_speed_ring}. We let
    $\bS_\ev$ denote the limit in $\CAlg(\FD(\bS))$, the $\infty$-category of
    $\bE_\infty$-algebras in filtered spectra. We call $\bS_\ev$ the {\bf synthetic sphere
    spectrum}, or the {\bf even sphere spectrum}. It is complete with
    underlying object given by $\bS$.
\end{definition}

\begin{remark}
    The $\E^1$-page of the spectral sequence associated to $\bS_\ev$ is the
    $\E^2$-page of the Adams--Novikov spectral sequence.
\end{remark}

To explain the terminology, recall that a spectrum $X$ is {\bf even} if $\pi_{2i-1}X=0$ for all $i\in\bZ$.
An $\bE_\infty$-ring is even if it is even as a spectrum.

\begin{variant}[Alternative constructions]\label{var:even}
    If $A$ is any $\bE_\infty$-ring, Hahn--Raksit--Wilson~\cite{hrw} introduce the even filtration
    $$\F^{\geq\star}_\ev A=\lim_{\text{$A\to B$ even}}\tau_{\geq 2\star}B,$$
    where the limit is over all maps of $\bE_\infty$-rings $A\rightarrow B$ where $B$ is even and is
    taken in $\bE_\infty$-algebras in filtered spectra using Corollary~\ref{cor:double_speed_ring}.
    They prove in~\cite[Cor.~2.2.17]{hrw} that $\F^{\geq\star}_\ev A$ can be computed using {\em Novikov
    descent}, which shows in particular that $\F^{\geq\star}_\ev\bS\we\bS_\ev$. In~\cite{pstragowski-even},
    \Pstragowski{} introduced another even filtration construction and shows that it agrees with that of~\cite{hrw}
    in many cases, including for the sphere spectrum.
\end{variant}

\begin{notation}
    We will write $\F^{\geq\star}_\ev A$ and $A_\ev$ interchangeably for the even
    filtration on an $\bE_\infty$-ring spectrum $A$.
\end{notation}

\begin{definition}[Synthetic spectra]
    We define $\SynSp$ to be the $\infty$-category $\Mod_{\bS_\ev}\FD(\bS)=\FD(\bS_\ev)$ of $\bS_\ev$-modules in
    filtered spectra. Objects of $\SynSp$ are called {\bf synthetic spectra} in this paper.
\end{definition}

\begin{definition}
    The evaluation at zero functor $\FD(\bS)\to \mathrm{Sp}$ from synthetic spectra to spectra
    admits a symmetric monoidal left adjoint $\ins^0\colon\mathrm{Sp}\to\FD(\bS)$. Concretely this
    functor is given by \[\F^n\ins^0M\we\begin{cases}
        0 & \textrm{ if }n>0\\
        M &\textrm{ if }n\leq 0,
    \end{cases}\] with all transition maps $\F^n\ins^0M\rightarrow\F^{n-1}\ins^0M$ being
    equivalent to the
    identity on $M$ for $n\leq 0$. If $M$ is in $\D(\bZ)$, then $\ins^0 M$ naturally admits an
    $\ins^0\bZ$-module structure and hence a
    $\bS_{\ev}$-module structure via restriction of scalars along $\bS_\ev\rightarrow\ins^0\bZ$.
\end{definition}

\begin{remark}[Bounds]\label{rem:bounds}
    If $A$ is an $\bE_\infty$-ring, then $\gr^i_\ev A$ is in $\D(\bS)_{\leq 2i}$ by construction. If
    $A$ is connective, then $\gr^i_\ev A$ is in fact $i$-connective so that it is in
    $\D(\bS)_{[i,2i]}$. This is a result of Burklund--Krause (private communication). However, when
    the even filtration coincides with the filtration of~\cite{pstragowski-even}, as is the case for
    the sphere spectrum and the other connective $\bE_\infty$-ring spectra studied in~\cite{hrw}, it also follows
    from~\cite[Thm.~1.7]{pstragowski-even}.
\end{remark}


\begin{remark}[History]\label{rem:history}
    The terminology of synthetic spectra originates in~\cite{hopkins-lurie-brauer} where it is
    introduced by Hopkins--Lurie in a special case. The idea and terminology is
    generalized by \Pstragowski{}
    in~\cite{pstragowski-synthetic} who introduces a stable $\infty$-category $\mathrm{Syn}_E$ for
    any $\bE_\infty$-ring spectrum $E$ and proves that a certain full subcategory
    $\mathrm{Syn}_\MU^\ev\subseteq\mathrm{Syn}_\MU$ is equivalent to the stable $\infty$-category of
    cellular motivic spectra
    over $\bC$ after $p$-completion at any prime $p$. The connection to filtered spectra is
    due to Gheorghe--Isaksen--Krause--Ricka~\cite{gikr} who study $\bS_\ev$ as defined above and show that $\FD(\bS_\ev)$ is
    equivalent to the stable $\infty$-category of cellular motivic spectra over $\bC$ after
    $2$-completion. They remark in~\cite[Rem.~6.13]{gikr} that it is possible to directly compare
    $\FD(\bS_\ev)$ and \Pstragowski's $\mathrm{Syn}_E$. Gregoric
    in~\cite{gregoric-moduli} compares $\FD(\bS_\ev)$ to so-called ind-coherent sheaves on
    the connective cover of the moduli stack of oriented formal groups.

    In more detail,
    \Pstragowski's $\infty$-category of even $\MU$-based synthetic spectra is
    defined as follows. One begins with the $\infty$-category
    $\Sp_\MU^{\mathrm{fpe}}$ of finite spectra $X$ such that $\MU_*X$ is even
    and projective as a graded $\MU_*$-module spectrum, which is given the topology where the coverings
    are maps $X\rightarrow Y$ such that $\MU_*X\rightarrow\MU_*Y$ is
    surjective.
    Then, $\Syn_\MU^\ev$ is
    defined to be the full subcategory of sheaves of spectra
    $(\Sp_\MU^{\mathrm{fpe}})^\op\rightarrow\Sp$ which preserve finite products.
    The symmetric monoidal structure on $\Syn_\MU^\ev$ is given by Day
    convolution arising from the smash product of spectra, which restricts to a
    symmetric monoidal structure on $\Sp_\MU^{\mathrm{fpe}}$. There is also a
    $t$-structure on $\Syn_\MU^\ev$ whose connective part consists of the full
    subcategory of sheaves of anima which preserve finite products.
    \Pstragowski{} defines the synthetic analogue functor
    $\nu\colon\Sp\rightarrow\Syn_\MU^\ev$, which is additive but not exact, as the composition of the
    (restricted) Yoneda
    functor $h\colon\Sp\rightarrow(\Syn_\MU^\ev)_{\geq 0}$ with the fully
    faithful suspension
    spectrum functor $\Sigma^\infty_+\colon(\Syn_\MU^\ev)_{\geq
    0}\rightarrow\Syn_\MU^\ev$.
    \Pstragowski{} shows that $\Syn_\MU^\ev$ is generated under colimits by
    the bigraded spheres $\bS^{t,2w}=\nu(\bS[2w])[t-2w]$.

    There is an equivalence $\Syn_\MU^\ev\we\SynSp$ discovered
    in~\cite{gikr}. Under this equivalence $\bS^{t,2w}$ corresponds to
    $\bS_\ev(w)[2w][t-2w]\we\bS_\ev(w)[t]$.

    Under the equivalence $\Syn_\MU^\ev\we\SynSp$, the canonical sheaf
    $t$-structure on $\Syn_\MU^\ev$ is given in terms of filtered
    $\bS_\ev$-modules as follows.
    We say that
    $\F^{\geq\star}M\in\SynSp$ is $\MU$-connective if
    $$\F^{\geq
    i}\left(\F^{\geq\star}
    M\otimes_{\bS_\ev}\MU_\ev\right)$$ is $(2i)$-connective for all $i\in\bZ$.
    Write $\SynSp_{\geq 0}^\MU$ for the connective objects, which are the
    connective part of a $t$-structure on $\SynSp$, which we will call the
    synthetic $t$-structure. \Pstragowski{} proves that $\SynSp^{\MU\heart}$ is
    equivalent to the symmetric monoidal abelian category of
    $\MU_*\MU$-comodules.
\end{remark}

\begin{construction}[The Postnikov $t$-structure on synthetic spectra]~\label{const: Postnikov t-structure on synthetic spectra}
    It follows from Remark~\ref{rem:bounds} that $\bS_\ev$ is connective in the Postnikov
    $t$-structure on $\FD(\bS)$. Thus, $\SynSp=\FD(\bS_\ev)$ inherits a Postnikov
    $t$-structure as well, where a synthetic spectrum is connective or coconnective if its
    underlying filtered spectrum is. This $t$-structure is accessible,
    compatible with filtered colimits, and compatible with the symmetric
    monoidal structure in $\SynSp$. It follows that
    $\SynSp^{\P\heart}\we\Mod_{\pi_0^\P\bS_\ev}(\FD(\bS)^{\P\heart})$. But,
    $\pi_*\F^*(\pi_0^\P\bS_\ev)\iso\pi_*\F^*\bS_\ev\iso\pi_*\gr^*\bS_\ev\iso\bZ[\eta]/2\eta$, where
    $\eta$ denotes the Hopf map, of weight $1$, and the second isomorphism
    follows from the connectivity bounds of Remark~\ref{rem:bounds}. Therefore, taking $\pi_*\F^*$ induces a symmetric monoidal equivalence
    $$\pi_*\F^*\colon\SynSp^\heart\we\Gr\Mod_{\bZ[\eta]/2\eta}$$
    using~\eqref{eq:postnikov_heart}, where
    $\Gr\Mod_{\bZ[\eta]/2\eta}^{\K\otimes}$ is the abelian category of graded
    $\bZ[\eta]/2\eta$-modules with the Koszul symmetric monoidal structure.
    The action of $\eta$ induces the slope $1$ lines visible in Adams--Novikov charts, such as can
    be found in~\cite{isaksen-wang-xu-charts}.
\end{construction}

\begin{remark}
    If we restrict our attention to synthetic spectra over
    $\bS[1/2]_\ev$, then the heart of the Postnikov $t$-structure is equivalent
    to graded modules over $\bZ[1/2]$, since $\pi_0^\P\bS[1/2]_\ev\iso\bZ[1/2]$.
\end{remark}

\begin{lemma}~\label{lem:uniqueness of ring structure on whitehead tower.}
    Let $R$ be an $\bE_\infty$-ring. Then the space of $\bE_\infty$-ring structures on $\tau_{\geq 2\star}R$ compatible with the one on $R$ is contractible.
\end{lemma}
\begin{proof}
    Consider the diagram of adjoints
    \[
    \begin{tikzcd}
\CAlg(\FD(\bS)_{\geq 0}^{2\P}) \arrow[rr, "i", shift left] \arrow[rrdd, "\F^{\geq -\infty}", shift
        left] &  & \CAlg(\FD(\bS)) \arrow[ll, "\tau_{\geq 0}^{2\P}", shift left] \arrow[dd, "\F^{\geq -\infty}"', shift right] \\
                                                                                                     &  &                                                                                                        \\
                                                                                                     &
                                                                                                     &
                                                                                                     \CAlg(\Sp)
                                                                                                     \arrow[uu,
                                                                                                     "\mathrm{const}"',
                                                                                                     shift
                                                                                                     right] \arrow[lluu, "\tau_{\geq 2\star}", shift left=2]  
\end{tikzcd}
    \]
    and note that the counit of the diagonal adjunction is an equivalence. Thus, the functor
    $\tau_{\geq 2\star}:\CAlg(\Sp)\to \CAlg(\FD(\bS)_{\geq 0}^{2\P})$ is fully faithful. The inclusion $i$ is fully faithful as well, whence the result.
\end{proof}

In the remainder of this section, we include some complements on the even
filtration from~\cite{hrw}.

\begin{lemma}[The even filtration is lax symmetric monoidal]\label{lem:even_monoidal}
    The functor $\F^{\geq\star}_\ev\colon\CAlg(\D(\bS))\rightarrow\CAlg(\FD(\bS))$ is lax symmetric
    monoidal.
\end{lemma}

\begin{proof}
    This follows from the construction of Variant~\ref{var:even}, which shows
    that $\F^{\geq\star}_\ev$ is
    indeed a functor from $\CAlg(\D(\bS))$ to $\CAlg(\FD(\bS))$ together with the fact that the
    symmetric monoidal structure on each side is cocartesian; any functor between
    $\infty$-categories with finite coproducts is lax symmetric
    mononidal with respect to cocartesian symmetric monoidal structures (see~\cite[Prop.~2.4.3.8]{ha}).
\end{proof}

\begin{corollary}
    The functor $\F^{\geq\star}_\ev$ takes values in $\CAlg(\SynSp)$, the $\infty$-category of
    $\bE_\infty$-algebras in synthetic spectra, and the induced functor
    $\F^{\geq\star}_\ev\colon\CAlg(\Sp)\rightarrow\CAlg(\SynSp)$ is lax symmetric monoidal.
\end{corollary}

\begin{proof}
    By Lemma~\ref{lem:even_monoidal}, each $\F^{\geq\star}_\ev A$ is naturally an
    $\bE_\infty$-algebra over
    $\F^{\geq\star}_\ev\bS\we\bS_\ev$ and hence naturally an $\bE_\infty$-algebra object in $\SynSp$.
    The proof of lax symmetric monoidality is the same as for Lemma~\ref{lem:even_monoidal}.
\end{proof}

\begin{notation}
    If $C$ is an $\bE_\infty$-ring, we will sometimes write $C_\ev$ for
    $\F^{\geq\star}_\ev C$, the
    associated $\bE_\infty$-algebra in synthetic spectra.
\end{notation}

\begin{remark}[Adams resolutions are synthetic]
    If $E$ is an $\bE_\infty$-algebra over $\MU$, then the \v{C}ech complex $E^\bullet$ of
    $\bS\rightarrow E$ admits a natural map $\MU^\bullet\rightarrow E^\bullet$. It follows that
    $\Tot(\tau_{\geq 2\star}E^\bullet)$ is an $\bE_\infty$-algebra over $\Tot(\tau_{\geq
    2\star}\MU^\bullet)\we\bS_\ev$, i.e., an $\bE_\infty$-algebra in synthetic spectra.
\end{remark}

To analyze the even circle, we will need the notion of even faithfully flat maps from~\cite{hrw}.

\begin{definition}[Even faithfully flat]\label{def:eff}
    A map $A\rightarrow B$ of $\bE_\infty$-rings is {\bf even faithfully flat} or {\bf eff} 
    if, for every map $A\rightarrow C$ of $\bE_\infty$-rings where $C$ is even, the pushout
    $C\otimes_AB$ is even and $\oplus_{*\in\bZ}\pi_*(C\otimes_AB)$ is a faithfully flat module over
    the commutative ring $\oplus_{*\in\bZ}\pi_*C$.
\end{definition}

\begin{notation}\label{not:graded}
    In the situation of Definition~\ref{def:eff}, we will say that $\pi_*(C\otimes_AB)$ is
    faithfully flat over $\pi_*C$ if $\oplus_{*\in\bZ}\pi_*(C\otimes_AB)$ is faithfully flat over
    $\oplus_{*\in\bZ}\pi_*C$, in conflict with the usual notion of flatness for graded modules,
    which we will not use.
\end{notation}

\begin{remark}[Base change for eff maps]
    If $A\rightarrow B$ is eff and $A\rightarrow C$ is an arbitrary map of $\bE_\infty$-rings, then
    the induced map $C\rightarrow C\otimes_AB$ is eff.
\end{remark}

\begin{example}
    The unit map $\bS\rightarrow\MU$ is eff.
\end{example}

\begin{definition}[Locally even]
    An $\bE_\infty$-ring $A$ is eff-locally even if it admits an eff map $A\rightarrow B$ where $B$
    is even.
\end{definition}

The following lemma is~\cite[Cor.~2.2.14(1)]{hrw}.

\begin{lemma}\label{lem:local_computation}
    Suppose that $A\rightarrow B$ is an eff map of $\bE_\infty$-rings with \v{C}ech complex $B^\bullet$.
    \begin{enumerate}
        \item[{\em (a)}] The natural map $\F^{\geq\star}_\ev
            A\rightarrow\Tot(\F^{\geq\star}_\ev B^\bullet)$ is
            an equivalence.
        \item[{\em (b)}] If $B$ is even, so that $A$ is eff-locally even, then the natural map
            $\F^{\geq\star}_\ev A\rightarrow\Tot(\tau_{\geq 2\star}B^\bullet)$ is an equivalence.
    \end{enumerate}
\end{lemma}

\begin{variant}[Discretely $p$-completely eff]\label{var:pcompletely}
    There are $p$-complete variants of the notions above. The one studied in~\cite{hrw} is the
    following. Say that a map $A\rightarrow B$ of commutative rings is discretely $p$-completely
    faithfully flat if for every commutative ring $C$ the $p$-completed (derived) pushout
    $(C\otimes_AB)_p^\wedge$ is discrete and $p$-completely faithfully flat over $C$, which means
    that $C/p\otimes_AB$ is discrete and is a faithfully flat $C/p$-module. Say that a map
    $A\rightarrow B$ of $\bE_\infty$-rings is discretely $p$-completely eff if for every even
    $\bE_\infty$-ring $C$ the $p$-completed pushout $(C\otimes_AB)_p^\wedge$ is even and
    $\pi_*(C\otimes_AB)_p^\wedge$ is discretely $p$-completely faithfully flat over $\pi_*C$ in the
    spirit of Notation~\ref{not:graded}.
\end{variant}

\subsection{The synthetic circle}

If $B$ is an $\bE_\infty$-ring, we let $B[S^1]\we B\otimes_\bS\bS[S^1]\we
B\otimes_{\bS}\Sigma^\infty_+S^1$ be the group ring of the anima $S^1$ over $B$, which is itself an
$\bE_\infty$-ring and in fact is a bicommutative bialgebra, meaning an
$\bE_\infty$-coalgebra in $\CAlg(\D(C))$. Let $B[S^1]^\vee$ be the $B$-linear dual of $B[S^1]$, which is equivalent to
$B^{S^1_+}$ by Atiyah duality and is again a bicommutative bialgebra.
Raksit introduced a filtered version of the group ring of $S^1$ circle
in~\cite{raksit} when $C=\bZ$; it appears in a related
way in~\cite{moulinos-robalo-toen}.

\begin{definition}[The filtered circle]
    Let $\bT=\bZ[S^1]$ and define $\bT_{\fil}=\tau_{\geq\star}\bT$. By Lemma~\ref{lem:lax_monoidal},
    the $\bE_\infty$-algebra structure on $\bT$ induces an $\bE_\infty$-algebra structure on
    $\bT_{\fil}$. However, $\bT$ admits the structure of a bicommutative bialgebra, i.e., it admits an
    $\bE_\infty$-coalgebra structure in $\CAlg(\D(\bZ))$. This too is inherited by $\bT_{\fil}$ as
    $\tau_{\geq\star}$ is symmetric monoidal on the full subcategory of
    $\D(\bZ)$ consisting of
    $\bT^{\otimes n}$ for $n\in\bZ$. The bicommutative bialgebra $\bT_{\fil}$ is called the {\bf
    filtered circle}. Its dual is $\bT_\fil^\vee$, which is also a bicommutative bialgebra.
\end{definition}

\begin{remark}[No $\bT_{\fil}$ over $\bS$]\label{rem:no_spherical_Tfil}
    It is well-known that $\bT_{\fil}$ does not lift to the sphere spectrum. This means that there is no 
    way to put a filtration on $\bS[S^1]$ with the property that $\bS$ has weight $0$ and the
    fundamental class $d$ of $S^1$ has weight $1$. Indeed, $d^2=\eta d$ in $\pi_*\bS[S^1]$, see \cite[Section 1.1.3]{Nikolaus_Hesselholt}, which
    leads to a contradiction by considering weights.
\end{remark}

We will see below that $\bT_{\fil}$ can also be constructed as the even filtration on $\bT$.
This motivates the following definition, suggested to us by Raksit.

\begin{definition}[Raksit]
    Let $\bT_{\ev}\in \CAlg(\SynSp)$ denote the $\bE_\infty$-algebra
    $\F^{\geq\star}_\ev\bS[S^1]$ in
    synthetic spectra. This is the {\bf synthetic circle}, or the {\bf even circle}.
    Let $\bT_\ev^\vee$ denote the internal dual $\bMap_{\bS_\ev}(\bT_\ev, \bS_\ev)$.
\end{definition}

\begin{warning}
    The dual $\bT_\ev^\vee$ is not $(\bS^{S^1_+})_\ev$. In fact, this fails $\bZ$-linearly.
\end{warning}

By lax symmetric monoidality, $\bT_\ev$ is an $\bE_\infty$-algebra in $\SynSp$.
Our goal below is to construct a bicommutative bialgebra structure on $\bT_\ev$ extending the
$\bE_\infty$-algebra structure
which, together with dualizability of $\bT_\ev$, will produce such a structure on $\bT_\ev^\vee$ as well.
To do so, we need some preliminary results.

\begin{lemma}\label{lem:even_s1}
    If $B$ is an even $\bE_\infty$-ring, then the augmentation map $B[S^1]\rightarrow B$ is eff.
\end{lemma}

\begin{proof}
    We follow the proof of~\cite[Lem.~4.2.5]{hrw}.
    Suppose that $R$ is an even $\bE_\infty$-ring under $B[S^1]$. Let $T=R\otimes_{B[S^1]}B$. The
    Tor spectral sequence computing $\pi_*T$ has $\E^1$-page given by the graded homotopy groups of
    $$\pi_*R\otimes^\L_{\pi_*B[S^1]}\pi_*B\iso\pi_*R\otimes_{\pi_*B}\pi_*B\otimes^\L_{\pi_*B[S^1]}\pi_*B\iso\pi_*R\otimes_{\pi_*B}\Gamma_{\pi_*B}(\sigma
    d),$$
    since the generator $d$ of $\pi_1B[S^1]$ must map to zero in $\pi_*R$ by evenness.
    As the divided power algebra $\Gamma_{\pi_*B}(\sigma d)$ is free as a $\pi_*B$-module on even
    classes, the spectral sequence degenerates and $T$ is even with $\pi_* T$ faithfully flat
    over $\pi_*R$, as desired.
\end{proof}



\begin{corollary}
    If $B$ is an eff-locally even $\bE_\infty$-ring, then $B[S^1]$ is eff-locally even.
\end{corollary}

\begin{proof}
    For example, if $B\rightarrow C$ is eff where $C$ is even, use the
    composition $B[S^1]\rightarrow C[S^1]\rightarrow C$ to see that $B[S^1]$ is
    locally eff.
\end{proof}

\begin{lemma}\label{lem:eff_even_descent}
    Let $A\rightarrow B\rightarrow C$ be maps of $\bE_\infty$-rings. Suppose that for every even
    $\bE_\infty$-ring $R$ under $A$ the induced map $R\otimes_AB\rightarrow R\otimes_AC$ is eff.
    Then, $B\rightarrow C$ is eff.
\end{lemma}

\begin{proof}
    Given $B\rightarrow R$ where $R$ is even, we can consider $R$ as an $\bE_\infty$-$A$-algebra and
    the use the induced map $R\otimes_AB\rightarrow R$ to construct a commutative diagram
    $$\xymatrix{
        B\ar[r]\ar[d]&C\ar[d]\\
        R\otimes_AB\ar[r]\ar[d]&R\otimes_AC\ar[d]\\
        R\ar[r]&R\otimes_BC.
    }$$
    All three squares are pushout squares. By hypothesis, the middle horizontal map
    is eff. Since $R$ is even, so is $R\otimes_BC$.
\end{proof}

\begin{corollary}\label{cor:augmentation_eff}
    For any $\bE_\infty$-ring $B$, the natural map $B[S^1]\rightarrow B$ is eff.
\end{corollary}

\begin{proof}
    By base change for eff maps, it is enough to prove the lemma when $B=\bS$. This case follows from
    Lemma~\ref{lem:eff_even_descent} applied to $\bS\rightarrow\bS[S^1]\rightarrow\bS$ using
    Lemma~\ref{lem:even_s1}.
\end{proof}

Now, we analyze the even filtration of $\bZ[S^1]$, showing that it agrees with Raksit's
$\bT_{\fil}=\tau_{\geq\star}(\bZ[S^1])$.

\begin{lemma}\label{lem:tfil}
    There is a natural equivalence $\F^{\geq\star}_\ev(\bZ[S^1])\we\bT_{\fil}$.
\end{lemma}

\begin{proof}
    By the result of Burklund and Krause quoted above in
    Remark~\ref{rem:bounds} or by~\cite[Thm.~1.7]{pstragowski-even}, $\F^{\geq\star}_\ev(\bZ[S^1])$ is connective in the Postnikov
    $t$-structure, so there is a natural map
    $$\F^{\geq\star}_\ev(\bZ[S^1])\rightarrow\bT_{\fil}=\tau_{\geq\star}\bZ[S^1].$$
    Consider the augmentation $\bT_{\fil}\rightarrow\bZ$ in filtered $\bE_\infty$-rings. Let
    $\bZ^\bullet_F$ be the \v{C}ech complex of the augmentation and let $\bZ^\bullet$ be the
    \v{C}ech complex of $\bZ[S^1]\rightarrow\bZ$. Note that these augmentations
    are descendable in the sense of~\cite{mathew-galois}; in particular, the natural map $\bT_{\fil}\rightarrow\Tot(\bZ^\bullet_F)$ is
    an equivalence. On the other hand, by Corollary~\ref{cor:augmentation_eff},
    $\F^{\geq\star}_\ev(\bZ[S^1])\we\Tot(\tau_{\geq 2\star}\bZ^\bullet_F)$.
    Note now that $\tau_{\geq 2\star}\bZ^\bullet$ is naturally equivalent, as
    a cosimplicial filtered $\bE_\infty$-algebra over $\bZ$, to
    $\bZ_F^\bullet$. This completes the proof.
\end{proof}

\begin{remark}\label{rem:z_bullet}
    Consider the \v{C}ech complex of $\bZ[S^1]\rightarrow\bZ$ denoted by $\bZ^\bullet$ in the proof
    of Lemma~\ref{lem:tfil}. In cosimplicial degree $n$, it is given by $\bZ[(BS^1)^n]$. Thus, the
    cosimplicial object $\bZ^\bullet$ of $\D(\bZ)$ takes the form
    \[
        \bZ  \stack{3} \bZ[BS^1]  \stack{5} \bZ[BS^1\times BS^1]\stack{7}\ldots.
    \]
    Applying $\pi_{2i}$ pointwise we obtain a cosimplicial abelian group which computes
    $\gr^i\bT_{\fil}\we(\pi_i\bZ[S^1])[i]$. For $i\geq 2$, this object is equivalent to zero, so it
    follows that the cochain complex associated to the cosimplicial abelian group $\pi_{2i}\bZ^\bullet$
    is exact. When $i=0$, $\pi_{0}\bZ^\bullet$ is the constant cosimplicial
    diagram on $\bZ$ and,
    when $i=1$, $\pi_2\bZ^\bullet$ is equivalent to $\bZ[1]$.
    It follows that the (double-speed) $\Tot$ spectral sequence
    $$\E_2^{s,t}=\H^{s-t}(\pi_{-2t}\bZ^\bullet)\Rightarrow\pi_{-s-t}\bZ[S^1]$$
    collapses at the $\E_2$-page, where it is zero except for $(s,t)=(0,0)$ and $(s,t)=(0,1)$.
\end{remark}

\begin{lemma}\label{lem:even_chains}
    If $C$ is an even $\bE_\infty$-ring, then there is a canonical equivalence
    $\F^{\geq\star}_\ev(C[S^1])\we\tau_{\geq 2\star-1}C[S^1]$ of filtered $C_\ev$-module spectra.
    In particular, $\F^{\geq\star}_\ev(C[S^1])$ is an exhaustive filtration on $C[S^1]$.
\end{lemma}

\begin{proof}
    Let $G^\bullet$ be the \v{C}ech complex of $S^1\rightarrow\ast$ in grouplike
    $\bE_\infty$-spaces, so that the \v{C}ech complex
    $$C^\bullet\colon C\stack{3}C[BS^1]\stack{5}\cdots$$ of $C[S^1]\rightarrow C$ is equivalent to
    $C[G^\bullet]$.
    The filtration $\tau_{\geq 2\circledast}C$ on $C$ induces a filtration $(\tau_{\geq
    2\circledast}C)[G^\bullet]$, which we will write as $\F^{\geq\circledast} C^\bullet$.
    Taking double-speed Whitehead towers pointwise yields a bifiltered
    cosimplicial spectrum $\F^{\geq\star}\F^{\geq\circledast}
    C^\bullet$ with $$\F^{\geq i}\F^{\geq j} C^\bullet\colon \tau_{\geq 2i}(\tau_{\geq 2j}C)\stack{3}\tau_{\geq
    2i}(\tau_{\geq 2j}(C[BS^1]))\stack{5}\cdots.$$
    By Lemma~\ref{lem:even_s1}, the object $\Tot(\F^{\geq\star}\F^{\geq -\infty}C^\bullet)$ is
    equivalent to $\F^{\geq\star}_\ev(C[S^1])$.
    We have $$\F^{\geq\star}\gr^j C^\bullet\we\tau_{\geq 2\star}(\bZ^\bullet\otimes_{\bZ}\pi_{2j}C[2j]),$$
    where $\bZ^\bullet$ is the \v{C}ech complex of $\bZ[S^1]\rightarrow\bZ$ analyzed in
    Lemma~\ref{lem:tfil} and Remark~\ref{rem:z_bullet}. Thus,
    $$\Tot(\F^{\geq\star}\gr^j
    C^\bullet)\we(\bT_{\fil}\otimes_\bZ\pi_{2j}C)(j)[2j].$$
    Let $\F^{\geq\circledast}\F^{\geq\star}_\ev(C[S^1])$ be the filtration on
    $\F^{\geq\star}_\ev(C[S^1])$ induced
    by $\F^{\geq\circledast} C^\bullet$. On $\gr^i_\ev(C[S^1])$, the induced filtration
    $\F^{\geq \circledast}\gr^i_\ev(C[S^1])$ is complete and exhaustive, as follows by considering the
    filtration on the complex associated to the cosimplicial abelian group $$\pi_{2i}C\stack{3}\pi_{2i}(C[BS^1])\stack{5}\cdots$$ induced by
    $\tau_{\geq 2\circledast}C$. As
    $$\gr^j_\circledast\gr^i_\ev(C[S^1])\we\gr^i\left(\bT_{\fil}\otimes_\bZ\pi_{2j}C)(j)[2j]\right)\we\begin{cases}
        \pi_{2j}C[2j]&\text{if $i=j$,}\\
        \pi_{2j}C[2j+1]&\text{if $i=j+1$,}\\
        0&\text{otherwise,}
    \end{cases}$$
    it follows that $\gr^i_\ev(C[S^1])$ fits into a fiber sequence
    $$\pi_{2i-2}C[2i-1]\rightarrow\gr^i_\ev(C[S^1])\rightarrow\pi_{2i}C[2i].$$
    It follows from an argument using the sparsity of the homotopy groups that
    $\F^{\geq\star}_\ev
    (C[S^1])\we\tau_{\geq 2\star-1}\F^{\geq-\infty}_\ev C[S^1]$ is an equivalence, from which we conclude
    that the natural map $C[S^1]\rightarrow\F^{\geq-\infty}_\ev C[S^1]$ is an equivalence, which
    completes the proof.
\end{proof}

\begin{lemma}~\label{lem: fiber sequence for Tsyn}
    If $B$ is an eff-locally even $\bE_\infty$-ring,
    then there is a split fiber sequence \[B_{\ev}(1)[1]\to
    \F^{\geq\star}_{\ev}(B[S^1])\to B_\ev\] of synthetic spectra where the
    second map is the augmentation map.
\end{lemma}

\begin{proof}
    We have already shown that this is true for $B$ replaced with an even eff cover $C$.
    Then, by Lemma~\ref{lem:even_s1} and descent, the result follows for $B$.
\end{proof}

\begin{remark}
    Lemma~\ref{lem: fiber sequence for Tsyn} makes no claim about the structure
    of the fiber $B_\ev(1)[1]$ as a synthetic $B[S^1]_\ev$-module. In fact, if
    $B$ is even, then the fiber is equivalent to $B_\ev(1)[1]$ as a
    $B[S^1]_\ev$-module; but, if $B=\bS$, then it is not.
\end{remark}

\begin{corollary}\label{cor:dualizable}
    The synthetic spectrum $\bT_\ev$ is dualizable in $\SynSp$.
\end{corollary}

\begin{proof}
    Dualizablility follows immediately from Lemma~\ref{lem: fiber sequence for
    Tsyn} as $\bS_\ev$, $\bS_\ev(1)[1]$, and $\bS_\ev(-1)[-1]$ are dualizable.
\end{proof}

\begin{proposition}\label{Prop: symmetric monoidal}
    If $B$ is an eff-locally even $\bE_\infty$-ring, then the natural map
    $\bT_{\ev}\otimes_{\bS_\ev}\F^{\geq\star}_\ev
    B\rightarrow\F^{\geq\star}_\ev(B[S^1])$ 
    is an equivalence.
\end{proposition}

\begin{proof}
    Let $B\to C$ be an eff cover of $B$ by an even ring $C$.
    Then $B[S^1]\to C[S^1]$ and $B^{S^1_+}\to C^{S^1_+}$ is an eff cover.
    By definition of eff it follows that $C^{\otimes_B n}$
    is even for all $n$ and so we have that
    \[\F^{\geq\star}_{\ev}(C^{\otimes_{B}n}[S^1])\simeq \tau_{\geq
    2\star-1}(C^{\otimes_{B}n}[S^1])\] 
    for all $n$ naturally.

    Applying eff descent we have that 
    \begin{align*}
        \F^{\geq\star}_{\ev}(B[S^1]) &\simeq \mathrm{Tot}(\tau_{\geq 2\star-1}(C^{\otimes_B n}[S^1]))\\
        &\simeq \mathrm{Tot}((\tau_{\geq 2\star}C^{\otimes_{B}n})\oplus
        \tau_{\geq 2\star-2}(C^{\otimes_B n})[1])\\
        &\simeq \mathrm{Tot}(\tau_{\geq 2\star}C^{\otimes_B n})\oplus
        \mathrm{Tot}(\tau_{\geq 2\star-2}C^{\otimes_B n})[1]\\
        &\simeq \F^{\geq\star}_{\ev}(B)\oplus \F^{\geq\star-1}_{\ev}(B)[1] 
    \end{align*}
    as desired.
\end{proof}

\begin{corollary}\label{cor:bialgebra}
    There is a natural bicommutative bialgebra structure on $\bT_{\ev}$ in $\SynSp$ compatible
    with the $\bE_\infty$-algebra  structure on $\bT_\ev$ and with the property that
    $\bZ\otimes_{\bS_\ev}\bT_\ev\we\bT_\fil$ as bicommutative bialgebras over $\bZ$.
\end{corollary}

\begin{proof}
    Let $\mathcal{C}\subseteq \CAlg(\Sp)$
    be the full subcategory spanned by objects of the form $\bS[S^1]^{\otimes
    m}$ for $m\in \bN$.
    The $\infty$-category $\mathcal{C}$ admits a symmetric monoidal structure given by the
    restriction of the usual one on $\CAlg(\Sp)$
    and for this symmetric monoidal
    structure $\bS[S^1]$ is a bicommutative bialgebra.
    It follows from Proposition~\ref{Prop: symmetric monoidal} that
    $\F_\ev^{\geq\star}(-)\colon\Cscr\rightarrow\CAlg(\SynSp)$
    is symmetric monoidal, whence the bicommutative bialgebra structure on
    $\bT_\ev$.
\end{proof}

\begin{corollary}
    There is a natural bicommutative bialgebra structure on $\bT_\ev^\vee$ in $\SynSp$ with the
    property that $\bZ\otimes_{\bS_\ev}\bT_\ev^\vee\we\bT_\fil^\vee$ as bicommutative bialgebras
    over $\bZ$.
\end{corollary}

\begin{proof}
    This follows from Corollaries~\ref{cor:bialgebra} and~\ref{cor:dualizable}.
\end{proof}

\begin{notation}
    Let $\rho(n)\colon\bS[S^1]\to \bS[S^1]$ denote the map of bicommutative
    bialgebras induced by the multiplication-by-$n$ map $S^1\to S^1$ of
    $\bE_\infty$-spaces.
    By functoriality, there is then an induced map $\rho(n)\colon\bT_{\ev}\to \bT_{\ev}$
    of bicommutative bialgebras in $\SynSp$. There are similarly bicommutative
    bialgebra maps $\rho(n)\colon\bT_\ev^\vee\to\bT_\ev^\vee$.
\end{notation}

\begin{lemma}\label{lem:dualizing}
    If $n\geq 1$, there is an equivalence \[\rho(n)_*\bT_{\ev}\simeq
    (\rho(n)_*\bT_{\ev}^\vee)[1](1)\] of $\bT_{\ev}$-modules. 
\end{lemma}

\begin{proof}
    It is enough to show this statement for $n=1$,
    in other words that $\bT_{\ev}\simeq \bT_{\ev}^{\vee}(1)[1]$ as
    $\bT_\ev$-modules, where the $\bT_\ev$-module structure on $\bT_\ev^\vee$
    is induced via the symmetric monoidality of the forgetful functor
    $\SynSp_{\bT_\ev}\rightarrow\SynSp$.
    To construct a map $\bT_\ev\rightarrow\bT_\ev^\vee(1)[1]$, we can evaluate
    $$\pi_0\left(\F^{\geq
    0}\left(\bT_\ev^\vee(1)[1]\right)\right)\we\pi_0\left(\bT^\vee[1]\right)\we\pi_0\bT.$$
    We pick the element corresponding to $1\in\pi_0\bT$. By adjunction, there
    is an induced map $\bT_\ev\rightarrow\bT_\ev^\vee(1)[1]$ of
    $\bT_\ev$-module spectra, which one sees is an equivalence by arguing
    even-locally.
\end{proof}

\begin{lemma}
    If $B$ is an even $\bE_{\infty}$-ring spectrum with trivial $\bT$-action, then the
    natural map
    \[B_\ev\otimes_{\bS_{\ev}}\rho(n)_*\bT_{\ev}\to
    B[S^1/C_n]_\ev\] is an equivalence of $B[S^1]_\ev$-modules
    for each $n\geq 1$.
\end{lemma}

\begin{proof}
    The forgetful functor $\SynSp_{\bT_\ev}\rightarrow\SynSp$ is conservative,
    so the claim follows from the symmetric monidality of the even filtration
    in this situation, Proposition~\ref{Prop: symmetric monoidal}.
\end{proof}

\begin{lemma}\label{lem:no_nontrivial}
    If $B$ is an even $\bE_{\infty}$-ring spectrum, then the fiber of the
    augmentation map $B[S^1]_\ev\rightarrow B_\ev$ is equivalent to $B_\ev(1)[1]$ as a
    $B[S^1]_\ev$-module with trivial action.
\end{lemma}

\begin{proof}
    We prove that the fiber
    of $B[S^1]_\ev\rightarrow B_\ev$ is equivalent to $B_\ev(1)[1]$ as a
    synthetic $B[S^1]_\ev$-module. For this, note that because
    $B[S^1]\rightarrow B$ is eff, we can compute $B_\ev$ as the even filtration
    on the $B[S^1]$-module corresponding to $B$. That is, we have that
    $B_\ev\we B_{\ev/B[S^1]}$, where the latter is defined as the limit
    $\lim_{B[S^1]\rightarrow C\text{ even}}\tau_{\geq 2\star} B\otimes_{B[S^1]}
    C$. By~\cite[Cor.~2.2.14(1)]{hrw}, this limit can be computed as $\Tot\tau_{\geq
    2\star}B\otimes_{B[S^1]}B^\bullet$, where $B^\bullet$ is the \v{C}ech
    complex of $B[S^1]\rightarrow B$. But, this \v{C}ech complex is equivalent
    to the \v{C}ech complex of the base changed morphism $B\rightarrow
    B\otimes_{B[S^1]}B$; the latter is eff as eff maps are closed under base
    change. Thus, the totalization computes $B_\ev$ by eff-local descent for the even
    filtration (see~\cite[Prop.~2.2.12(1)]{hrw}). It follows that
    $B[S^1]_\ev\rightarrow B_\ev$ is obtained as the functor
    $\F^{\geq\star}_{\ev/B[S^1]}$ applied to the augmentation map
    $B[S^1]\rightarrow B$ of $B$-modules. By functoriality, there is thus a map
    $\F^{\geq\star}_{\ev/B[S^1]}(\fib(B[S^1]\rightarrow
    B))\rightarrow\fib(B[S^1]_\ev\rightarrow B_\ev)$, and it is straightforward
    to see that this map is an equivalence of synthetic spectra. As the
    forgetful functor from $B[S^1]_\ev$-modules to synthetic spectra is
    conservative, this map is an equivalence. Finally, we note that
    $\fib(B[S^1]\rightarrow B)$ is equivalent to $B[1]$ as a $B[S^1]$-module.
    To see this, note that the space of $B[S^1]$-module structures on $B$ (or
    equivalently $B[1]$) is equivalent to the space of $B$-algebra maps
    $B[S^1]\rightarrow B$, or equivalent the space of pointed maps
    $\B S^1\rightarrow\B\Omega^\infty B$.\footnote{We apologize for the use of
    $B$ as the classifying space and as the even $\bE_\infty$-ring.} As the cells of $\B S^1$ are in even
    degrees and the homotopy groups of the infinite loopspace $\B\Omega^\infty B$ are in odd degrees,
    we see that this space is path connected.
    This completes the proof as we have $\F^{\geq\star}_{\ev/B[S^1]}(B[1])\we B(1)[1]$.
\end{proof}

\begin{corollary}\label{cor:nseries}
    Let $B$ be an even $\bE_\infty$-ring spectrum.
    For each $n\geq 1$, there is a pullback square
    $$\xymatrix{
        B_\ev(1)[1]\ar[r]\ar[d]&B_\ev(1)[1]\ar[d]\\
        B[S^1]_\ev\ar[r]&\rho(n)_*B[S^1]_\ev
    }$$ in $\SynSp_{B[S^1]_\ev}$.
\end{corollary}

\begin{proof}
    We have a commutative square
    $$\xymatrix{
        B[S^1]_\ev\ar[r]\ar[d]&\rho(n)_* B[S^1]_\ev\ar[d]\\
        B_\ev\ar@{=}[r]&B_\ev
    }$$ in $\SynSp_{B[S^1]_\ev}$. Taking vertical fibers yields the square
    claimed in the corollary in light of Lemma~\ref{lem:no_nontrivial}, which
    guarantees that each fiber is equivalent to $B_\ev(1)[1]$.
\end{proof}

\begin{remark}
    The horizontal map $B_\ev(1)[1]\rightarrow B_\ev(1)[1]$ is some deformation
    of the multiplication-by-$n$ map. We will see in
    Section~\ref{sec:comparison} that
    $$\Map_{B[S^1]_\ev}(B_\ev(1)[1],B_\ev(1)[1])\we\Map_{B[S^1]_\ev}(B_\ev,B_\ev)\we\F^0((B_\ev)^{\bT_\ev})\we\tau_{\geq
    0}(B^{\h S^1}).$$ If $B$ is even periodic, then in terms of a formal group law $\bG$ structure on $B^{\h S^1}$
    induced by the choice of a complex orientation,
    this map is given by the $n$-series $[n]_\bG(t)$.
\end{remark}

The following construction will be helpful in some arguments in the rest of the
paper.

\begin{construction}[The CW filtration]\label{const:cw}
    Let $B$ be an even $\bE_\infty$-ring spectrum. The fiber sequence
    $B_\ev(1)[1]\rightarrow B[S^1]_\ev\rightarrow B_\ev$ of
    Lemma~\ref{lem:no_nontrivial} can be iterated to give a
    ``periodic'' resolution of $B_\ev$ in $\SynSp_{B[S^1]_\ev}$.
    This gives an exhaustive decreasing filtration
    $\F^{\geq\boxtimes}_{\mathrm{CW}}B_\ev$ on
    $B_\ev\in\SynSp_{B[S^1]_\ev}$ with $\F^{\geq s}_{\mathrm{CW}}B_\ev\we 0$
    for $s>0$ and $\gr^s_{\mathrm{CW}}B_\ev\we B[S^1]_\ev(s)[2s]$.
\end{construction}

We come to our notion of filtered circle-equivariant synthetic spectra.

\begin{definition}[Synthetic spectra with synthetic circle action]
    A $\bT_\ev$-module in $\SynSp$ is called a synthetic spectrum with synthetic circle action. The
    stable $\infty$-category of such synthetic spectra is denoted by
    $\SynSp_{\bT_\ev}=\Mod_{\bT_\ev}(\SynSp)\we\FD(\bT_\ev)$. It is a presentably symmetric monoidal
    stable $\infty$-category where the symmetric monoidal structure arises from
    the $\bE_\infty$-coalgebra structure on $\bT_\ev$. 
\end{definition}

We note the following comparison to~\cite{raksit}.

\begin{proposition}
    There is a natural equivalence $\bZ\otimes_{\bS_\ev}\bT_\ev\we\bT_\fil$. In
    particular, base change to $\bZ$ induces a colimit-preserving symmetric
    monoidal functor $\SynSp_{\bT_\ev}\rightarrow\FD(\bT_\fil)$.
\end{proposition}

\begin{proof}
    It follows from Proposition~\ref{Prop: symmetric monoidal} that
    $\bZ\otimes_{\bS_\ev}\bT_\ev\we(\bZ\otimes_{\bS}\bT)_\ev\we(\bZ[S^1])_\ev$.
    The latter is equivalent to $\bT_\fil$ by Lemma~\ref{lem:tfil}.
\end{proof}

It follows that $\SynSp_{\bT_\ev}$ satisfies the requirement of being a lift to
the synthetic sphere spectrum of the notion of a filtered $\bZ$-module spectrum
with $\bT_\fil$-action.

\subsection{Synthetic orbits, fixed points, and Tate}

Let $C$ be an $\bE_\infty$-ring and let $G$ be a grouplike $\bE_1$-space, with group ring $C[G]$.
Given a $C$-module spectrum $M$ with $G$-action, i.e., a module over $C[G]$, one defines $$M_{\h
G}=C\otimes_{C[G]}M\quad\text{and}\quad M^{\h
G}=\Map_{C[G]}(C,M).$$ These are the homotopy orbits
and homotopy fixed points respectively, and each admits the structure of a $C$-module spectrum.
The homotopy orbits functor is the left adjoint to the ``trivial action'' functor
$\D(C[G])\leftarrow\D(C)$ induced by restriction of scalars along the counit
map $C[G]\rightarrow C$ and the homotopy fixed points functor is the right
adjoint.

\begin{definition}[$\bT_\ev$-orbits and fixed points]
    Given $M\in\SynSp_{\bT_\ev}$, we let
    $$M_{\bT_\ev}=\bS_\ev\otimes_{\bT_\ev}M\quad\text{and}\quad M^{\bT_\ev}=\bMap_{\bT_\ev}(\bS_\ev,M),$$
    where we view $\bS_\ev$ as a synthetic spectrum with synthetic circle action via the ``trivial''
    action, i.e., via the augmentation map $\bT_\ev\rightarrow\bS_\ev$. Again
    these are the the left and right adjoints of the trivial action functor
    $\SynSp_{\bT_\ev}\leftarrow\SynSp$ obtained by restriction of scalars along
    $\bT_\ev\rightarrow\bS_\ev$.
\end{definition}

\begin{remark}
    More generally, if $A$ is eff-locally even, then we define $\bT_\ev$-orbits
    and fixed points for objects of $\SynSp_{A[S^1]_\ev}$. The underlying
    synthetic spectra can be first computed by applying the forgetful functor
    $\SynSp_{\bT_\ev}\leftarrow\SynSp_{A[S^1]_\ev}$.
\end{remark}

\begin{remark}
    If $B$ is an even $\bE_\infty$-ring spectrum and
    $\F^{\geq\star}M\in\SynSp_{B[S^1]_\ev}$, then the CW filtration of
    Construction~\ref{const:cw} can be used to construct complete filtrations
    on $(\F^{\geq\star}M)_{\bT_\ev}$ and $(\F^{\geq\star}M)^{\bT_\ev}$
    analogous to those used to define the homotopy orbits and fixed points
    spectral sequences. See the proof of Lemma~\ref{lem:complete_underlying} for more details.
\end{remark}

We follow~\cite[Sec.~2.4]{raksit} to construct the norm for the synthetic circle and hence the synthetic
analogue of the $S^1$-Tate construction. We repeat Raksit's construction here to prepare for a parallel
construction in the case of the synthetic analogue of $C_n$.

\begin{construction}[Synthetic $S^1$-Tate]\label{const:syn_tate}
    By Lemma~\ref{lem:dualizing}, we can choose a specific equivalence
    $\alpha\colon\bT_\ev^\vee\rightarrow\bT_\ev(-1)[-1]\we\bT_\ev\otimes_{\bS_\ev}\bS_\ev(-1)[-1]$
    of $\bT_\ev$-modules. There is also a $\bT_\ev$-module map
    $\eta\colon\bS_\ev\rightarrow\bT_\ev^\vee$ which corresponds to the identity on $\bS_\ev$ under
    the equivalence
    $$\Map_\SynSp(\bS_\ev,\bS_\ev)\we\Map_{\bT_\ev}(\bT_\ev,\bS_\ev)\we\Map_{\bT_\ev}(\bS_\ev,\bT_\ev^\vee).$$
    Given $M\in\SynSp_{\bT_\ev}$, we then have a map
    \begin{align*}
        \bS_\ev(1)[1]\otimes_{\bS_\ev}M_{\bT_\ev}&=\bS_\ev(1)[1]\otimes_{\bS_\ev}(M\otimes_{\bT_\ev}\bS_\ev)\\
        &\xrightarrow{\eta}\bS_\ev(1)[1]\otimes_{\bS_\ev}(M\otimes_{\bT_\ev}\bT_\ev^\vee)\\
        &\we^\alpha\bS_\ev(1)[1]\otimes_{\bS_\ev}(M\otimes_{\bT_\ev}\bT_\ev\otimes_{\bS_\ev}\bS_\ev(-1)[-1])\\
        &\we M
    \end{align*}
    of synthetic spectra.
    The left-hand side, which we denote by $M_{\bT_\ev}(1)[1]$, is a synthetic spectrum with trivial synthetic circle
    action, so the map above factors canonically through a map $M_{\bT_\ev}(1)[1]\rightarrow M^{\bT_\ev}$,
    called the synthetic $S^1$-norm and denoted by $\mathrm{Nm}_{\bT_\ev}$.
    We define $M^{\t\bT_\ev}$ to be the cofiber, in synthetic spectra, of
    $\mathrm{Nm}_{\bT_\ev}$.
    This is the synthetic $S^1$-Tate construction and is functorial in $M$.
\end{construction}

\begin{example}\label{ex:z_filtered_tate}
    We view $\bZ=\ins^0\bZ$ with its trivial $\bT_\ev$-action. Then,
    the fiber sequence
    $\bZ^{\bT_\ev}\rightarrow\bZ^{\t\bT_\ev}\rightarrow\bZ_{\bT_\ev}(1)[2]$
    is equivalent to $$\tau_{\geq
    2\star}(\bZ^{\h S^1})\rightarrow\tau_{\geq
    2\star}(\bZ^{\t S^1})\rightarrow\tau_{\geq 2\star}(\bZ_{\h S^1}[2]).$$
    See~\cite[Sec.~6.3]{raksit}.
\end{example}

Taking $\bT_\ev$-fixed points defines a functor $\SynSp_{\bT_\ev}\rightarrow\SynSp$ which is
naturally lax symmetric monoidal as it is the right adjoint of the symmetric monoidal functor
$\SynSp\rightarrow\SynSp_{\bT_\ev}$ obtained by restriction of scalars along
$\bT_\ev\rightarrow\bS_\ev$.

\begin{lemma}
    The functor $(-)^{\t\bT_\ev}\colon\SynSp_{\bT_\ev}\rightarrow\SynSp$ admits a natural lax symmetric monoidal structure
    and there is a natural lax symmetric monoidal structure on the natural transformation
    $(-)^{\bT_\ev}\rightarrow(-)^{\t\bT_\ev}$.
\end{lemma}

\begin{proof}
    This is a special case of~\cite[Prop.~2.4.10]{raksit}.
\end{proof}

Now, we turn our attention to the synthetic analogue of $C_n$-orbits and fixed points.
Given a spectrum $M$ with $S^1$-action, we can compute its homotopy $C_n$-orbits
and fixed points as $$M_{\h C_n}\we(\rho(n)_*\bZ[S^1])\otimes_{\bZ[S^1]}M\quad\text{and}\quad M^{\h
C_n}\we\Map_{\bZ[C_n]}(\rho(n)_*\bZ[C_n],M).$$ These observations motivate the following approach to
$C_n$-orbits, fixed points, and Tate in the synthetic setting.

\begin{definition}[Synthetic $C_n$-orbits and fixed points]
    Given $M\in\SynSp_{\bT_\ev}$ we let
    $$M_{C_{n,\ev}}=\rho(n)_*\bT_\ev\otimes_{\bT_\ev}M\quad\text{and}\quad
    M^{C_{n,\ev}}=\bMap_{\bT_\ev}(\rho(n)_*\bT_\ev,M),$$ where the latter is
    the internal mapping spectrum and where $C_{n,\ev}$ is our notation for the
    synthetic analogue of $C_n$, manifest here only via its orbit and fixed
    point constructions on filtered $\bT_\ev$-modules.
    These functors are the left and right adjoints, respectively, of
    $\SynSp_{\bT_\ev}\leftarrow\SynSp_{\bT_\ev}{\colon\rho(n)_*}$.
    As $\bMap_{\bT_\ev}(\rho(n)_*\bT_\ev,-)$ is the right adjoint to a
    symmetric monoidal functor, it acquires a lax symmetric monoidal
    structure.
\end{definition}

\begin{remark}[A $p$-complete synthetic group algebra for $C_n$]
    Without a filtration the group ring $\bZ[C_n]$
    is a dualizable bicommutative bialgebra and $\bZ[C_n]^\vee\we\bZ[C_n]$ as $\bZ[C_n]$-modules. Thus,
    the formalism developed in~\cite[Sec.~2.4]{raksit} applies to produce a $C_n$-Tate construction. 
    In the filtered case, it is not immediately clear which filtration to put on $\bZ[C_n]$. The
    Postnikov filtration would be trivial. One could use filtrations by powers of the maximal ideal.
    But, what arises naturally is the following construction using Eilenberg--Moore in the $p$-complete
    case. Of course, as a
    bicommutative bialgebra, $\bZ_p[C_p]$ is dual to $\bZ_p[C_p]^\vee$. But, as $C_p$ fits into a pullback
    square
    $$\xymatrix{
        C_p\ar[r]\ar[d]&S^1\ar[d]^{\times p}\\
        \ast\ar[r]&S^1,
    }$$ we can also compute $\bZ_p[C_p]^\vee$ by Eilenberg--Moore as
    $$\rho_*\bZ_p[S^1]^\vee\otimes_{\bZ_p[S^1]^\vee}\bZ,$$ where we write $\rho_*\bZ_p[S^1]$ for the restriction of
    scalars of $\bZ_p[S^1]^\vee$ along the multiplication-by-$p$ map
    $\bZ_p[S^1]^\vee\rightarrow\bZ_p[S^1]^\vee$.
    We can similarly $\F^{\geq\star}\bS_p^{C_{p,+}}$ as the $p$-completion of
    $\rho_*\bT_\ev^\vee\otimes_{\bT_\ev^\vee}\bS_\ev$.
    However, this bicommutative bialgebra is not dualizable. Indeed, this is visible already at the filtered level, where we invite the reader to
    contemplate the $p$-completion of $\rho_*\bT_{\fil}\otimes_{\bT_{\fil}}\bZ_p$.
    In particular, while comodule spectra over
    this tensor product correspond to some kind of synthetic spectra with $C_p$-action, the
    approach above to the Tate construction does not apply.
\end{remark}

\begin{construction}[Synthetic $C_n$-Tate]~\label{const: Tate Construction finite groups.}
    As discussed in Construction~\ref{const:syn_tate}, we have an equivalence
    $\alpha\colon\bT_\ev^\vee\rightarrow\bT_\ev\otimes\bS_\ev(-1)[-1]$ of synthetic $\bT_\ev$-module
    spectra. As $\rho(n)_*$ is symmetric monoidal, we also obtain an equivalence of $\bT_\ev$-modules
    $\rho(n)_*(\alpha)\colon\rho(n)_*(\bT_\ev^\vee)\rightarrow\rho(n)_*(\bT_\ev\otimes\bS_\ev(-1)[-1])\we\rho(n)_*(\bT_\ev)\otimes\bS_\ev(-1)[-1]$
    and hence an equivalence of $\bT_\ev$-modules
    $\beta=\rho(n)_*(\alpha)^{-1}\otimes\bS_\ev(1)[1]\colon\rho(n)_*(\bT_\ev)\rightarrow\rho(n)_*(\bT_\ev^\vee)\otimes\bS_\ev(1)[1]$.
    Moreover, applying $\bS_\ev$-linear duals to the map
    $\bT_\ev\rightarrow\rho(n)_*\bT_\ev$, we obtain a map $\eta\colon\rho(n)_*\bT_\ev^\vee\rightarrow\bT_\ev^\vee$
    of synthetic $\bT_\ev$-module spectra. Now, consider the natural map
    \begin{align*}
        M_{C_{n,\ev}}&=M\otimes_{\bT_\ev}\rho(n)_*\bT_\ev\\
        &\we^\beta M\otimes_{\bT_\ev}\rho(n)_*\bT_\ev^\vee\otimes\bS_\ev(1)[1]\\
        &\xrightarrow{\eta} M\otimes_{\bT_\ev}\bT_\ev^\vee\otimes\bS_\ev(1)[1]\\
        &\we^{\alpha} M\otimes_{\bT_\ev}\bT_\ev\otimes\bS_\ev(-1)[-1]\otimes\bS_\ev(1)[1]\\
        &\we M
    \end{align*}
    of $\bT_\ev$-modules. As the left-hand side $M_{C_{n,\ev}}$ is naturally a
    $\rho(n)_*\bT_\ev$-module, this map canonically induces a map $$\Nm_p\colon M_{C_{n,\ev}}\rightarrow
    M^{C_{n,\ev}}$$ of synthetic $\rho(n)_*\bT_\ev$-module spectra,
    which we will call the synthetic $C_{n,\ev}$-norm. Its cofiber will be
    written as $$M^{\t C_{n,\ev}},$$ and referred to as the $C_{n,\ev}$-Tate construction of $M$.
    It is naturally a synthetic $\rho(n)_*\bT_\ev\we\bT_\ev$-module spectrum.
\end{construction}

We will now show the lax symmetric monoidality of the synthetic
$C_{n,\ev}$-Tate construction following a variant of the arguments
in~\cite{nikolaus-scholze} and~\cite{raksit}.

\begin{definition}
    Let $\SynSp_{\bT_{\ev}}^{ind}\subseteq \SynSp_{\bT_{\ev}}$ be the thick
    subcategory generated by objects of the form $X\otimes \bT_{\ev}$ where
    $X\in \SynSp$.
\end{definition}

\begin{remark}\label{rem: tensor subcat}
    Following~\cite[Remark 2.4.7]{raksit}, $\SynSp_{\bT_\ev}^{ind}$ is a tensor ideal.
\end{remark}

\begin{lemma}~\label{lem: Tate vanishes on induced things}
    The functor $(-)^{\t C_{n,\ev}}$ vanishes on $\SynSp_{\bT_{\ev}}^{ind}$.
\end{lemma}

\begin{proof}
    By exactness of the $C_{n,\ev}$-Tate construction,
    it is enough to show that $X^{\t C_{n,\ev}}=0$ for $X\simeq X_0\otimes
    \bT_{\ev}$ for some $X_0\in \SynSp$. To this end, note that the map in
    Construction~\ref{const: Tate Construction finite groups.} is given in this
    case by the map \[X_0\otimes \rho(n)_*\bT_{\ev}\to X_0\otimes \bT_{\ev}\] given
    by tensoring the canonical map $\rho(n)_*\bT_{\ev}\simeq^{\beta^{-1}}
    \rho(n)_*\bT_{\ev}^\vee [1](1)\to\bT_{\ev}^\vee[1](1) \simeq^{\alpha^{-1}}
    \bT_{\ev}$ with $X_0$. The norm map in this case is then given by
    \begin{align*}
        X_0\otimes \rho(n)_*\bT_{\ev} &\simeq
        \bMap_{\rho(n)_*\bT_{\ev}}(\rho(n)_*\bT_{\ev}, X_0\otimes
        \rho(n)_{*}\bT_{\ev})\\
        &\to \bMap_{\bT_{\ev}}(\rho(n)_*\bT_{\ev}, X_0\otimes \rho_*\bT_{\ev})\\
        &\to \bMap_{\bT_{\ev}}(\rho(n)_*\bT_{\ev}, X_0\otimes \bT_{\ev}).
    \end{align*}
    Note that the target of this map is $\rho(n)_*\bT_{\ev}$-equivariantly
    equivalent to $X_0\otimes \rho(n)_*\bT_{\ev}$. In fact, the map is given by
    \begin{align*}
        \bMap_{\bT_{\ev}}(\rho(n)_*\bT_{\ev}, X_0\otimes \bT_{\ev}) &\simeq^{\alpha}
        \bMap_{\bT_{\ev}}(\rho(n)_*\bT_{\ev}, X_0\otimes \bT_{\ev}^\vee [1](1))\\
        &\simeq \bMap_{\bT_{\ev}}(\rho(n)_*\bT_{\ev}\otimes \bT_{\ev}, X_0)[1](1)\\
        &\simeq \bMap(\rho(n)_*\bT_{\ev},X_0)[1](1)\\
        &\simeq X_0\otimes \rho(n)_{*}\bT_{\ev}^\vee [1](1)\\
        &\simeq^{\beta} X_0\otimes \rho(n)_*\bT_{\ev},
    \end{align*} which gives an inverse of the norm map.
\end{proof}

\begin{proposition}\label{prop:monoidal}
    There is a natural transformation
    $$(-)^{C_{n,\ev}}\rightarrow(-)^{\t C_{n,\ev}}\colon\SynSp_{\bT_\ev}\rightarrow\SynSp_{\bT_\ev}$$
    of functors which makes $(-)^{\t C_{n,\ev}}$ the universal approximation to $(-)^{C_{n,\ev}}$
    which vanishes on $\SynSp_{\bT_{\ev}}^{ind}$.
    Moreover, $(-)^{\t C_{n,\ev}}$ admits a natural symmetric monoidal structure and the natural transformations
    $(-)^{C_{n,\ev}}\rightarrow(-)^{\t C_{n,\ev}}$
    and $(-)^{\t\bT_\ev}\rightarrow (-)^{\t C_{n,\ev}}$ do as well.
\end{proposition}

\begin{proof}
    Recall that by Remark~\ref{rem: tensor subcat} the thick subcategory
    $\SynSp_{\bT_{\ev}}^{ind}\subseteq \SynSp_{\bT_{\ev}}$ is a tensor ideal.
    Take $\mathcal{C}:= \SynSp_{\bT_{\ev}}$, $\mathcal{D}:=
    \SynSp_{\bT_{\ev}}^{ind}$, and $\mathcal{E}:=
    \SynSp_{\rho_*\bT_{\ev}}\we\Cscr$. Then,
    by~\cite[Thm.~I.3.6(ii)]{nikolaus-scholze}, there is an initial functor
    $H:\mathcal{C}/\mathcal{D}\to \mathcal{E}$ such that the composition
    $H\circ p\colon \mathcal{C}\to \mathcal{E}$ receives a natural transformation
    from $(-)^{C_{n,\ev}}$, and both $H$ and the natural transformation admit
    unique lax-monoidal enhancements. By Lemma~\ref{lem: Tate vanishes on
    induced things} there is then an induced natural transformation $H\to
    (-)^{\t C_{n,\ev}}$. This natural transformation is then an equivalence by
    \cite[Lem.~I.3.3(ii)]{nikolaus-scholze} and the corresponding equivalences
    \[\colim_{Y\in \mathcal{D}_{/X}}\cofib(Y\to X)^{C_{n,\ev}}\simeq
    \colim_{Y\in \mathcal{D}_{/X}}\cofib(Y\to X)^{\t C_{n,\ev}}\simeq
    X^{\t C_{n,\ev}},\] which follows from the first part of \cite[Lemma
    2.4.8]{raksit} and the same argument as the second part of {\em ibid.}
    Symmetric monoidality of the induced functor
    $(-)^{\t\bT_\ev}\rightarrow(-)^{\t C_{n,\ev}}$ follows
    from~\cite[Thm.~I.3.6(i)]{nikolaus-scholze}.
\end{proof}

\subsection{The Postnikov $t$-structure}

We show that $\bT_\ev$ is connective in the Postnikov $t$-structure on $\SynSp$
and compute $\pi_0^\P\bT_\ev$.

\begin{lemma}
    The synthetic spectrum $\bT_\ev$ is connective in the Postnikov
    $t$-structure on $\SynSp$: $\F^{\geq i}(\bT_\ev)$ is $i$-connective for all
    $i\in\bZ$.
\end{lemma}

\begin{proof}
    This is a feature the even filtration on any connective $\bE_\infty$-ring
    spectrum by the result of Burklund and Krause of Remark~\ref{rem:bounds} or
    by~\cite[Thm.~1.7]{pstragowski-even}. Alternatively, in this case, it
    follows from the fiber sequence
    $\bS_\ev(1)[1]\rightarrow\bT_\ev\rightarrow\bS_\ev$ of synthetic spectra
    established in Lemma~\ref{lem: fiber sequence for Tsyn}.
\end{proof}

\begin{lemma}\label{lem:tev_pi0}
    There is an isomorphism $\pi_0^\P\bT_\ev\iso\bZ[\eta,d]/(2\eta,d^2-\eta d)$
    of graded-commutative $\bZ[\eta]/(2\eta)$-algebras, where $|\eta|=|d|=1$.
\end{lemma}

\begin{proof}
    The unit map $\bS_\ev\rightarrow\bT_\ev$ makes $\pi_0^\P\bT_\ev$ naturally
    into a graded-commutative $\pi_0^\P\bS_\ev\iso\bZ[\eta]/(2\eta)$-algebra.
    We also have that $\bT_\ev$ fits into a split fiber sequence
    $$\bS_\ev(1)[1]\rightarrow\bT_\ev\rightarrow\bS_\ev$$ from which we obtain
    an exact sequence
    $$0\rightarrow\pi_0^\P(\bS_\ev(1)[1])\rightarrow\pi_0^\P(\bT_\ev)\rightarrow\pi_0^\P(\bS_\ev)\rightarrow
    0$$ in $\SynSp^{\P\heart}$. This is equivalently an exact sequence
    $$0\rightarrow\bZ[\eta]/(2\eta)(1)\rightarrow\pi_0^\P(\bT_\ev)\rightarrow\bZ[\eta]/(2\eta)\rightarrow
    0$$
    in graded $\bZ[\eta]/(2\eta)$-modules. Writing $d$ for the image of $1(1)$
    in $\pi_0^\P(\bT_\ev)$, we obtain a surjective map
    $\bZ[\eta,d]/(2\eta)\rightarrow\pi_0^\P(\bT_\ev)$. As $d^2=\eta d$ in
    $\bS[S^1]$ (see Remark~\ref{rem:no_spherical_Tfil}), we see that in fact we
    must already have $d^2=\eta d$ in
    $\pi_0^\P(\bT_\ev)$ as there is no room for a differential which could
    create this relation in the associated spectral sequence. Now, the induced
    map $\bZ[\eta,d]/(2\eta,d^2-\eta d)\rightarrow\pi_0^\P(\bT_\ev)$ is
    surjective and must then be injective by counting ranks.
\end{proof}

\begin{construction}
    The Postnikov $t$-structure on $\SynSp_{\bT_\ev}$ is the one obtained from
    $\SynSp$ using that $\bT_\ev$ is connective. In other words, an object of
    $\SynSp_{\bT_\ev}$ is
    (co)connective if and only if its underlying synthetic spectrum is
    (co)connective. This $t$-structure is accessible, compatible with filtered
    colimits, and compatible with the symmetric monoidal structure on
    $\SynSp_{\bT_\ev}$. By Lemma~\ref{lem:tev_pi0}, there is an equivalence
    \begin{equation}\label{eq:heart_equivalence}
        \pi_0^\P\colon\SynSp_{\bT_\ev}^{\P\heart}\we\GrMod_{\bZ[\eta,d]/(2\eta,d^2-\eta d)}
    \end{equation}
    of abelian categories.

    In general, if $\Cscr$ is a symmetric monoidal stable $\infty$-category whose tensor product is
    exact in each variable and if $\Cscr$ is equipped with a
    $t$-structure $(\Cscr_{\geq 0},\Cscr_{\leq 0})$ which is compatible with
    the symmetric monoidal structure, then $\pi_0\colon\Cscr_{\geq
    0}\rightarrow\Cscr^\heart$ is symmetric monoidal. In particular,
    $\pi_0^\P\bT_\ev$ is naturally a bicommutative bialgebra in
    $\bZ[\eta]/(2\eta)$-modules. We view
    $\GrMod_{\bZ[\eta,d]/(2\eta,d^2-\eta d)}$ as being symmetric monoidal via
    the coalgebra structure on $\pi_0^\P\bT_\ev$.
    Then, the equivalence~\eqref{eq:heart_equivalence} is symmetric monoidal.
\end{construction}

\begin{lemma}
    \begin{enumerate}\label{lem:t_exact}
        \item[{\em (a)}] Restriction of scalars along $\bT_\ev\rightarrow\bS_\ev$ is $t$-exact with
            respect to the Postnikov $t$-structures. As a consequence,
            $(-)_{\bT_\ev}\colon\SynSp_{\bT_\ev}\rightarrow\SynSp$ is right $t$-exact
            and $(-)^{\bT_\ev}\colon\SynSp_{\bT_\ev}\rightarrow\SynSp$ is left
            $t$-exact.
        \item[{\em (b)}] Restriction of scalars along
            $\rho(n)\colon\bT_\ev\rightarrow\bT_\ev$ is $t$-exact with respect
            to the Postnikov $t$-structures. Thus,
            $(-)_{C_{n,\ev}}\colon\SynSp_{\bT_\ev}\rightarrow\SynSp_{\rho(n)_*\bT_\ev}\we\SynSp_{\bT_\ev}$
            is right $t$-exact and
            $(-)^{C_{n,\ev}}\colon\SynSp_{\bT_\ev}\rightarrow\SynSp_{\rho(n)_*\bT_\ev}\we\SynSp_{\bT_\ev}$
            is left $t$-exact.
    \end{enumerate}
\end{lemma}

\begin{proof}
    Both claims follow from the fact that left adjoints to $t$-exact functors
    are right $t$-exact and right adjoints to $t$-exact functors are left
    $t$-exact.
\end{proof}

\begin{example}
    Suppose that $M_*\in\GrMod_{\bZ[\eta,d]/(2\eta,d^2-\eta d)}$ and that, for
    simplicity, $2$ acts invertibly on $M$. In this case, $M_*$ is
    a cochain complex: $d\colon M_i\rightarrow M_{i+1}$ and $d^2=0$. Let $\F^{\geq\star}M$ denote the
    synthetic spectrum associated to $M_*$ via the inclusion
    $\SynSp_{\bT_\ev}^{\P\heart}\hookrightarrow(\SynSp_{\bT_\ev})^\P_{\geq 0}$.
    Then, $\pi_0^\P((\F^{\geq\star} M)_{\bT_\ev})\iso
    M_*\otimes_{\bZ[1/2,d]/(d^2)}\bZ[1/2]$ is the quotient of $M_*$ by the
    boundaries. On the other hand,
    $\pi_0^\P((\F^{\geq\star}M)^{\bT_\ev})\iso\bMap_{\bZ[1/2,d]/(2d)}(\bZ,M_*)$
    is the sub-object of cycles. The norm map
    $(\F^{\geq\star}M)(1)[1]_{\bT_\ev}\rightarrow(\F^{\geq\star}M)^{\bT_\ev}$
    induces the differential on $\pi_0^{\P}$, which in weight $i$ takes the
    form $d\colon\tfrac{M_{i-1}}{d(M_{i-2})}\rightarrow Z_i(M_*)$, from which we find that
    $\pi_0^\P((\F^{\geq\star}M)^{\t\T_\ev})$ is the graded abelian cohomology of
    $M_*$, while $\pi_1^\P((\F^{\geq\star}M)^{\t T_\ev})$ is also the cohomology,
    up to a shift.
\end{example}

\begin{example}~\label{example: norm on pi_0 postnikov}
    Let $A_*:= \bZ[\eta]/2\eta$
    be the graded ring with $|\eta|=1$,
    and let $M_*$ be a graded $A_*$-module.
    Suppose $d\colon M_*\to M_{*+1}$ is a
    graded $A_*$-module map with $d^2=\eta d$.
    This data gives rise to a $\pi_0^\P\bT_\ev$-module structure, and therefore
    a $\bT_{\ev}$-module structure, on the filtered spectrum $\F^{\geq
    \star}M:=M_\star[\star]$ (with all the maps $\F^{\geq i+1}M\to\F^{\geq
    i}M$ being zero). To describe what the filtered norm map is in this
    case, note that $\F^{\geq \star}M$  is connective in the Postnikov
    t-structre and therefore $(\F^{\geq\star}M)_{C_{p,\ev}}$ is connective
    and $(\F^{\geq\star}M_\bullet)^{C_{p,\ev}}$ is coconnective. Thus the map
    $(\F^{\geq\star}M)_{C_{p,\ev}}\to (\F^{\geq\star}M)^{C^{p,\ev}}$ is
    completely determined by what it does on $\pi_0^\P(-)$.

    We have that 
    \begin{align*}
        \pi_0^\P((\F^{\geq\star}M)_{C_{p,\ev}}) &\simeq M_* \otimes_{\bZ[\eta, d]/(2\eta, d^2-\eta d)}\bZ\left[\eta,\frac{d}{p}\right]/(2\eta, (d/p)^2-\eta(d/p))\\
    &\simeq M_*[d/p]
    \end{align*}
    as a graded abelian group. Similarly one finds that \[\pi_0^\P\left(\F^{\geq
    \star}M^{C_{p,\ev}}\right)\simeq \F^{\geq\star}N=N_\star[\star],\] where
    $N_*$ is the graded object with $N_i=\{(x,y)\in
    M_i\times M_{i+1}:d(x)=py\}$. 
    The map $\pi_0^\P(\F^{\geq
    \star}M)_{C_{p,\ev}}\to \pi_0^\P(\F^{\geq\star}M)^{C_{p,\ev}}$ is
    given by sending $m\in M_i$ to $(pm,d(m))\in N_i\subseteq M_i\times M_{i+1}$, which can be
    checked by reducing to the case $M_* =
    \pi_0^\P(\bT_{\ev})$ and noting that the norm map is induced by the map
    $\rho(n)_*\bT_{\ev}\to \bT_{\ev}$ which as $\bS_{\ev}$-modules is the map
    $\bS_{\ev}\oplus \bS_{\ev}[1](1)\to \bS_{\ev}\oplus \bS_{\ev}[1](1)$ which
    is $\id\oplus p(1)[1]$.
\end{example}

\begin{remark}
    Suppose that $\F^{\geq\star}M$ is a rational $\bT_\ev$-module (or
    equivalently
    $\bT_\fil$-module). Then, the $C_{n,\ev}$-norm
    $(\F^{\geq\star}M)_{C_{n,\ev}}\rightarrow(\F^{\geq\star}M)^{C_{n,\ev}}$ is an
    equivalence, so $(\F^{\geq\star}M)^{\t C_{n,\ev}}$ is trivial as a synthetic
    spectrum. Indeed, it is a synthetic $\bQ^{\t C_{n,\ev}}$-module spectrum.
    But, $\bQ^{\t C_{n,\ev}}\we\tau_{\geq 2\star}(\bQ^{\t C_n})\we 0$ as in
    Example~\ref{ex:z_filtered_tate}.
\end{remark}

We conclude the section by analzying how our fixed points and Tate
constructions behave with respect to complete objects and underlying objects.

\begin{lemma}\label{lem:complete_underlying}
    Let $\F^{\geq\star}M$ be a $\bT_\ev$-module in $\SynSp$ with underlying
    object $M$. As the colimit functor is symmetric monoidal, $M$ is naturally
    equipped with a $\bT$-action in $\Sp$. Fix an integer $n\geq 1$.
    \begin{enumerate}
        \item[{\em (i)}] The natural maps
            $|(\F^{\geq\star}M)_{\bT_\ev}|\rightarrow M_{\h S^1}$
            and $|(\F^{\geq\star}M)_{C_{n,\ev}}|\rightarrow M_{\h C_n}$
            are equivalences.
        \item[{\em (ii)}] If $\F^{\geq\star}M$ is connective in the Postnikov
            $t$-structure, then so are $(\F^{\geq\star}M)_{\bT_\ev}$ and
            $(\F^{\geq\star}M)_{C_{n,\ev}}$. In particular, they are complete.
        \item[{\em (iii)}] If $\F^{\geq\star}M$ is complete, then so are
            $(\F^{\geq\star}M)^{\bT_\ev}$ and $(\F^{\geq\star}M)^{C_{n,\ev}}$.
        \item[{\em (iv)}] If $\F^{\geq i}M\rightarrow M$ is
            $i$-truncated for all $i\in\bZ$, then the natural maps $\F^{\geq
            i}((\F^{\geq\star}M)^{\bT_\ev})\rightarrow M^{\h S^1}$ 
            and $\F^{\geq i}((\F^{\geq\star}M)^{C_{n,\ev}})\rightarrow M^{\h C_n}$
            are $i$-truncated for all $i\in\bZ$. In particular, the underlying
            spectra of $(\F^{\geq\star}M)^{\bT_\ev}$ and
            $(\F^{\geq\star}M)^{C_{n,\ev}}$ are $M^{\h S^1}$ and $M^{\h C_n}$,
            respectively.
        \item[{\em (v)}] Suppose now that $\F^{\geq\star}M$ is a
            $B[S^1]_\ev$-module where $B$ is an even $\bE_\infty$-ring
            spectrum. Suppose that $\F^{\geq i}M \rightarrow M$ is
            $2i$-truncated for all $i\in\bZ$. Then, the natural maps $\F^{\geq
            i}((\F^{\geq\star}M)^{\bT_\ev})\rightarrow M^{\h S^1}$ are
            $2i$-truncated for all $i\in\bZ$ and the natural maps $\F^{\geq
            i}((\F^{\geq\star}M)^{C_{n,\ev}})\rightarrow M^{\h C_n}$ are
            $(2i+1)$-truncated for all $i\in\bZ$. In particular, the underlying
            spectra of
            $(\F^{\geq\star}M)^{\bT_\ev}$ and $(\F^{\geq\star}M)^{C_{n,\ev}}$
            are naturally equivalent to $M^{\h S^1}$ and $M^{\h C_n}$, respectively.
        \item[{\em (vi)}] In the situation of (v), if $\F^{\geq
            i}M\we\tau_{\geq 2i}M$ for all $i\in\bZ$, then
            $(\F^{\geq\star}M)^{\bT_\ev}$
            is the double-speed Whitehead tower for $M^{\h S^1}$. If additionally
            $M$ is even, then $(\F^{\geq\star}M)^{C_{n,\ev}}$ is the double-speed
            Whitehead tower of $M^{\h C_n}$.
    \end{enumerate}
\end{lemma}

\begin{proof}
    Part (i) follows because the colimit functor
    $\F^{\geq-\infty}\colon\SynSp\rightarrow\Sp$ is symmetric monoidal.
    Part (ii) follows because (relative) tensor products of
    Postnikov-connective objects are connective since the Postnikov
    $t$-structure is compatible with the symmetric monoidal structure on
    $\SynSp$; moreover, $\bT_\ev$ and $\bS_\ev$ are connective.
    Part (iii) follows from the following general fact: that if $\F^{\geq\star}M$ and $\F^{\geq\star}N$ are
    filtered spectra and if $\F^{\geq\star}N$ is complete, then the internal
    mapping object $\bMap(\F^{\geq\star}M,\F^{\geq\star}N)$ is complete. This
    follows because it corepresents
    $$\Map_{\FD(\bS)}(-\otimes\F^{\geq\star}M,\F^{\geq\star}N)\we\Map_{\widehat{\FD}(\bS)}(-\widehat{\otimes}\F^{\geq\star}M,\F^{\geq\star}N).$$
    
    By adjunction, there is a canonical map $\F^{\geq\star}M\rightarrow c(M)$
    of $\bT_\ev$-modules,
    where $c(M)$ denotes the constant filtration on $M$. The condition that
    $\F^{\geq i}M\rightarrow M$ is $i$-truncated for $i\in\bZ$ is equivalent to the condition
    that the fiber of $\F^{\geq i}M\rightarrow c(M)$ is coconnective in the
    Postnikov $t$-structure. As $(-)^{\bT_\ev}$ and $(-)^{C_{n,\ev}}$ are left
    $t$-exact in this $t$-structure by Lemma~\ref{lem:t_exact} and as these
    functors preserve constant objects, part (iv) holds.

    For part (v), we first consider the $\bT_\ev$-fixed points.
    We will use the CW filtration of Construction~\ref{const:cw}.
    Since taking $\F^{\geq i}\colon\SynSp\rightarrow\Sp$ commutes with all
    limits and colimits, we can compute $$\F^{\geq
    i}((\F^{\geq\star}M)^{\bT_\ev})\we\lim_{s\to-\infty}\F^{\geq
    i}(\bMap_{B[S^1]_\ev}(\F^{\geq s}_{\mathrm{CW}}B_\ev,\F^{\geq\star}M)).$$
    As $\F^{\geq s}_{\mathrm{CW}}B_\ev$ is made up of a cell $B[S^1]_\ev(t)[2t]$
    for each $0\leq t\leq s$, we see that for fixed $s$ we have that $$\F^{\geq
    i}\bMap_{B[S^1]_\ev}(\F^{\geq s}_{\mathrm{CW}}B_\ev,\F^{\geq\star}M)$$ is
    made an iterated extension of $\F^{\geq
    i}(\F^{\geq\star}M(-t)[-2t])\we\F^{\geq i+t}M[-2t]$ for
    $0\leq t\leq s$. This maps to the corresponding cellular filtration
    $\bMap_{B[S^1]}(\F^{\geq s}_{\mathrm{CW}}B,M)$ which has associated graded
    pieces $M[-2t]$ for $0\leq t\leq s$. As $\F^{\geq i+t}M[-2t]\rightarrow
    M[-2t]$ is $2i$-truncated and as the $2i$-equivalences are closed under
    limits and filtered colimits, part (v) follows for $\bT_\ev$-fixed points.
    Now, we use the commutative square
    $$\xymatrix{
        B_\ev(1)[1]\ar[r]\ar[d]&B_\ev(1)[1]\ar[d]\\
        B[S^1]_\ev\ar[r]&\rho(n)_*B[S^1]_\ev
    }$$ in $\SynSp_{B[S^1]_\ev}$ of Corollary~\ref{cor:nseries}. It follows by mapping to
    $\F^{\geq\star}M$ that there is a
    pullback square
    $$\xymatrix{(\F^{\geq\star}M)^{C_{n,\ev}}\ar[r]\ar[d]&\F^{\geq\star}M\ar[d]\\
    (\F^{\geq\star}M)^{\bT_\ev}(-1)[-1]\ar[r]&(\F^{\geq\star}M)^{\bT_\ev}(-1)[-1]
    }$$
    of synthetic spectra. As taking underlying objects preserves finite limits
    and as $(2i+1)$-truncated maps are closed under limits,
    we see that part (v) for $\bT_\ev$-fixed points implies it for $C_{n,\ev}$-fixed points.

    Under the hypothesis of (vi), we have that $\F^{\geq i}M\rightarrow M$ is
    $(2i-2)$-truncated for every $i\in\bZ$. Thus, $\F^{\geq
    i}((\F^{\geq\star}M)^{\bT_\ev})\rightarrow M^{\h S^1}$ is $(2i-2)$-truncated
    for every $i\in\bZ$. However, the argument of the previous paragraph also
    presents $\F^{\geq
    i}((\F^{\geq\star}M)^{\bT_\ev})\rightarrow M^{\h S^1}$ as a filtered colimit
    of terms which are built out of finitely many extensions of the form
    $\F^{i+t}M[-2t]$ for $t\geq 0$. But, these terms are $2i$-connective, so
    the filtered colimit is too. This proves part (vi) in the case of
    $\bT_\ev$-fixed points. For $C_{n,\ev}$-fixed points, we use the pullback
    square and the result for the $\bT_\ev$-fixed points to conclude that
    $\F^{\geq i}((\F^{\geq\star}M)^{C_{n,\ev}})$ is again $2i$-connective,
    which is enough in light of the fact that the map from this spectrum to
    $M^{\h C_n}$ is $(2i-2)$-truncated (since $\F^{\geq i}M\rightarrow M$ is now
    $(2i-3)$-truncated for each $i\in\bZ$).
\end{proof}

\subsection{A synthetic Tate orbit lemma}

One extremely helpful lemma in the theory of cyclotomic spectra is the Tate
orbit lemma of~\cite{nikolaus-scholze}. We record here a version of this lemma in our setting.

\begin{definition}~\label{def:t_star}
    Fix a $t$-structure $\C=(\SynSp_{\geq 0}^\C,\SynSp_{\leq 0}^\C)$ on $\SynSp$. Consider the following
    conditions:
    \begin{enumerate}
        \item[(a)] $\C$ is compatible with the symmetric monoidal structure;
        \item[(b)] $\C$ is left complete
        \item[(c)] $\C$ is accessible;
        \item[(d)] $\C$ is compatible with filtered colimits;
        \item[(e)] $\C$ is compatible with countable products;
        \item[(f)] $\bT_\ev\in\SynSp_{\geq 0}^\C$.
    \end{enumerate}
    If $\C$ satisfies conditions (a)-(f), we say that $\C$ satisfies condition $\mathbf{(\star)}$.
\end{definition}

\begin{example}
    Both the Postnikov $t$-structure and the neutral $t$-structure satisfy condition $\mathbf{(\star)}$.
\end{example}

\begin{remark}
    If $\C$ is a $t$-structure on $\SynSp$ satisfying $\mathbf{(\star)}$, then there is an induced
    $t$-structure on $\SynSp_{\bT_\ev}$, which we will also call $\C$.
\end{remark}

\begin{lemma}~\label{lem: Tate orbit lemma}
    Let $\F^{\geq\star}X\in \SynSp_{\bT_{\ev}}$ be such that the underlying
    synthetic spectrum is bounded below with respect to a $t$-structure $\C$ satisfiying
    $\mathbf{(\star)}$. Then \[(\F^{\geq\star}X_{C_{p,\ev}})^{\t C_{p,\ev}}\simeq 0.\]
\end{lemma}

As in the proof of the Tate orbit lemma for spectra in \cite[Lemma I.2.1]{nikolaus-scholze} we will
need to commute the Tate construction past certain limits. We record the necessary (co)continuity
results below.

\begin{lemma}~\label{lem: orbit/fixedpoints/tate continuity}
    Let $1\leq n\leq \infty$ and let $\C$ be a $t$-structure on $\SynSp$ satisfying condition $\mathbf{(\star)}$.
    \begin{enumerate}
        \item[{\em (1)}] Let
            $\{\F^{\geq\star}X_k\}_k\in \SynSp_{\bT_{\ev}}^{\bN^{op}}$ be a tower of synthetic spectra with
            $\bT_{\ev}$-action such that there is some $K\gg 0$ and some $N\in \bZ$ such that
            $\fib(\F^{\geq\star}X_{i+1}\to\F^{\geq\star}X_{i})\in \SynSp_{\geq N}^\C$ whenever $i\geq K$.
            The natural maps 
        \begin{enumerate}
            \item[{\em (a)}] $(\lim_k\F^{\geq\star}X_k)^{C_{n,\ev}}\xrightarrow{\simeq} \lim_k\F^{\geq\star}X_k^{C_{n,\ev}}$,
            \item[{\em (b)}] $(\lim_k
                \F^{\geq\star}X_k)_{C_{n,\ev}}\xrightarrow{\simeq}\lim_k
                (\F^{\geq\star}X_k)_{C_{n,\ev}}$, and
            \item[{\em (c)}] $(\lim_k\F^{\geq\star}X_k)^{\t C_{n,\ev}}\xrightarrow{\simeq} \lim_k(\F^{\geq\star}X_k)^{\t C_{n,\ev}}$
        \end{enumerate}
        are all equivalences.
    \item[{\em (2)}] Dually, let $\{\F^{\geq\star}X_k\}_k\in
        \SynSp_{\bT_{\ev}}^{\bN}$
        be a tower of synthetic spectra with $\bT_{\ev}$-action such that there
            is some $K\gg 0$ and some $N\in \bZ$
            such that $\fib(\F^{\geq\star}X_{i}\to
            \F^{\geq\star}X_{i+1})\in \SynSp_{\leq N}^\C$ whenever $i\geq K$. The natural maps 
        \begin{enumerate}
            \item[{\em (a)}] $\colim_k
                (\F^{\geq\star}X_k)^{C_{n,\ev}}\xrightarrow{\simeq} (\colim_k
                \F^{\geq\star}X_k)^{C_{n,\ev}}$,
            \item[{\em (b)}] $\colim_k
                (\F^{\geq\star}X_k)_{C_{n,\ev}}\xrightarrow{\simeq}(\colim_k
                \F^{\geq\star}X_k)_{C_{n,\ev}}$, and
            \item[{\em (c)}] $\colim_k (\F^{\geq\star}X_k)^{\t C_{n,\ev}}\xrightarrow{\simeq}
                (\colim_k\F^{\geq\star}X_k)^{\t C_{n,\ev}}$
        \end{enumerate}
        are equivalences.
    \end{enumerate}
\end{lemma}

\begin{proof}
    We will only prove part (1) since the proof for (2) is dual. For (1), note that it is
    enough to prove that the map (b) is an equivalence since (a) follows from
    the fact that $(-)^{C_{n,\ev}}$ is left $t$-exact with respect to $\C$ since we assume that
    $\bT_\ev$ is $\C$-connective. By assumption, we may
    assume that the terms $\F^{\geq\star}X_k$ are all $(N+1)$-connective with
    respect to $\C$, since we have a cofiber sequence \[(\lim_k \tau_{\geq
    N+1}^C\F^{\geq\star}X_k)_{C_{n,\ev}}\to (\lim_k\F^{\geq\star}X_k)_{C_{n,\ev}}
    \to (\lim_k \tau_{\leq N}^C\F^{\geq\star}X_k)_{C_{n,\ev}}\] and the limit on the
    right is eventually constant by assumption and so will commute with
    $(-)_{C_{n,\ev}}$.

    By left-completness of $\C$,
    it is now enough to show that the fiber of the map
    $(\lim_k\F^{\geq\star}X_k)_{C_{n,\ev}}\to \lim_k(\F^{\geq\star}X_k)_{C_{n,\ev}}$ is $i$-connective for all $i$.
    This follows by using bar constructions to compute the tensor product: 
    \begin{align*}
        (\lim_k\F^{\geq\star}X_k)_{C_{n,\ev}} &:= (\lim_k\F^{\geq\star}X_k)\otimes_{\bT_{\ev}}\rho_* \bT_{\ev}\\
        &\simeq |(\lim_k\F^{\geq\star}X_k)\otimes \bT_{\ev}^{\otimes \bullet}\otimes \rho_*\bT_{\ev}|\\
        &\simeq |\lim_k(\F^{\geq\star}X_k\otimes \bT_{\ev}^{\otimes
        \bullet}\otimes \rho_*\bT_{\ev})|,
    \end{align*}
    where the second equivalence comes from the fact that $\bT_{\ev}$ is a
    compact $\bS_{\ev}$-module. In particular each term in the above geometric
    realization is $(N+1)$-connective with respect to $\C$ by
    compatibility of $C$ with the monoidal structure, and so $\tau_{\leq
    i}^\C((\lim_k\F^{\geq\star}X_k)_{C_{n,\ev}})$ depends only on the
    $(i-N-1)$-skeleton of the above geometric realization.
    More precisely, the maps $|\mathrm{sk}_i(\bullet \mapsto\F^{\geq\star}X_k\otimes
    \bT_\ev^{\otimes \bullet}\otimes \rho(n)_*\bT_\ev)|\to |\bullet \mapsto\F^{\geq
    \star}X_k\otimes \bT_\ev^{\otimes \bullet}\otimes \rho_*\bT_\ev|$ have cofibers, say $Z_{k,i}$,
    which are at least $(i+N)$-connective, and $|\mathrm{sk}_i(\bullet \mapsto\F^{\geq
    \star}X_k\otimes \bT_\ev^{\otimes \bullet}\otimes \rho_*\bT_\ev)|$ is a finite colimit in
    $\SynSp$. By compatiblity of $C$ with countable products we have that $\lim_k Z_{k,i}$ is at
    least $(i+N-1)$-connective.

    Now, $\lim_k\F^{\geq \star}X_k$ is at least $(N-1)$-connective with respect to $\C$, since by
    compatibility with countable products we have that a countable inverse limit can decrease
    connectivity by at most $1$. Consequently, by the same argument as above, the cofiber $Z_i$ of
    the map $|\mathrm{sk}_i(\bullet \mapsto (\lim_k\F^{\geq \star}X_k)\otimes \bT_\ev^{\otimes
    \bullet}\otimes \rho(n)_*\bT_\ev)|\to |\bullet \mapsto (\lim_k\F^{\geq \star}X_k)\otimes
    \bT_\ev^{\otimes \bullet}\otimes \rho(n)_*\bT_\ev|$ is at least $(i+N-1)$-connective. For each
    $i$, there is
    a map of cofiber sequences in synthetic spectra
    \[
    \begin{tikzcd}
     \mid \mathrm{sk}_i(\bullet \mapsto (\lim_k\F^{\geq \star}X_k)\otimes \bT_\ev^{\otimes \bullet}\otimes \rho(n)_*\bT_\ev)| \ar[d,"\simeq"] \ar[r] & (\lim_k\F^{\geq \star}X_k)_{C_{n,\ev}} \ar[d] \ar[r] & Z_i\ar[d]\\
    \lim_k|\mathrm{sk}_i(\bullet \mapsto\F^{\geq \star}X_k\otimes \bT_\ev^{\otimes \bullet}\otimes
        \rho(n)_*\bT_\ev)| \ar[r] & \lim_k(\F^{\geq \star}X_k)_{C_{n,\ev}} \ar[r] & \lim_k Z_{k,i},
    \end{tikzcd}
    \]
    and the left hand vertical map is an equivalence since the functors $(-)\otimes
    \bT_{\ev}^{\otimes \bullet}\otimes \rho_*\bT_\ev$ preserve limits as $\bT_{\ev}$ is a
    perfect $\bS_\ev$-module, and the geometric
    realization of a finite skeleton is a finite colimit which will commute with limits of
    synthetic spectra. Thus, the right hand square above is a pullback square and we have a cofiber
    sequence $(\lim_k\F^{\geq \star}X_k)_{C_{n,\ev}}\to \lim_k(\F^{\geq \star}X_k)_{C_{n,\ev}}\to
    (\lim_{k}Z_{k,i}/Z_i)$ for each $i$. The term $(\lim_{k}Z_{k,i}/Z_i)$ is at least $(i+N-1)$-connective as
    the cofiber of $(i+N-1)$-connective synthetic spectra, and since there is such a cofiber
    sequence for each $i$, result follows.
\end{proof}

\begin{notation}
    Given a graded spectrum $M(*)$, we let $\zeta_*(M(*))$ be the filtered
    spectrum $\cdots\rightarrow M(i+1)\rightarrow M(i)\rightarrow
    M(i-1)\rightarrow\cdots$ where the transition maps are all zero.
\end{notation}

\begin{lemma}
    Let $\C$ be a $t$-structure on $\SynSp$ satisfying condition $\mathbf{(\star)}$.
    For a connected space $X$, the map $\bS_{\ev}\{X\}\to \bS_\ev$ from the free
    $\bS_\ev$-$\bE_\infty$ ring on $\bS_\ev\otimes X$ to $\bS_\ev$ is $1$-connective with respect
    to $\C$.
\end{lemma}
\begin{proof}
    We have that the free $\bS_{\ev}$-$\bE_\infty$ ring is given by the formula \[\bigoplus_{n\geq
    0}\bS_{\ev}\otimes (X^{\wedge n})_{\h\Sigma_n}\] and so it is enough to show that for a
    connected space $X$ the synthetic spectrum $\bS_{\ev}\otimes X$ is $1$-connective with respect
    to $\C$. Using compatibility with colimits we may assume that $X$ is compact. Using closure
    under extensions and induction we can then reduce to the case of $X=S^n$, where $n\geq 1$ by
    the assumption that $X$ is connected. Then $\bS_\ev\otimes S^n\simeq \bS_\ev[n]$ by definition
    which is $n$-connective with respect to $\C$.
\end{proof}

\begin{corollary}\label{cor:ins}
    Let $\F^{\geq\star}R$ be an $\bE_\infty$ ring in $\SynSp$ with $\F^{\geq 0}R$ connective. Let
    $\C$ be a $t$-structure
    on $\SynSp_{\F^{\geq\star}R}$ satisfying condition $\mathbf{(\star)}$.
    Then every element in $\SynSp_{\F^{\geq\star}R}^{\heartsuit C}$ is an $\ins^0 \pi_0(\F^{\geq 0}R)$-module.
\end{corollary}
\begin{proof}
    It is enough to produce a map of $\bE_{\infty}$ rings $\F^{\geq\star}R\to
    \F^{\geq\star}S$ such that $\F^{\geq 0}S\simeq \pi_0(\F^{\geq 0}R)$ and with
    $\pi_0^{C}(\F^{\geq\star}R)\to \pi_0^C(\F^{\geq\star}S)$ and
    equivalence. We will do this in stages. Define
    \[\F^{\geq\star}R_1:=\F^{\geq\star}R\otimes_{\bS_\ev\{\bigoplus_{\alpha\in
    \pi_1(\F^{\geq 0}R)}S^1\}}\bS_{\ev}\] where the map
    $\bS_\ev\{\bigoplus_{\alpha\in \pi_1(\F^{\geq 0}R)}S^1\}\to\F^{\geq\star}R$
    is induced by the maps $\bigoplus_{\alpha \in \pi_1(\F^{\geq 0}R)}S^1\to
    \F^{\geq 0}R$ sending the circle indexed by $\alpha$ to $\alpha$, and the
    map $\bS_\ev\{\bigoplus_{\alpha\in \pi_1(\F^{\geq 0}R)}S^1\}\to \bS_{\ev}$
    is the map induced by the zero map on each of the $S^1$ summands. Then
    there is a natural map of $\bE_\infty$ rings $\F^{\geq\star}R\to
    \F^{\geq\star}R_1$ with fiber
    \[\F^{\geq\star}R\otimes_{\bS_\ev\{\bigoplus_{\alpha\in \pi_1(\F^{\geq
    0}R)}S^1\}}\mathrm{fib}(\bS_\ev\{\bigoplus_{\alpha\in \pi_1(\F^{\geq
    0}R)}S^1\}\to \bS_{\ev})\] which is $1$-connective in the $C$
    t-structure by compatibility with the monoidal structure, filtered
    colimits, and the previous lemma. Consequently the map $\F^{\geq\star}R\to
    \F^{\geq\star}R_1$ is an isomorphism after applying $\pi_0^C$.
    Additionally by using the geometric realization formula for the relative
    tensor product we have that $\pi_0(\F^{\geq 0}R_1)\cong \pi_0(\F^{\geq 0}R)$
    and $\pi_1(\F^{\geq 0}R_1)\cong 0$.

    Inductively, suppose that we have constructed a map $\F^{\geq\star}R\to\F^{\geq\star}R_n$ of $\bE_\infty$ rings such that 
    \begin{enumerate}
        \item[(1)] $\F^{\geq\star}R_n$ is connective with respect to $\C$;
        \item[(2)] the induced map $\pi_0^C(\F^{\geq\star}R)\to \pi_0^C(\F^{\geq\star}R_n)$ is an isomorphism;
        \item[(3)] we have that \[\pi_i^\C(\F^{\geq 0}R_n)\cong \begin{cases}
            \pi_0^\C(\F^{\geq 0}R) & \textrm{ if }i=0,\\
            0 &\textrm{ if }0<i\leq n.
        \end{cases}\]
    \end{enumerate}
    Then we can construct an $\bE_\infty$ ring map $\F^{\geq\star}R_n\to\F^{\geq\star}R_{n+1}$ such that 
    \begin{enumerate}
        \item[(1)] $\F^{\geq\star}R_{n+1}$ is connective with respect to $\C$;
        \item[(2)] the induced map $\pi_0^\C(\F^{\geq\star}R_n)\to \pi_0^\C(\F^{\geq\star}R_{n+1})$ is an isomorphism;
        \item[(3)] we have that \[\pi_i^\C(\F^{\geq 0}R_{n+1})\cong \begin{cases}
            \pi_0^\C(\F^{\geq 0}R) & \textrm{ if }i=0,\\
            0 &\textrm{ if }0<i\leq n+1
        \end{cases}\]
    \end{enumerate}
    by repeating the argument from the first paragraph with $\F^{\geq\star}R$ replaced with
    $\F^{\geq\star}R_n$, $\pi_1^\C$ replaced by $\pi_{n+1}^\C$, and $S^1$ replaced with $S^n$.
    Define $\F^{\geq\star}S:=\colim\F^{\geq\star}R_n$. Then, all the desired properties of
    $\F^{\geq\star}S$ follow from those of $\F^{\geq\star}R_n$ together with the compatibility of
    $\C$ with filtered colimits.  
\end{proof}

We are now ready to prove the Tate orbit lemma.

\begin{proof}[Proof of Lemma~\ref{lem: Tate orbit lemma}]
    Applying Corollary~\ref{cor:ins}, we see that $\pi_0^\C(\bS_{\ev})$ is an
    $\ins^0\bZ$-modules in $\SynSp$. Thus $(\F^{\geq
    *}X_{C_{p,\ev}})^{tC_{p,\ev}}$ will be a module over
    $(\ins^0\bZ^{C_{p,\ev}})^{tC_{p,\ev}}$. By Lemma~\ref{lem: even filtration
    comparison in the even case} we then have that
    \begin{align*}
        (\ins^0\bZ^{C_{p,\ev}})^{\t C_{p,\ev}} &\simeq (\tau_{\geq 2\star}\bZ^{\h C_p})^{\t C_{p,\ev}}\\
        &\simeq \tau_{\geq 2\star}\left[(\bZ^{\h C_p})^{\t C_p}\right]
    \end{align*} and this vanishes by \cite[Lemma I.2.7]{nikolaus-scholze}. Since $(\F^{\geq\star}X_{C_{p,\ev}})^{\t C_{p,\ev}}$ is a module over the zero ring it also vanishes.
\end{proof}

As an application of Lemma~\ref{lem: orbit/fixedpoints/tate continuity},
we include the following lemma which gives some information about compatibility between the
double-speed Postnikov $t$-structure and constructions like $\bT_\ev$-fixed points, even though
$\bT_\ev$ is not connective in the double-speed Postnikov $t$-structure.

\begin{lemma}~\label{lem: fixed points are somehow also left t-exact I guess}
    Let $\F^{\geq\star}X\in \SynSp_{\bT_{\ev}}$ be $n$-connective in the double-speed Postnikov
    $t$-structure and suppose that it is bounded below in the Postnikov $t$-structure. Then each of
    $\F^{\geq\star}X^{\bT_{\ev}}$, $\F^{\geq\star}X^{C_{p,\ev}}$, $\F^{\geq\star}X^{t\bT_{\ev}}$,
    and $\F^{\geq\star}X^{\t C_{p,\ev}}$ are $n$-connective in the double-speed Postnikov
    $t$-structure.
\end{lemma}

\begin{proof}
    Suspending or desuspending as needed, we may assume that $n=0$. We will
    also assume that $\F^{\geq\star}X$ is $N$-connective for the
    Postnikov $t$-structure. We will first show this result for $(-)^{\bT_{\ev}}$
    by reducing to showing that each $\pi_i^\P(\F^{\geq \star}X)[i]^{\bT_\ev}$ is
    connective in the double-speed Postnikov $t$-structure, where $\pi_i^\P$ denotes the $i$th
    homotopy object in
    the (single-speed) Postnikov $t$-structure. Since the Postnikov $t$-structure
    satisfies all the conditions of Lemma~\ref{lem: orbit/fixedpoints/tate continuity},
    $(\F^{\geq \star}X)^{\bT_{\ev}}\simeq \lim_k(\tau_{\leq k}^\P\F^{\geq \star}X)^{\bT_\ev}$, so
    it is enough to show the result for $\tau_{\leq k}^\P\F^{\geq \star}X$ for all $k\geq 0$ and that
    the limit of the sequence $(\tau_{\leq k}^\P\F^{\geq \star}X)^{\bT_\ev}$ stays connective.

    Assume for the moment that we have shown that each $(\pi_i^\P\F^{\geq \star}X[i])^{\bT_\ev}$ is
    connective in the double-speed Postnikov $t$-structure. Then, from the fiber sequence
    \[\pi_{k+1}^\P\F^{\geq \star}X[k+1]\to \tau_{\leq k+1}^\P\F^{\geq \star}X\to \tau_{\leq k}^\P\F^{\geq
    \star}X,\] we get that $(\tau_{\leq k+1}^\P\F^{\geq \star}X)^{\bT_\ev}$ is an extension of
    $(\pi_{k+1}^\P\F^{\geq \star}X[k+1])^{\bT_\ev}$ and $(\tau_{\leq k}^\P\F^{\geq \star}X)^{\bT_\ev}$,
    so inductively, since we assumed $\F^{\geq\star}X$ is bounded-below in the Postnikov
    $t$-structure, $(\tau_{\leq k+1}^\P\F^{\geq \star}X)^{\bT_\ev}$ will be connective in the
    double-speed Postnikov $t$-structure. Furthermore, we have an exact sequence
    \[\pi_0^{2\P}\left(\pi_{k+1}^\P\F^{\geq \star}X[k+1]^{\bT_\ev}\right)\to
    \pi_0^{2\P}\left((\tau_{\leq k+1}^\P\F^{\geq \star}X)^{\bT_\ev}\right)\to
    \pi_0^{2\P}\left((\tau_{\leq k}^\P\F^{\geq \star}X)^{\bT_\ev}\right)\to 0\] and so since the
    double-speed Postnikov $t$-structure is compatible with countable products it follows that
    there is no $\lim^1$-term and the inverse limit $\lim_k(\tau_{\leq k}\F^{\geq
    \star}X)^{\bT_\ev}$ will again be connective.\footnote{The Mittag--Leffler condition is enough
    to imply that $\lim^1$ vanishes in this case because the heart of the double-speed Postnikov
    $t$-structure is equivalent to graded abelian groups.} Thus, we have reduced to showing that each
    $(\pi_i^\P\F^{\geq \star}X[i])^{\bT_\ev}$ is connective in the double-speed Postnikov
    $t$-structure.
    
    Note that if $\F^{\geq\star}X$ is connective in the double-speed Postnikov $t$-structure, then
    so is $\pi_i^{\P}\F^{\geq\star}X[i]$ as each $\F^{\geq
    j}(\pi_i^\P\F^{\geq\star}X[i])\we\pi_{i+j}\F^{\geq j}X[i+j]$ already satisfies the condition that,
    if it is nonzero, then $i\geq j$. Moreover, $\F^{\geq j}(\pi_i^\P\F^{\geq\star}X[i])\we 0$ for
    $j>i$. It follows that, under the equivalence between $\SynSp_{\bT_\ev}^{\P\heart}$ and graded
    modules for $\bZ[\eta,d]/(2\eta,d^2-\eta d)$, we can view the object $\pi_i^\P\F^{\geq\star}X$
    as an $\eta$-complex $$\cdots\rightarrow M^{i-2}\rightarrow M^{i-1}\rightarrow M^i\rightarrow
    0\rightarrow\cdots.$$ We can filter this $\eta$-complex in the naive way by letting
    $\G_{\leq n}M^\bullet$ be the subcomplex $$0\rightarrow M^{i-n}\rightarrow\cdots\rightarrow
    M^i\rightarrow 0.$$ This is an exhaustive $\bN$-indexed filtration on $M^\bullet$. We claim
    that $\colim_n(\G_{\leq n}M^\bullet)^{\bT_\ev}\we(M^\bullet)^{\bT_\ev}$;
    indeed, this is a special case of Lemma~\ref{lem: orbit/fixedpoints/tate continuity}(2.a)
    for the {\em neutral} $t$-structure
    applied to the Whitehead tower $\tau_{\geq\star}^{\N}(\pi_i^\P\F^{\geq\star}X[i])$, which agrees
    with the filtration $\G_{\leq\star} M^\bullet$ up to a shift.
    As connective objects are closed under colimits and extensions, we are reduced to showing that the
    $\bT_\ev$-fixed points of
    the synthetic spectrum with $\bT_\ev$ action associated to the $i$-fold suspension of $\gr_n^\G M^\bullet\we M^{i-n}$
    is connective in the double-speed Postnikov $t$-structure for each $n\geq 0$. This object is precisely the
    filtered spectrum $$\cdots\rightarrow 0\rightarrow\pi_{2i-n}\F^{\geq i-n}X[2i-n]\rightarrow
    0\rightarrow\cdots,$$ which is a filtered spectrum with only one, possibly, nonzero term, in
    weight $(i-n)$. Write $\F^{\geq\star}Y$ for this spectrum. As it is naturally a
    $\bT_\fil$-module spectrum, the increasing exhaustive CW filtration on the
    homotopy fixed points shows that $(\F^{\geq\star}Y)^{\bT_\ev}$ is a colimit of iterated
    extensions of terms of the form
    $(\F^{\geq\star}Y)(-k)[-2k]$ for $k\geq 0$. Each of these is connective in the double-speed
    Postnikov $t$-structure, hence so is the colimit.

    To prove that $(\F^{\geq\star}X)^{\t\bT_\ev}$ is connective in the double-speed Postnikov
    $t$-structure, it is necessary and sufficient to show that $(\F^{\geq\star}X)_{\bT_\ev}$,
    using the fiber sequence
    $(\F^{\geq\star}X)^{\bT_\ev}\rightarrow(\F^{\geq\star}X)^{\t\bT_\ev}\rightarrow(\F^{\geq\star}X)_{\bT_\ev}(1)[2]$
    and the fact that the functor $(-)(1)[2]$ is $t$-exact with respect to the double-speed
    Postnikov $t$-structure. We have by Lemma~\ref{lem: orbit/fixedpoints/tate
    continuity}(1.b) that the natural map $(\F^{\geq\star}X)_{\bT_\ev}\rightarrow\lim_i(\tau_{\leq
    i}^{\P}\F^{\geq\star}X)_{\bT_\ev}$ is an equivalence. Moreover, by connectivity of $\bT_\ev$ in
    the Postnikov $t$-structure, the fiber of $(\F^{\geq\star}X)_{\bT_\ev}\rightarrow(\tau_{\leq
    i}^\P\F^{\geq\star}X)_{\bT_\ev}$ is $(i+1)$-connective in the Postnikov $t$-structure for each
    $i$. Thus, the fiber of $\F^{\geq j}(\F^{\geq\star}X)_{\bT_\ev}\rightarrow\F^{\geq j}(\tau_{\leq
    2j}^\P\F^{\geq\star}X)_{\bT_\ev}$ is $(2j+1)$-connective. It follows by induction that if each
    $(\pi_i^\P\F^{\geq\star}X[i])_{\bT_\ev}$ is shown to be connective in the double-speed
    Postnikov $t$-structure, then so is $(\F^{\geq\star}X)_{\bT_\ev}$. Using the notation $\G_{\leq
    \star}M^\bullet$ from the previous paragraph, we have $(\colim_n\G_{\leq
    n}M^\bullet)_{\bT_\ev}\we(M^\bullet)_{\bT_\ev}$. We are reduced to checking that the graded
    pieces $\F^{\geq\star}Y$ of the $i$-fold suspension of the $\G$-filtration have double-speed connective $\bT_\ev$-orbits.
    These are given by an $\bN$-indexed limit of iterated extensions of $(\F^{\geq\star}Y)(k)[2k]$ for $k\geq 0$.
    These are again double-speed connective and the limit stabilizes in each filtration weight at a finite
    stage, so the limit is double-speed connective as well.

    Similar arguments apply to the case of $C_{p,\ev}$-fixed points and Tate by reducing as above to the
    case of $\bZ$-modules and using the square of Corollary~\ref{cor:nseries} to reduce to the
    $\bT_\ev$-case established above. 
\end{proof}

\section{Synthetic cyclotomic spectra}\label{sec:syncycsp}

In this section, we define the $\infty$-category of cyclotomic synthetic spectra and construct
examples of such objects using the BMS and even filtrations on $\THH$.
We work only with the synthetic analogue of $p$-typical cyclotomic spectra as introduced
in~\cite{nikolaus-scholze}. The interested reader
can work out the relevant definition for integral cyclotomic synthetic spectra.

\subsection{\texorpdfstring{Construction of $\mathrm{CycSyn}$}{Construction of CycSyn}}

Here is the main definition.

\begin{definition}[Cyclotomic synthetic spectra]
    We define $\CycSyn$ as the $\infty$-category 
    \[\CycSyn:= \mathrm{LEq}\left(\id, (-)^{\t C_{p,\ev}}\colon \SynSp_{\bT_{\ev}}\rightrightarrows \SynSp_{\bT_{\ev}}\right).\]
\end{definition}

\begin{lemma}~\label{cor: cycsin is a nice category}
    The $\infty$-category $\CycSyn$ is a stable presentable $\infty$-category and
    the forgetful functor $\CycSyn\to \SynSp_{\bT_{\ev}}$ is
    conservative and preserves colimits. Moreover, $\CycSyn$ admits a natural symmetric monoidal structure
    compatible with the forgetful functor.
\end{lemma}

\begin{proof}
    See~\cite[Const.~IV.2.1(ii)]{nikolaus-scholze} for the claim about symmetric monoidal structures
    and~\cite[Prop.~II.1.5]{nikolaus-scholze} for the rest. We use that $(-)^{\t C_{p,\ev}}$ is lax
    symmetric monoidal and $\kappa$-accessible for any uncountable cardinal $\kappa$. The former is
    our Proposition~\ref{prop:monoidal}; the latter follows from the fact that $\rho_*\bT_\ev$ is $\kappa$-compact for any
    uncountable $\kappa$.
\end{proof}

\begin{example}[The monoidal unit]
    There is a canonical $\bT_\ev$-equivariant map $\bS_\ev\rightarrow(\bS_\ev)^{\t C_{p,\ev}}$
    arising from symmetric monoidality of the $C_{p,\ev}$-Tate construction. This data is
    equivalent to the unit of $\CycSyn$, which we write as $\bS_\ev^\triv$. 
\end{example}

\begin{example}\label{ex:inserted_cyclotomic}
    For $R$ an Eilenberg--Mac Lane spectrum we can lift $\ins^0R$ further to a synthetic cyclotomic spectrum in the following way.
    The $\bT_{\ev}$-module structure of $\ins^0R$ is induced by the ring maps $\bT_{\ev}\to
    \bS_{\ev}\to\ins^0\bZ$ and the natural $\ins^0\bZ$-module structure on $\ins^0R$.
    Then the identity map $\ins^0R\to\ins^0R$ factors functorially through a map $\ins^0R\to (\ins^0R)^{C_{p,\ev}}$.
    Post-composing with the norm natural transformation then gives a functorial synthetic
    cyclotomic structure on $\ins^0R$, which we will denote by $\ins^0R^{\triv}$. By the
    functoriality and the fact that $\tau_{\leq 0}(\ins^0R\otimes_{\bS_{\ev}}\ins^0S)=\ins^0(R\otimes_{\bZ} S)$, where $\tau_{\leq 0}$
    denotes the level-wise truncation functor, we see that the functor
    $\ins^0(-)^{\triv}\colon\mathrm{Ab}\to \mathrm{CycSyn}$ is lax symmetric monoidal.
\end{example}

\begin{definition}[Synthetic $\TC$]
    We let
    $\TC\colon\CycSyn\to \SynSp$ denote the functor corepresented by $\bS_\ev^\triv$, so that on
    objects $\F^{\geq\star} X\in\CycSyn$ it is given by \[\mathrm{TC}(\F^{\geq\star} X)=
    \bMap_{\CycSyn}(\bS_{\ev}^\triv,\F^{\geq\star}X),\] and similarly for complete cyclotomic
    spectra. Using the formula of~\cite[Prop.~II.1.5(ii)]{nikolaus-scholze}, we see that
    $$\TC(\F^{\geq\star}X)\we\Eq\left((\F^{\geq\star}X)^{\bT_\ev}\rightrightarrows((\F^{\geq\star}X)^{\t C_{p,\ev}})^{\bT_\ev}\right),$$
    which is a version of the formula given in~\cite{nikolaus-scholze}.
    Here, the first map is given by sending $f$ in
    $\bMap_{\bT_\ev}(\bS_\ev,\F^{\geq\star}X)\we(\F^{\geq\star}X)^{\bT_\ev}$ to the composition
    $$\bS_\ev\rightarrow(\bS_\ev)^{\t C_{p,\ev}}\xrightarrow{f^{\t C_{p,\ev}}}(\F^{\geq\star}X)^{\t
    C_{p,\ev}}.$$ The second map takes $f$ to the composition
    $$\bS_\ev\xrightarrow{f}\F^{\geq\star}X\xrightarrow{\varphi_X}(\F^{\geq\star}X)^{\t C_{p,\ev}},$$
    where $\varphi_X$ denotes the synthetic cyclotomic structure map for $\F^{\geq\star}X$.
\end{definition}

\begin{remark}
    We do not address here the question of whether the synthetic analogue of the Segal conjecture
    holds, i.e., whether $\bS_\ev\rightarrow(\bS_\ev)^{\t C_{p,\ev}}$ is a $p$-adic equivalence.
    However, this has been communicated to the authors by Robert Burklund to be true.
\end{remark}

It will also be helpful for us later to have a notion of cyclotomic synthetic spectra with
Frobenius lifts. In the non-synthetic case, these were introduced in~\cite[Sec.~IV.3]{nikolaus-scholze}.

\begin{definition}~\label{defn: cycsyn with fr lift}
    Let $\CycSyn^{\Fr}$ be the $\infty$-category defined by the lax equalizer
    \[
    \CycSyn^{\Fr}:= \LEq\left(\id, (-)^{C_{p,\ev}}\colon\SynSp_{\bT_{\ev}}\rightrightarrows
    \SynSp_{\bT_{\ev}} \right).
    \]
\end{definition}

\begin{corollary}
    The $\infty$-category $\CycSyn^\Fr$ is a stable presentable $\infty$-category and
    the forgetful functor $\CycSyn^\Fr\to \SynSp_{\bT_{\ev}}$ is
    conservative and preserves all limits and colimits. Moreover, $\CycSyn^\Fr$ admits a natural symmetric monoidal structure
    compatible with the forgetful functor.
\end{corollary}

\begin{proof}
    The proof is the same as in Corollary~\ref{cor: cycsin is a nice category}, with the only
    difference being that $(-)^{C_{p,\ev}}\colon\SynSp_{\bT_{\ev}}\to \SynSp_{\bT_{\ev}}$ also preserves
    all limits and so the forgetful functor $\CycSyn^{\Fr}\to \SynSp_{\bT_{\ev}}$ does as well
    by~\cite[Prop.~II.1.5(v)]{nikolaus-scholze}.
\end{proof}

\subsection{Comparison to the equivariant even
filtration}\label{sec:comparison}

Recall that an $\bE_\infty$-ring with $\bT$-action, or an $S^1$-equivariant $\bE_\infty$-ring, is a
$\bT^\vee$-comodule in $\CAlg(\Sp)$. The $\infty$-category $\coMod_{\bT^\vee}(\CAlg(\Sp))$ is
equivalent to the $\infty$-category of $\bE_\infty$-algebras in $\Sp^{\B S^1}$;
see~\cite[Prop.~2.2.10]{raksit}.

As $\bT_\ev$ is a bicommutative bialgebra, we have observed that
$\SynSp_{\bT_\ev}$ admits a symmetric monoidal structure such that the
forgetful functor $\SynSp_{\bT_\ev}\rightarrow\SynSp$ is naturally symmetric
monoidal. Before bootstrapping into the category of synthetic cyclotomic spectra, we will first compare
the filtrations and objects appearing in Section~\ref{sec:even_circle} to those constructed by
Hahn, Raksit, and Wilson in \cite{hrw}. This will both serve as a warm-up to the cyclotomic
case as well as a technical resource.

\begin{definition}
    A synthetic $\bE_\infty$-ring with $\bT_\ev$-action is an object of
    $$\coMod_{\bT_\ev^\vee}(\CAlg(\SynSp))\we\CAlg(\coMod_{\bT_\ev^\vee}(\SynSp)).$$
    We will call these synthetic $\bE_\infty$-rings with $\bT_\ev$-action or with
    $\bT_\ev^\vee$-coaction.
\end{definition}

\begin{remark}[Trivial coactions]\label{rem:trivial_coaction}
    If $\F^{\geq\star}R\in\CAlg(\SynSp)$, then there is a canonical trivial $\bT_\ev^\vee$-coaction
    $(\F^{\geq\star} R)^\triv$ on
    $\F^{\geq\star}R$. This is precisely the corestricted comodule structure on $\F^{\geq\star}R$
    obtained from corestriction along the coalgebra map $\bS_\ev\rightarrow\bT_\ev^\vee$. It follows
    that there is a functor $(-)^\triv\colon\CAlg(\SynSp)\rightarrow\coMod_{\bT_\ev^\vee}(\CAlg(\SynSp))$.
\end{remark}

Suppose that $A$ is an $\bE_\infty$-ring with a circle action.
In~\cite{hrw}, the authors give an approach to constructing filtrations on
$A^{\h S^1}$ and $A^{\t S^1}$, by using $S^1$-equivariant even covers. For example,
they define $$\F^{\geq\star}_{\ev,\bT} A=\lim_{A\rightarrow B\text{
    even, $S^1$-equivariant}}\tau_{\geq 2\star}(B^{\h S^1}),$$
where the limit is over $S^1$-equivariant maps to even $\bE_\infty$-rings with
$S^1$-action, and similarly
$$\F^{\geq\star}_{\ev,\t\bT} B=\lim_{A\rightarrow B\text{ even,
$S^1$-equivariant}}\tau_{\geq
2\star}(B^{\t S^1}).$$
To justify these definitions, recall that if $M$ is an even spectrum with
$S^1$-action, then $M^{\h S^1}$ and $M^{\t S^1}$ are both even. While $C_n$
analogues are not given in~\cite{hrw}, we can take a similar approach and
define
\begin{align*}
    \F^{\geq\star}_{\ev,C_n}A&=\lim_{A\rightarrow B\text{ even,
    $S^1$-equivariant}}\tau_{\geq 2\star} B^{\h C_n}\text{ and}\\
    \F^{\geq\star}_{\ev,\t C_n}A&=\lim_{A\rightarrow B\text{ even,
    $S^1$-equivariant}}\tau_{\geq 2\star} B^{\t C_n}.
\end{align*}
These constructions agree with those using $\bT_\ev$ and $C_{n,\ev}$ in many
cases of interest.


%


Let $\CAlg(\Sp^{\B S^1})^\lev\subseteq\CAlg(\Sp^{\B S^1})$ be the full subcategory of
$\bE_\infty$-rings $A$ with $S^1$-action admitting an eff $S^1$-equivariant map $A\rightarrow B$ to an even
$\bE_\infty$-ring with $S^1$-action. Let $\CAlg(\Sp^{\B S^1})^\ev\subseteq\CAlg(\Sp^{\B S^1})^\lev$
be the full subcategory of even $\bE_\infty$-rings with $S^1$-action. Note that for
$A\in\CAlg(\Sp^{\B S^1})^\lev$, the underlying $\bE_\infty$-ring of the circle-equivariant even
filtration $$\lim_{\text{$A\rightarrow B$, $S^1$-equivariant}}\tau_{\geq 2\star}B$$ agrees
with $\F^{\geq\star}_\ev A$.

The following unipotence lemma is inspired by the results of~\cite{MNN17}.

\begin{lemma}[Synthetic unipotence]\label{lem:unipotence}
    Let $R\in\CAlg(\Sp)$ be an even $\bE_\infty$-ring and consider  $R_{\ev}^\triv$ the
    associated synthetic $\bE_\infty$-ring with trivial $\bT_\ev$-action. The functor
    \[\Mod_{R_\ev^\triv}(\SynSp_{\bT_\ev})\xrightarrow{(-)^{\bT_{\ev}}}\SynSp_{R_{\ev}^{\bT_\ev}}\] is
    fully faithful with essential image the objects complete with respect to
    $R_\ev^{\bT_\ev}\rightarrow R_\ev$.
\end{lemma}

\begin{proof}
    The proof follows the argument of~\cite[Thm.~7.35]{MNN17}. To be more precise,
    note that $\bMap(\bT_\ev, R_\ev)$ is a compact, dualizable $R_\ev^\triv$-algebra in
    $\Mod_{R_\ev^\triv}(\SynSp_{\bT_\ev})$, the dual is also compact and generates $\SynSp_{\bT_\ev,
    R_\ev}$ as a localizing $\SynSp$-linear subcategory of $\SynSp_{\bT_\ev}$, and belongs to the
    thick subcategory generated by $\SynSp^\omega \otimes \{R_\ev^\triv\}\subseteq
    \Mod_{R_\ev^\triv}(\SynSp_{\bT_\ev})$.
    These are the conditions needed to make the synthetic analogue of \cite[Prop.~7.13]{MNN17} work.
\end{proof}

\begin{construction}
    Let $\CAlg(\SynSp_{\bT_\ev})^\ev\subseteq \CAlg(\SynSp_{\bT_\ev})$ denote the full subcategory
    spanned by objects of the form $\F^{\geq \star}E\simeq\tau_{\geq 2\star}E$ as underlying synthetic
    spectra, where $E$ is some even $\bE_\infty$-ring. Note that the functor $\F^{\geq
    -\infty}\colon\CAlg(\SynSp_{\bT_\ev})\to \CAlg(\Sp^{\B S^1})$ restricts to a functor
    $\CAlg(\SynSp_{\bT_\ev})^\ev\to \CAlg(\Sp^{\B S^1})^{\ev}$.
\end{construction}

\begin{notation}
    Let $E$ be an even $\bE_\infty$-ring spectrum with $S^1$-action.
    Then the homotopy fixed point spectral sequence computing $\pi_*(E^{\h S^1})$ will collapse, so
    $\pi_{-2}(E^{\h S^1})$ will surject onto $\H^2(\B S^1, \pi_0(E))$.
    We will denote by $v$ a fixed choice of lift of the generator of $\H^2(\B S^1, \pi_0(E))$ to $\pi_{-2}(E^{\h S^1})$.
    Similarly, we will denote by $v$ a choice of lift of this element to $\pi_{-2}(\F^{-1}(\tau_{\geq 2\star}E_{\h S^1}))$,
    and later to the same element in $\pi_{-2}(\F^{\geq -1}(E_\ev^{\bT_\ev}))$.
\end{notation}

\begin{lemma}~\label{lem: underling is equivalence on even}
    The functor $\F^{\geq -\infty}\colon\CAlg(\SynSp_{\bT_\ev})^{\ev}\to \CAlg(\Sp^{\B S^1})^\ev$ is an equivalence of $\infty$-categories.
\end{lemma}

\begin{proof}
    It is enough to show that $\F^{\geq -\infty}$ is fully faithful and essentially surjective. We
    will begin by showing essential surjectivity, so let $E\in \CAlg(\Sp^{\B S^1})^\ev$. Note first
    that for such an $E$ there is a canonical $\tau_{\geq
    2\star}((E^{\h S^1})^{\h S^1})$-algebra structure on $\tau_{\geq 2\star}(E^{\h S^1})$ coming
    from the lax-monoidality of the double-speed Postnikov filtration. Let $R:=E^{\h S^1}$. It then
    follows from Lemma~\ref{lem:complete_underlying}(vi) that $((R_\ev)^{\triv})^{\bT_\ev}\simeq
    \tau_{\geq 2\star}(R^{\h S^1})$, and this identification is of $\bE_\infty$-rings by
    Lemma~\ref{lem:uniqueness of ring structure on whitehead tower.}. Therefore we have a
    natural lift $\tau_{\geq 2\star}E^{\h S^1}\in
    \CAlg_{R_\ev^{\bT_\ev}}(\SynSp_{\bT_\ev})$
    This algebra is complete with respect to
    the map $(R_\ev^\triv)^{\bT_\ev}\to R_\ev^\triv$, and so Lemma~\ref{lem:unipotence} provides a
    natural object $$R^{\triv}_\ev\otimes_{(R^{\triv})^{\bT_\ev}}(\tau_{\geq 2\star}E^{\h S^1})\in \CAlg(\SynSp_{\bT_\ev})$$ 
    with underlying
    synthetic spectrum
    \[\tau_{\geq 2\star}(E^{\h S^1})\otimes_{\tau_{\geq 2\star}((E^{\h S^1})^{\h S^1})}\tau_{\geq 2\star}E^{\h S^1}
    \simeq (\tau_{\geq 2\star}E^{hS^1})/v
    \simeq \tau_{\geq 2\star}E,
    \]
    so this produces a lift of $\tau_{\geq 2\star}E$, as desired.

    It remains to show that $\F^{\geq -\infty}$ is fully faithful. Let $\tau_{\geq 2\star}E_1$ and
    $\tau_{\geq 2\star} E_2$ be two objects of $\CAlg(\SynSp_{\bT_\ev})^{\ev}$. Then any ring map
    $\tau_{\geq 2\star} E_1\to \tau_{\geq 2\star}E_2$ induces a ring map
    $(\tau_{\geq2\star}E_1)^{\bT_{\ev}}\to (\tau_{\geq2\star}E_2)^{\bT_\ev}$ which sends $v\in
    \pi_{-2}(\F^{\geq -1}(\tau_{\geq2\star}E_1)^{\bT_{\ev}})$ to a choice of $v$ for $E_2$.
    Conversely, any such map of rings $(\tau_{\geq2\star}E_1)^{\bT_{\ev}}\to
    (\tau_{\geq2\star}E_2)^{\bT_\ev}$ induces a ring map $\tau_{\geq 2\star} E_1\to \tau_{\geq
    2\star}E_2$ by Lemma~\ref{lem:unipotence}. For any fixed map $f\colon\tau_{\geq 2\star} E_1\to
    \tau_{\geq 2\star}E_2$, Lemma~\ref{lem:unipotence} also shows that the natural map
    \[(-)^{\bT_{\ev}}:\mathrm{Map}_{\CAlg(\SynSp_{\bT_\ev})}(\tau_{\geq 2\star} E_1,
    \tau_{\geq 2\star}E_2)_f\to\mathrm{Map}_{\CAlg(\SynSp)}(\tau_{\geq 2\star} E_1^{\bT_\ev},
    \tau_{\geq 2\star}E_2^{\bT_\ev})_{f^{\bT_\ev}}\] is an equivalence on the connected components
    of $f$ onto $f^{\bT_\ev}$. The same argument applies
    after applying $\F^{\geq -\infty}$, so we are reduced to showing the lifts of ring maps on the
    double-speed Whitehead filtration on rings is unique. This is Lemma~\ref{lem:uniqueness
    of ring structure on whitehead tower.}.
\end{proof}

\begin{lemma}~\label{lem:tv_coaction}
    The even filtration naturally lifts to a functor $(-)_{\ev}:\CAlg(\Sp^{\B S^1})^\lev\to \CAlg(\SynSp_{\bT_\ev})$.
\end{lemma}

\begin{proof}
    The even filtration on eff-locally even $\bE_\infty$-rings with $S^1$-action is right Kan
    extended from even $\bE_\infty$-rings with $S^1$-action, and the functor
    $\CAlg(\SynSp_{\bT_\ev})\to \CAlg(\SynSp)$ preserves limits. It is therefore enough to produce,
    for each even $E\in\CAlg(\Sp^{\B S^1})$,
    a functorial lift $\F^{\geq \star}A\in \CAlg(\SynSp_{\bT_\ev})$ of $\tau_{\geq 2\star}E\in
    \SynSp_{\bT_\ev}$. This functorial lift is given by the inverse of $\F^{\geq -\infty}$, which exists in the even case by Lemma~\ref{lem: underling is equivalence on even}. 
\end{proof}

\begin{corollary}\label{cor:contractible}
    Let $E$ be an even $\bE_\infty$-ring with $\bT$-action. The space of $\bT_\ev$-actions on
    $E_\ev$ compatible with the $\bT$-action on $E$ is contractible: the fiber of
    $\CAlg(\SynSp_{\bT_\ev})\rightarrow\CAlg(\Sp_{\bT})$ over $E$ is contractible.
\end{corollary}

\begin{proof}
    Apply Lemma~\ref{lem: underling is equivalence on even}.
\end{proof}

The previous corollary allows us to identify the homotopy types of $(E_\ev)^{\bT_\ev}$ and
$(E_\ev)^{\t\bT_\ev}$ as synthetic spectra with trivial $\bT_\ev$-action using the construction of
Remark~\ref{rem:trivial_coaction}.

\begin{corollary}\label{cor:two_trivs}
    If $E$ is an even $\bE_\infty$-ring with $\bT$-action, then there are natural equivalences
    $(E_\ev)^{\bT_\ev}\we (\tau_{\geq 2\star}(E^{\h S^1}))^\triv$ and
    $(E_\ev)^{\t\bT_\ev}\we(\tau_{\geq 2\star}(E^{\t S^1}))^\triv$.
\end{corollary}

\begin{proof}
    Both $(E_\ev)^{\bT_\ev}$ and $(\tau_{\geq 2\star}(E^{\h S^1}))^\triv$ are trivial
    $\bT_\ev$-actions on $\tau_{\geq 2\star}(E^{\h S^1})$. Thus, they agree by
    Corollary~\ref{cor:contractible}. As the space of such identifications is contractible, the
    equivalence is natural. The argument for Tate is the same.
\end{proof}

\begin{lemma}~\label{lem: even filtration comparison in the even case}
    Fix $n\geq 1$. There are natural equivalences
    \begin{align*}
        (\F^{\geq\star}_{\ev}E)^{\bT_{\ev}}&\simeq \tau_{\geq 2\star}(E^{\h S^1})\we\F^{\geq\star}_{\ev,\bT}E,\\
        (\F^{\geq\star}_{\ev}E)^{\t\bT_{\ev}}&\simeq \tau_{\geq 2\star}(E^{\t S^1})\we\F^{\geq\star}_{\ev,\t\bT}E.\\
        (\F^{\geq\star}_{\ev}E)^{C_{n,\ev}}&\simeq \tau_{\geq 2\star}(E^{\h
        C_n})\we\F^{\geq\star}_{\ev,C_n}E,\text{ and}\\
        (\F^{\geq\star}_{\ev}E)^{\t C_{n,\ev}}&\simeq \tau_{\geq 2\star}(E^{\t
        C_n})\we\F^{\geq\star}_{\ev,\t C_n}E
    \end{align*}
    of $\bE_\infty$-algebras in synthetic spectra
    for all even $\bE_\infty$-rings $E$ with $S^1$-action.
\end{lemma}

\begin{proof}
    The right-hand equivalences follow from the definitions since $E$ is even.
    The existence of natural equivalences from the left-hand side to the right-hand
    side follows in the case of fixed points from
    Lemma~\ref{lem:complete_underlying}(vi), using the fact that $E$ is a
    $\bT$-equivariant $\bE_\infty$-ring under $E^{\h S^1}$, which is even and has trivial
    $S^1$-action.

    To prove the result for the Tate constructions, one proves an analogous
    result for the homotopy orbits $(\F^{\geq\star}_\ev E)_{\bT_\ev}$ and
    $(\F^{\geq\star}_\ev E)_{C_{n,\ev}}$ using
    Lemma~\ref{lem:complete_underlying}. The case of $\bT_\ev$-orbits follows
    using the CW filtration of Construction~\ref{const:cw}, while for $C_{n,\ev}$-orbits one uses 
    the pullback square of Corollary~\ref{cor:nseries}. We leave the details to
    the reader.

    To prove the final claim, we assume for example that $E^{\t C_n}$ is even. Then, thanks to its
    residual $\bT$-action, $\tau_{\geq 2\star}(E^{\t C_n})$ admits a natural $\bT_\ev$-action by
    Lemma~\ref{lem:tv_coaction}. We are claiming that this action agrees with the one arising on
    $(\tau_{\geq 2\star}E)^{\t C_{n,\ev}}$. In fact, the space of $\bT_\ev$-actions compatible with
    the double-speed Whitehead filtration on an even $\bE_\infty$-ring with $\bT$-action is
    conctractible. This follows from Lemma~\ref{lem:tv_coaction}.
\end{proof}

\begin{example}\label{ex:even_tate}
    In the setting of Lemma~\ref{lem: even filtration comparison in the even case}, it follows that
    if $E^{\t C_n}$ is even, then the equivalence $(F^{\geq\star}_\ev E)^{\t
    C_{n,\ev}}\we\tau_{\geq 2\star}(E^{\t C_n})$ is naturally (uniquely) $\bT_\ev$-equivariant, where
    $(\F^{\geq\star}_\ev E)^{\t C_{n,\ev}}$ is given the $\bT_\ev$-action arising from taking $\t
    C_{n,\ev})$ fixed points and $\tau_{\geq 2\star}(E^{\t C_{n,\ev}})$ is given the
    $\bT_\ev$-action arising from Lemma~\ref{lem:tv_coaction}.
\end{example}

As another application of Lemma~\ref{lem: even filtration comparison in the even case},
we can give a more precise description of the $E_\ev^{\bT_\ev}$-$\bE_\infty$-algebra
structure on $E_\ev^{C_{n,\ev}}$ and $E_\ev^{\t C_{n,\ev}}$ for arbitrary $n$.

\begin{corollary}~\label{cor: mod n-series is cn fixed points}
    Let $E$ be an $S^1$-equivariant even $\bE_\infty$ algebra. Let $\bS_\ev[-2](1)\to
    \tau_{\geq 2\star}(E^{\h S^1})\simeq E_\ev^{\bT_\ev}$ be the map $[n](v)\in \pi_{-2}(E^{\h S^1})$ given
    by the $n$-series associated to the formal group law of $E$. Then the composition
    $\bS_\ev[-2](1)\to E_\ev^{\bT_\ev}\to E_\ev^{C_{n,\ev}}$ is nullhomotopic and the induced map
    \[E_\ev^{\bT_\ev}/[n](v)\to E_\ev^{C_{n,\ev}}\] is an equivalence of synthetic spectra.
\end{corollary}

\begin{proof}
    By adjunction, the map of synthetic spectra is determined by the map $\ins^0\bS[-2](1)\to
    E_\ev^{C_{n,\ev}}$ which in turn by Lemma~\ref{lem: even filtration comparison in the even
    case} is determined by the map $\bS[-2]\to E^{\h C_n}$ given by the $n$-series on $v$. This
    element is zero, and therefore the map on synthetic spectra $\bS_\ev[-2](1)\to
    E_\ev^{C_{n,\ev}}$ is also nullhomotopic.

    We now have a cofiber sequence of $E_\ev^{\bT_\ev}$-modules \[E_\ev^{\bT_{\ev}}[-2](1)\to
    E_\ev^{\bT_\ev}\to E_\ev^{\bT_\ev}/[n](v)\] which on $\F^{\geq i}$ is then a fiber sequence
    \[\Sigma^{-2}\tau_{\geq 2i+2}E^{\h S^1}\to \tau_{\geq 2i}E^{\h S^1}\to \F^{\geq
    i}(E_\ev^{\bT_\ev})\] and the map from this fiber sequence to the fiber sequence
    \[\Sigma^{-2}\tau_{\geq 2i+2}E^{\h S^1}\to \tau_{\geq 2i}E^{\h S^1}\to \tau_{\geq 2i}E^{\h C_n}\] shows
    that $E^{\bT_\ev}_\ev/[n](v)\to E_\ev^{C_{n,\ev}}$ is an equivalence, as desired.
\end{proof}

\begin{corollary}
    Let $E$ be an $S^1$-equivariant even $\bE_\infty$-algebra. Then the map
    \[E_\ev^{\t\bT_\ev}\otimes_{E_\ev^{\bT_\ev}}E_{\ev}^{C_{n,\ev}}\to E_\ev^{\t C_{n,\ev}}\] is an
    equivalence of $\bT_\ev$ equivariant $\bE_\infty$-algebras.
\end{corollary}

\begin{proof}
    It is enough to show that the induced map on underlying synthetic spectra
    $E^{\t\bT_\ev}_\ev/[n](v)\to E_\ev^{\t C_{n,\ev}}$ is an equivalence. Both are the double speed
    Postnikov filtration on $E^{\t C_n}$ by Lemma~\ref{lem: even filtration comparison in the even
    case}, so the result follows.
\end{proof}

By descent, we find that Lemma~\ref{lem: even filtration comparison in the even case} extends to $S^1$-equivariantly
eff-locally even $\bE_\infty$-rings.

\begin{lemma}~\label{lem: even filtration comparison general}
    Let $R$ be a connective $\bE_\infty$-ring with $S^1$-action
    which admits an $S^1$-equivariant eff cover by an even $\bE_\infty$-ring
    $E$ with $S^1$-action. Fix $n\geq 1$. There are natural equivalences
    \begin{align*}
        (\F^{\geq\star}_{\ev}R)^{\bT_{\ev}}&\simeq \F^{\geq\star}_{\ev,\bT}R,\\
        (\F^{\geq\star}_{\ev}R)^{\t\bT_{\ev}}&\simeq\F^{\geq\star}_{\ev,\t\bT}R,\\
        (\F^{\geq\star}_{\ev}R)^{C_{n,\ev}}&\we\F^{\geq\star}_{\ev,C_n}R,\text{ and}\\
        (\F^{\geq\star}_{\ev}R)^{\t C_{n,\ev}}&\we\F^{\geq\star}_{\ev,\t C_n}R.
    \end{align*}
\end{lemma}

\begin{proof}
    The equivalence on $\bT_\ev$ and $C_{n,\ev}$ fixed points follows from the
    fact that these functors preserve limits and by Lemma~\ref{lem: even filtration
    comparison in the even case}.

    Let $E^\bullet$ be the \v{C}ech complex of $R\rightarrow E$. This is a
    cosimplicial even $\bE_\infty$-ring with $S^1$-action.
    By functoriality, there is a natural map $$(\F^{\geq\star}_\ev
    R)^{\t S^1}\rightarrow\Tot((\F^{\geq\star}_\ev
    E^\bullet)^{\t S^1})\we\Tot(\tau_{\geq 2\star}((E^\bullet)^{\t S^1})),$$
    where the equivalence is by Lemma~\ref{lem: even filtration comparison in
    the even case}. It suffices to argue that under the hypotheses of the
    lemma, taking the $\bT_\ev$-Tate construction commutes with the limit.
    For this, it is equivalent to checking that the induced map
    $$(\F^{\geq\star}_\ev
    R)_{\bT_\ev}\rightarrow\Tot((\F^{\geq\star}E^\bullet)_{\bT_\ev})$$ is an
    equivalence.

    To prove this,
    we may assume without loss of generality that $E$ is connective since $R$ is.
    We then have that $\F^{\geq\star}_\ev R$ is connective in the Postnikov
    $t$-structure by the results of Burkland and Krause, so
    the filtration on $(\F^{\geq\star}_\ev R)_{\bT_\ev}$
    is complete by Lemma~\ref{lem:complete_underlying}(ii).
    Since $\gr^i_\ev R\we 0$ for $i<0$, we have that $\gr^N(\F^{\geq\star}_\ev R)_{\bT_\ev}$ admits a finite
    filtration with associated graded pieces the weight $N$ piece of $\gr^i_\ev
    R\otimes_{\bT_\gr}\bS_\gr$, for $0\leq i\leq N$.
    It is thus enough to check that
    $\gr^i_\ev R\otimes_{\bT_\gr}\bS_\gr\we\Tot((\pi_{2i}E^\bullet[2i])\otimes_{\bT_\gr}\bS_\gr)$
    for all $i\geq 0$. Since both sides are concentrated in a single weight,
    the $\bT_\gr$-action is trivial, so we must check that $$\gr^i_\ev
    R\otimes_{\bS_\gr}\bS_\gr\otimes_{\bT_\gr}\bS_\gr\rightarrow\Tot((\pi_{2i}E^\bullet(i)[2i])\otimes_{\bS_\gr}\bS_\gr\otimes_{\bT_\gr}\bS_\gr)$$
    is an equivalence of graded spectra. However, the calculation in a single
    weight $j$ involves only $\bS_\gr$ and $\bT_\gr$ in weights $0\leq k\leq
    j$. In these weights, $\bS_\gr$ is a perfect graded $\bT_\gr$-module, and
    hence tensoring (in graded spectra with weights in $[0,j]\cap\bZ$) commutes
    with totalizations. This completes the proof for the $\bT_\ev$-Tate
    construction. The $C_{n,\ev}$-Tate construction is similar, using that
    $\rho(n)_*\bS_\gr$ is perfect over $\bS_\gr$ in any bounded range of
    non-negative weights.
\end{proof}

We conclude with the example of rational synthetic spectra.

\begin{example}[Rational Tate vanishing]\label{ex:rational_vanishing}
    Note that $(\bS_\ev)_\bQ\we\ins^0\bQ\we\bQ_\ev$. If $\F^{\geq\star}M$ is a synthetic $\bQ$-module with
    $\bT_\ev$-action, then $(\F^{\geq\star}M)^{\t C_{p,\ev}}\we 0$ for all primes $p$.
    Indeed, $(\F^{\geq\star}M)^{\t C_{p,\ev}}$ is a synthetic $(\bQ_\ev)^{\t C_{p,\ev}}$-module by
    Proposition~\ref{prop:monoidal},
    but the latter vanishes as it is equivalent to $\tau_{\geq 2\star}(\bQ^{\t C_{p}})$ by
    Lemma~\ref{lem: even filtration comparison in the even case}.
\end{example}

\subsection{Topological Hochschild homology as a synthetic cyclotomic spectrum}

In this section we achive the proof that $\CycSyn$ does indeed make a home for the BMS
and even filtrations.

\begin{definition}[Quasisyntomic rings]
    We recall some definitions from~\cite{bms2,hrw}.
    \begin{enumerate}
        \item[(a)] A commutative ring $R$ is integrally quasisyntomic if it has bounded $p$-torsion
            for all primes $p$ and if it is quasi-lci over $\bZ$ in the sense that $\L_{R/\bZ}$ has Tor-amplitude in $[0,1]$.
            Write $\QSyn$ for the category of integrally quasisyntomic commutative rings.
        \item[(b)] A commutative ring $R$ is $p$-quasisyntomic if it has bounded $p$-torsion and is
            $p$-quasi-lci over $\bZ$ in the sense that $\L_{R/\bZ}$ has $p$-complete Tor-amplitude
            in $[0,1]$. If $R$ is additionally $p$-complete, this definition agrees with the one
            given in~\cite{bms2}. Let $\QSyn_p$ denote the category of $p$-quasisyntomic
            commutative rings.
        \item[(c)] An $\bE_\infty$-ring spectrum $R$ is chromatically quasisyntomic if
            $R\otimes_\bS\MU$ is even and $\oplus_*(\pi_*(R\otimes_\bS\MU))$ is quasi-lci over $\bZ$ and
            has bounded $p$-power torsion for all primes $p$. Let $\CQSyn$ denote the $\infty$-category of
            chromatically quasisyntomic $\bE_\infty$-ring spectra.
        \item[(d)] An $\bE_\infty$-ring spectrum $R$ is chromatically $p$-quasisyntomic if
            $R\otimes_\bS\MU$ is even and $\oplus_*(\pi_*(R\otimes_\bS\MU))$ is $p$-quasisyntomic.
            Let $\CQSyn_p$ denote the $\infty$-category of chromatically $p$-quasisyntomic
            $\bE_\infty$-ring spectra.
    \end{enumerate}
\end{definition}

Motivic filtrations have been constructed on $\THH(R)$ or its $p$-completion as well as on $\TC^-$,
$\TP$, $\TC$ in these cases in~\cite{bhatt-lurie-apc,bms2,hrw,morin}. We will show in this
section that the motivic filtrations agree with those arising from natural
functors from the categories of quasisyntomic rings singled out above to commutative algebras in
cyclotomic synthetic spectra. The real point of our work is the construction with target
$\CAlg(\CycSyn)$. Once this has been carried out, the main work of the comparison results has been
carried out already in~\cite[Sec.~5]{hrw}. 

We will begin by considering the case of producing a cyclotomic synthetic spectrum from an even cyclotomic spectrum.
Let $\CAlg(\CycSp)^\ev$ denote the full subcategory of commutative algebras in ($p$-typical)
cyclotomic spectra whose underlying spectrum is even.

\begin{proposition}
    The functor $\F^{\geq -\infty}\colon\CAlg(\CycSyn)^{\ev}\to \CAlg(\CycSp)^{\ev}$ is an equivalence of $\infty$-categories.
\end{proposition}

\begin{proof}
    Once again it is enough to show that this functor is essentially surjective and fully faithful. Both follow from Lemma~\ref{lem:unipotence} via a very similar argument as in Lemma~\ref{lem: underling is equivalence on even}.

    Let $(E,\varphi_E)\in \CAlg(\CycSp)^{\ev}$. We have that $E_\ev$ lifts to an element of $\CAlg(\SynSp_{\bT_\ev})$, so it is enough to produce a filtered lift of $\varphi_E$. The Frobenius $\varphi_E$ induces a map $\tau_{\geq 2\star}(E^{\h S^1})\rightarrow\tau_{\geq 2\star}((E^{\t C_p})^{\h S^1})$ of synthetic $\bE_\infty$-algebras. Under the equivalence of Lemma~\ref{lem:unipotence} this corresponds to a map $E_\ev\to (\tau_{\geq 2\star}E^{\t C_p})^{\h S^1}/\varphi_E(v)$ of $\bE_\infty$-algebras in $\SynSp_{\bT_\ev}$ for some such structure on $(\tau_{\geq 2\star}E^{\t C_p})^{\h S^1}/\varphi_E(v)$. We claim that this agrees with $E_\ev^{\t C_{p,\ev}}$. 

    Note that on underlying synthetic spectra we have that $(\tau_{\geq 2\star}E^{\t C_p})^{\h S^1}/\varphi_E(v)\simeq \tau_{\geq 2\star}E^{\t C_p}$. Thus $\left((\tau_{\geq 2\star}E^{\t C_p})^{\h S^1}/\varphi_E(v)\right)^{\bT_\ev}\simeq \tau_{\geq 2\star}((E^{\t C_p})^{\h S^1})$ by Lemma~\ref{lem:complete_underlying}(vi). This agrees with $(E_\ev^{\t C_{p,\ev}})^{\bT_{\ev}}$, so by Lemma~\ref{lem:unipotence} we have that $(\tau_{\geq 2\star}E^{\t C_p})^{\h S^1}/\varphi_E(v)\simeq E_\ev^{\t C_{p,\ev}}$ as objects of $\CAlg(\SynSp_{\bT_\ev})$, as desired.

    We will now show full faithfulness. To this end let $(\tau_{\geq 2\star}E, \varphi_1),
    (\tau_{\geq 2\star}E_2,\varphi_2)\in \CAlg(\CycSyn)^\ev$.
    Note first that the map \[\mathrm{Map}_{\CAlg(\CycSyn)}((\tau_{\geq 2\star} E_1,\varphi_1),
    (\tau_{\geq 2\star}E_2, \varphi_2))\to \mathrm{Map}_{\CAlg(\SynSp_{\bT_\ev})}(\tau_{\geq
    2\star}E_1, \tau_{\geq 2\star}E_2)\] is injective on $\pi_0$. This follows from the fact that
    for any $\infty$-category $\mathcal{C}$, the functor
    $\mathcal{C}^{\Delta^1}\xrightarrow{(\ev_0, \ev_1)}\mathcal{C}\times \mathcal{C}$ is a
    categorical fibration, see \cite[Cor.~2.3.2.5, Cor.~2.4.6.5]{htt}.
    The faithfulness of the functor $\F^{\geq -\infty}$ on $\pi_0$ of mapping spaces then follows from Lemma~\ref{lem: underling is equivalence on even}. Surjectivity on $\pi_0$ follows from the same argument as in the proof of essential surjectivity in the previous paragraphs.

    Now fix a map $f\in \mathrm{Map}_{\CAlg(\CycSyn)}((\tau_{\geq 2\star}E_1, \varphi_1),
    (\tau_{\geq 2\star}E_2, \varphi_2))$. We now must show that the induced map
    \[\mathrm{Map}_{\CAlg(\CycSyn)}((\tau_{\geq 2\star}E_1, \varphi_1), (\tau_{\geq 2\star}E_2,
    \varphi_2))_f\to \mathrm{Map}_{\CAlg(\CycSp)}((E_1,\varphi_1), (E_2, \varphi_2))_{\F^{\geq
    -\infty}f}\] on connected components is a weak equivalence. The fiber sequence of
    \cite[Prop.~II.5.1(ii)]{nikolaus-scholze} then reduces this to showing that
    \[\mathrm{Map}_{\CAlg(\SynSp_{\bT_\ev})}(\tau_{\geq 2\star}E_1, \tau_{\geq 2\star}E_2)_f\to
    \mathrm{Map}_{\CAlg(\Sp^{\B S^1})}(E_1, E_2)_{\F^{\geq -\infty}f}\] and
    \[\mathrm{Map}_{\CAlg(\SynSp_{\bT_{\ev}})}(\tau_{\geq 2\star}E_1, (\tau_{\geq 2\star}E_2)^{\t
    C_{p,\ev}})_{\varphi_2\circ f}\to \mathrm{Map}_{\CAlg(\Sp^{\B S^1})}(E_1, E_2^{\t
    C_p})_{\varphi_2\circ \F^{\geq -\infty}f}\] are weak equivalences. The first is a direct
    consequence of Lemma~\ref{lem: underling is equivalence on even}, and the second is a slight
    variation using Lemma~\ref{lem:unipotence}.
\end{proof}

Descending this result then gives us the following.

\begin{corollary}\label{cor:cyclotomic_functor}
    Let $R$ be a cyclotomic $\bE_\infty$-ring. If $R$ is locally even in cyclotomic spectra, then
    $\F^{\geq\star}_\ev R$ naturally admits the
    structure of a commutative algebra object in $\CycSyn$. In other words, there is a functor
    $\CAlg(\CycSp)^\lev\rightarrow\CAlg(\CycSyn)$.
\end{corollary}

\begin{proof}
    Let $R\to E$ be an eff cover of $R$ by a cyclotomic spectrum $(E,\varphi_E)$.
    Then the previous Lemma produces a natural lift of $E^\bullet_{\ev}$ to a diagram in
    $\CAlg(\CycSyn)$. Thus taking totalizations and noting the the forgetful functors
    $\CAlg(\CycSyn)\to \CycSyn\to \SynSp_{\bT\ev}$ preserve limits we get a natural lift of $R_\ev$
    to an element of $\CAlg(\CycSyn)$.
\end{proof}

\begin{theorem}[Even filtration comparison]\label{thm:even_comparison}
    There are functors
    \[\F^{\geq\star}_\ev\THH(-)\colon\mathrm{CQSyn}\to \CAlg(\CycSyn)\] and
    \[\F^{\geq\star}_\ev\THH(-;\bZ_p)\colon\mathrm{CQSyn}_p\to \CAlg(\CycSyn_p^\wedge)\] 
    with the following properties.
    \begin{enumerate}
        \item[{\em (a)}] For $R\in\mathrm{CQSyn}$, there are natural equivalences
            \begin{align*}
                \F^{\geq\star}_\ev\THH(R)&\we\Fil^\star_\HRW\THH(R),\\
                (\F^{\geq\star}_\ev\THH(R))^{\bT_\ev}&\we\Fil^\star_\HRW\TC^-(R),\text{ and}\\
                (\F^{\geq\star}_\ev\THH(R))^{\t\bT_\ev}&\we\Fil^\star_\HRW\TP(R),
            \end{align*}
            of filtered $\bE_\infty$-algebras (with circle action for the first equivalence).
        \item[{\em (b)}] For $R\in\mathrm{CQSyn}_p$, there are natural equivalences
            \begin{align*}
                \F^{\geq\star}_\ev\THH(R;\bZ_p)&\we\Fil^\star_\HRW\THH(R;\bZ_p),\\
                (\F^{\geq\star}_\ev\THH(R;\bZ_p))^{\bT_\ev}&\we\Fil^\star_\HRW\TC^-(R;\bZ_p),\\
                ((\F^{\geq\star}_\ev\THH(R;\bZ_p))^{\t\bT_\ev})_p^\wedge &\we\Fil^\star_\HRW\TP(R;\bZ_p),\text{ and}\\
                \TC(\F^{\geq\star}_\ev\THH(R;\bZ_p))&\we\Fil^\star_\HRW\TC(R;\bZ_p)
            \end{align*}
            of filtered $\bE_\infty$-algebras (with circle action for the first equivalence).
    \end{enumerate}
\end{theorem}

\begin{proof}
    Existence of the functors and all parts of (a) and (b) not pertaining to $\TC$ follow from
    Corollary~\ref{cor:cyclotomic_functor}, Lemma~\ref{lem: even filtration comparison general},
    and the results of~\cite[Sec.~4.2]{hrw}.

    For $\TC$, we have
    by~\cite[Proposition II.1.5(ii)]{nikolaus-scholze}, an equalizer formula for
    $\TC(\F^{\geq \star}_\ev \THH(-;\bZ_p))$ and the result amounts to identifying
    $\bMap_{\SynSp_{\bT_{\ev}}}(\bS_{\ev}, (\F^{\geq \star}_\ev \THH(-;\bZ_p))^{\t C_{p,\ev}})$ with
    $(\F^{\geq \star}_\ev \THH(-;\bZ_p))^{t\bT_{\ev}}$. This follows from Lemma~\ref{lem: Tate then
    fixed points is Tate} below.
\end{proof}

\begin{variant}[Relative theories]
    More generally, if $k\rightarrow R$ is a chromatically quasi-lci map of connective
    $\bE_\infty$-rings, then there is a natural $\F^{\geq\star}_\ev\THH(R/k)\in\CAlg(\CycSyn)$
    lifting the filtration $\fil^\star\THH(R/k)$ constructed in~\cite[Def.~2.4.1]{hrw}.
    The analogue of property (b) from Theorem~\ref{thm:even_comparison} holds in this relative case
    as well.
\end{variant}

We can also compare to the filtration on $p$-adic $\THH$ for $p$-quasisyntomic rings constructed
in~\cite{bms2} and to the filtration on integral $\THH$ for quasisyntomic rings constructed by
Morin in~\cite[Def.~1.1]{morin} and by Bhatt--Lurie in~\cite[Sec.~6.2]{bhatt-lurie-apc}.

\begin{theorem}[BMS comparison]\label{thm:bms_comparison}
    There are functors
    \[\F^{\geq\star}_\ev\THH(-)\colon\mathrm{QSyn}\to \CAlg(\CycSyn)\] and
    \[\F^{\geq\star}_\ev\THH(-;\bZ_p)\colon\mathrm{QSyn}_p\to \CAlg(\CycSyn_p^\wedge)\] 
    with the following properties.
    \begin{enumerate}
        \item[{\em (a)}] For $R\in\mathrm{QSyn}$, there are natural equivalences
            \begin{align*}
                \F^{\geq\star}_\ev\THH(R)&\we\Fil^\star_\BLM\THH(R),\\
                (\F^{\geq\star}_\ev\THH(R))^{\bT_\ev}&\we\Fil^\star_\BLM\TC^-(R),\text{ and}\\
                (\F^{\geq\star}_\ev\THH(R))^{\t\bT_\ev}&\we\Fil^\star_\BLM\TP(R),
            \end{align*}
            of filtered $\bE_\infty$-algebras (with circle action for the first equivalence).
        \item[{\em (b)}] For $R\in\mathrm{QSyn}_p$, there are natural equivalences
            \begin{align*}
                \F^{\geq\star}_\ev\THH(R;\bZ_p)&\we\Fil^\star_\BMS\THH(R;\bZ_p),\\
                (\F^{\geq\star}_\ev\THH(R;\bZ_p))^{\bT_\ev}&\we\Fil^\star_\BMS\TC^-(R;\bZ_p),\\
                ((\F^{\geq\star}_\ev\THH(R;\bZ_p))^{\t\bT_\ev})_p^\wedge &\we\Fil^\star_\BMS\TP(R;\bZ_p),\text{ and}\\
                \TC(\F^{\geq\star}_\ev\THH(R;\bZ_p))&\we\Fil^\star_\BMS\TC(R;\bZ_p)
            \end{align*}
            of filtered $\bE_\infty$-algebras (with circle action for the first equivalence).
    \end{enumerate}
\end{theorem}

%
%


\begin{proof}
    For part (b),
    the functors are constructed as in~\cite{bms2} by unfolding from the quasiregular semiperfectoid case.
    If $R$ is a quasiregular semiperfectoid, then $\THH(R;\bZ_p)$ is $p$-completely Tate-even. This
    is proved in~\cite[Sec.~3]{riggenbach-truncated}, together with statements for $C_{p^n}$-fixed
    points and Tate for all $n\geq 1$. In the $n=1$ case, it also follows via a comparison to
    Nygaard-complete Hodge--Tate cohomology. Thus, $\THH(R;\bZ_p)$ canonically
    admits the structure of a $p$-complete synthetic cyclotomic $\bE_\infty$-algebra by
    Corollary~\ref{cor:cyclotomic_functor}.
    Using that the forgetful functor $\CAlg(\CycSyn_p^\wedge)\rightarrow\SynSp_p^\wedge$ is
    conservative, the general case follows by quasisyntomic descent.

    For part (a), we use the pullback square
    $$\xymatrix{
        \THH(R)\ar[r]\ar[d]&\prod_p\THH(R;\bZ_p)\ar[d]\\
        \HH(R)_\bQ\ar[r]&(\prod_p\THH(R;\bZ_p))_\bQ
    }$$
    of spectra with $\bT$-action. Taking even filtrations,
    we obtain a $\bT_\ev$-action on
    $\THH(R)$ by pullback. We make this into a cyclotomic synthetic spectrum by equipping the
    rational synthetic spectra with $\bT_\ev$-action with the zero Frobenii (the only possible
    choice thanks to Example~\ref{ex:rational_vanishing}). We let $\F^{\geq\star}_\ev\THH(R)$ be the 
    pullback in cyclotomic synthetic spectra. The comparison results now follow by construction of
    the Bhatt--Lurie--Morin filtrations in~\cite[Sec.~6.4]{bhatt-lurie-apc} and~\cite[Def.~1.1]{morin}.
\end{proof}

%


There is a natural map of fiber sequences 
\[
\begin{tikzcd}
    \F^{\geq \star}X_{C_{p^{n+1},\ev}} \ar[r] \ar[d] & \F^{\geq \star}X^{C_{p^{n+1},\ev}} \ar[d]
    \ar[r] & \F^{\geq \star}X^{tC_{p^{n+1},\ev}} \ar[d]\\
    (\F^{\geq \star}X_{C_{p,\ev}})^{C_{p^n,\ev}} \ar[r] & \F^{\geq \star}X^{C_{p^{n+1},\ev}} \ar[r]
    & (\F^{\geq \star}X^{tC_{p,\ev}})^{C_{p^n,\ev}}
\end{tikzcd}
\] where the left vertical map is the norm. We used the following lemmas above.

\begin{lemma}\label{lem:homotopy_to_tate}
    Let $\F^{\geq \star} X\in \SynSp_{\bT_\ev}$ be bounded below in the Postnikov $t$-structure.
    Then the natural map $(\F^{\geq \star}X)^{\t C_{p^{n+1},\ev}}\to (\F^{\geq
    \star}X^{\t C_{p,\ev}})^{C_{p^n,\ev}}$ is an equivalence for all $n\geq 0$.
\end{lemma}

\begin{proof}
    The proof that the maps $(\F^{\geq \star}X)^{\t C_{p^{n+1},\ev}}\to (\F^{\geq
    \star}X^{\t C_{p,\ev}})^{C_{p^n,\ev}}$ are equivalences follows from Lemma~\ref{lem: Tate orbit
    lemma} in the same way as in \cite[Lemma II.4.1]{nikolaus-scholze}.
\end{proof}

\begin{lemma}~\label{lem: Tate then fixed points is Tate}
    Let $R$ be a connective $\bE_\infty$-ring with $S^1$-action which admits an $S^1$-equivariant
    eff cover by and even $\bE_\infty$-ring $E$ with $S^1$-action. Then the maps in the commutative
    diagram
    \[
    \begin{tikzcd}
        \F^{\geq \star}_\ev R^{\t\bT_\ev} \ar[r] \ar[d] & (\F^{\geq \star}_\ev R^{\t C_{p,\ev}})^{\bT_{\ev}}\ar[d]\\
        \lim \F^{\geq \star}_\ev R^{\t C_{p^{n+1},\ev}} \ar[r] & \lim (\F^{\geq \star}_\ev R^{\t C_{p,\ev}})^{C_{p^n,\ev}}
    \end{tikzcd}
    \]
    are $p$-adic equivalences.
\end{lemma}

\begin{proof}
    We have already seen in Lemma~\ref{lem:homotopy_to_tate} that
    the bottom map is an equivalence even before $p$-completion (or before taking the limit). The
    right hand vertical map is an equivalence $p$-adically since this reduces to the statement that
    the map $\colim \rho(p^n)_*\bT_{\ev}\to \bS_{\ev}$ is a $p$-adic equivalence which follows from
    the fiber sequence of Lemma~\ref{lem: fiber sequence for Tsyn} and the identification of the
    map $\rho(nm)_*\bT_\ev\to \rho(n)_*\bT_{\ev}$ as multiplication by $m$ on the fiber. Finally,
    the left vertical map is a $p$-adic equivalence by Lemma~\ref{lem: even filtration comparison
    general} and the proof of~\cite[Lem.~II.4.2]{nikolaus-scholze}.
\end{proof}

\section{The synthetic cyclotomic $t$-structure}\label{sec:t_structure}

Following \cite{an1}, we will now construct a $t$-structure on $\CycSyn$ starting
with a $t$-structure $\C$ on $\SynSp$ satisfying $\mathbf{(\star)}$ in the sense of
Definition~\ref{def:t_star}.

\subsection{Synthetic $t$-structures}

\begin{theorem}~\label{thm: construction of cyclotomic $t$-structure}
    If $\C$ is a $t$-structure on $\SynSp$ which satisfies $\mathbf{(\star)}$,
    then there exists an accessible $t$-structure on $\CycSyn$ which is
    \begin{enumerate}
        \item[{\em (1)}] compatible with the symmetric monoidal structure on $\CycSyn$;
        \item[{\em (2)}] left separated;
        \item[{\em (3)}] left complete and compatible with countable products.
    \end{enumerate}
    Moreover, the forgetful functor $\CycSyn\rightarrow\SynSp$ is right $t$-exact.
\end{theorem}

We will write $\C=(\CycSyn_{\geq 0}^\C,\CycSyn_{\leq 0}^\C)$ for this $t$-structure as well.

\begin{proof}[Proof of Theorem~\ref{thm: construction of cyclotomic $t$-structure}]
    Let $\CycSyn_{\geq 0}^\C\subseteq\CycSyn$ be the full subcategory of objects $(\F^\star
    X,\varphi)$ such that the underlying synthetic spectrum $\F^\star X$ is in $\SynSp_{\geq
    0}^\C$. Note that we can describe $\CycSyn_{\geq 0}^\C$ instead as \[\CycSyn_{\geq 0}^\C\simeq
    \mathrm{Eq}(\id, \tau_{\geq 0}^\C(-)^{\t C_{p,\ev}}:\SynSp^\C_{\bT_{\ev},\geq
    0}\rightrightarrows\SynSp^\C_{\bT_{\ev},\geq 0})\] since any map $\F^{\geq \star}X\to (\F^{\geq
    \star}X)^{\t C_{p,\ev}}$ factors functorially and uniquely through a map $\F^{\geq \star}X\to
    \tau_{\geq 0}^\C\F^{\geq \star}X^{\t C_{p,\ev}}$ for $\F^{\geq \star}X$ connective with respect
    to $\C$. Thus $\CycSyn_{\geq 0}$ is presentable by \cite[Proposition
    II.1.5]{nikolaus-scholze} and closed under extensions in $\SynSp$. By \cite[1.4.4.11]{ha}, there is a unique $t$-structure
    $\C=(\CycSyn_{\geq 0}^\C,\CycSyn_{\leq 0}^\C)$ on $\CycSyn$.

    The forgetful functor $\CycSyn\to \SynSp_{\bT_{\ev}}$ is right $t$-exact by definition. We also
    have that $\C$ is compatible with the symmetric monoidal structure on $\CycSyn$
    since the unit $\bS_\ev^\triv$ is connective and since the tensor product of connectives is
    connective. This proves (1).
    
    Left separatedness follows from the fact that the functor $\CycSyn\to \SynSp_{\bT_{\ev}}$ is
    conservative; this shows (2).
    To show part (3),
    it is enough to show that $\CycSyn_{\geq 0}^\C$ is closed under countable products in $\CycSyn$ by
    \cite[1.2.1.19]{ha}. This follows assuming that the forgetful functor $\CycSyn\to \SynSp_{\bT_{\ev}}$
    commutes with countable products of $\C$-connective objects, which in turn follows from
    \cite[Prop.~II.1.5(v)]{nikolaus-scholze} and the fact that for $\{\F^{\geq \star}Y_k\}_k$
    a family of
    $\C$-connective objects, \[\left(\prod_{k}\F^{\geq \star}Y_k\right)^{\t C_{p,\ev}}\to \prod_k \F^{\geq
    \star}Y_k^{\t C_{p,\ev}}\] is an equivalence by Lemma~\ref{lem: orbit/fixedpoints/tate
    continuity} (taking $\F^{\geq \star}X_k:=\prod_{i=1}^k \F^{\geq \star}Y_i$.).
\end{proof}

\subsection{Synthetic topological Cartier Modules}

In order to analyze the heart of the cyclotomic synthetic $t$-structures, we introduce synthetic
Cartier modules and their associated $t$-structures.

\begin{definition}
    Define the $\infty$-category $\syncart$ to
    be the $\infty$-category of objects $\F^{\geq \star}X\in \SynSp_{\bT_{\ev}}$
    equipped with an equivariant factorization \[(\F^{\geq
    \star}X)_{C_{p,\ev}}\xrightarrow{V}\F^{\geq \star}X\xrightarrow{F}(\F^{\geq
    \star}X)^{C_{p,\ev}}\] of the norm map $\Nm_p\colon\F^{\geq \star}X_{C_{p,\ev}}\to \F^{\geq
    \star}X^{C_{p,\ev}}$ of Construction~\ref{const: Tate Construction finite groups.}. More
    precisely, $\syncart$ is the pullback of the diagram \[
    \begin{tikzcd}
     & \SynSp_{\bT_{\ev}} \ar[d,"(\id{,} \Nm_p)"]\\
     \SynSp_{\bT_{\ev}}^{\Delta^2}\ar[r,"(\ev_{1}{,}\partial^1)"]&\SynSp_{\bT_{\ev}}\times\SynSp_{\bT_{\ev}}^{\Delta^1}
    \end{tikzcd}
    \] in $\mathrm{Cat}_{\infty}$. We call the objects of $\syncart$ synthetic Cartier modules.
\end{definition}

There are two key examples to keep in mind, which we outline now.

\begin{example}
    Let $(\F^{\geq \star}X,\varphi)\in \CycSyn_p$. Define the filtered topological restriction
    homology of $\F^{\geq\star}X$ to be the synthetic spectrum \[\mathrm{TR}(\F^{\geq \star}X):=
    \mathrm{Eq}\left(\prod_{n\geq 0}\F^{\geq \star}X^{C_{p^n,\ev}}\substack{\xrightarrow{\prod
    \can^{C_{p^{n-1},\ev}}}\\\xrightarrow[\prod \varphi_p^{C_{p^n,\ev}}]{\hphantom{\prod
    \can^{C_{p^{n-1},\ev}}}}} \prod_{n\geq 0}(\F^{\geq \star}X^{\t
    C_{p,\ev}})^{C_{p^n,\ev}}\right)\] where the top map is given on the $n$th factor as the
    composition \[\F^{\geq \star}X^{C_{p^n,\ev}}\simeq (\F^{\geq
    \star}X^{C_{p,\ev}})^{C_{p^{n-1},\ev}}\xrightarrow{\can^{C_{p^{n-1},\ev}}} (\F^{\geq \star}X^{\t
    C_{p,\ev}})^{C_{p^{n-1},\ev}}\] and the bottom map is given by $\varphi^{C_{p^n,\ev}}\colon\F^{\geq
    \star}X^{C_{p^n,\ev}}\to (\F^{\geq \star}X^{\t C_{p,\ev}})^{C_{p^{n},\ev}}$.
    
    The filtered topological restriction homology naturally carries a synthetic Cartier module
    structure. To see this, note that $\mathrm{TR}(\F^{\geq \star}X)$ is the equalizer of two
    objects in $\SynSp_{\bT_{\ev}}$, and so naturally has a $\bT_{\ev}$ action. We may then define
    $V\colon\mathrm{TR}(\F^{\geq \star}X)_{C_{p,\ev}}\to \mathrm{TR}(\F^{\geq \star}X)$ to be the map
    induced by \[\left(\prod_{n\geq 0}\F^{\geq \star}X^{C_{p^n,\ev}}\right)_{C_{p,\ev}}\to
    \prod_{n\geq 0}(\F^{\geq \star}X^{C_{p^n,\ev}})_{C_{p,\ev}}\xrightarrow{\prod_{n\geq
    0}\Nm_p}\prod_{n\geq 0}\F^{\geq \star}X^{C_{p^{n+1},\ev}}\] and the same map with $\F^{\geq
    \star}X$ replaced with $\F^{\geq \star}X^{\t C_{p,\ev}}$. The Frobenius map
    $F\colon\mathrm{TR}(\F^{\geq \star}X)\to \mathrm{TR}(\F^{\geq \star}X)^{C_{p,\ev}}$ is the map
    induced by the maps \[\prod_{n\geq 0}\F^{\geq \star}X^{C_{p^n,\ev}}\xrightarrow{\proj}
    \prod_{n\geq 1}\F^{\geq \star}X^{C_{p^{n},\ev}}=\prod_{n\geq 0}\F^{\geq
    \star}X^{C_{p^{n+1},\ev}}\simeq \left(\prod_{n\geq 0}\F^{\geq
    \star}X^{C_{p^n,\ev}}\right)^{C_{p,\ev}}\] and the same map with $\F^{\geq \star}X$ replaced
    with $\F^{\geq \star}X^{\t C_{p,\ev}}$.
\end{example}

\begin{remark}~\label{rem: tc=trf=1}
    Classically topological cyclic homology was defined as a certain equalizer of topological
    restriction homology. One might wonder if such a formula is still true in this synthetic
    language. Indeed, there is a functorial fiber sequence \[\TC(\F^{\geq \star}X)\to \TR(\F^{\geq
    \star}X)\to \TR(\F^{\geq \star}X)\] for any $p$-complete and bounded below $\F^{\geq \star}X\in
    \CycSyn$. To see this, consider the map $F\colon\TR(\F^{\geq \star}X)\to \TR(\F^{\geq \star}X)$
    induced by the restriction maps $(\F^{\geq \star}X)^{C_{p^{n+1},\ev}}\to (\F^{\geq
    \star}X)^{C_{p^n,\ev}}$ and $(\F^{\geq \star}X^{\t C_{p,\ev}})^{C_{p^{n+1},\ev}}\to (\F^{\geq
    \star}X^{\t C_{p,\ev}})^{C_{p^n,\ev}}$. Then taking the equalizer of $F=1$ on $\TR(\F^{\geq
    \star}X)$ produces a fiber sequence \[\TR(\F^{\geq \star}X)^{F=1}\to \lim (\F^{\geq
    \star}X)^{C_{p^n,\ev}}\to \lim (\F^{\geq \star}X^{\t C_{p,\ev}})^{C_{p^n,\ev}}\] which under
    our assumptions on $\F^{\geq \star}X$ is then the same fiber sequence giving $\TC(\F^{\geq
    \star}X)$ by Lemma~\ref{lem: Tate then fixed points is Tate}.
\end{remark}

\begin{example}~\label{ex: discrete syncart objects}
    Let $(M_*,d, F,V)\in \mathrm{DCart}^\eta$, where $\mathrm{DCart}^\eta$ is the abelian category of $\eta$-deformed Cartier complexes of definition~\ref{def:eta_deformed}. Define $\F^{\geq \star}M\in
    \SynSp_{\bT_\ev}^\heartsuit$ by $\F^{\geq i}M\colon =M_i[i]$ with maps $\F^{\geq
    i+1}M\xrightarrow{0}\F^{\geq i}M$ and $\pi_0(\bT_{\ev})=\bZ[\eta,d]/(2\eta, d^2=\eta d)$-module
    structure given by $d$. From example~\ref{example: norm on pi_0 postnikov} we have that the
    maps $F,V:\F^{\geq \star}M_*\to \F^{\geq \star}M_*$ do determine maps $V\colon (\F^{\geq
    \star}M_*)_{C_{p,\ev}}\to \F^{\geq \star}M_*$ and $F\colon \F^{\geq \star}M_* \to \F^{\geq \star}M_*^{C_{p,\ev}}$ (by the first two relations) and factor the norm map (by the third and fourth relation). Thus this determines a Synthetic Cartier module.
\end{example}

In order to access the mapping spectra, it will be more convenient to work with a slightly different model of synthetic Cartier modules. 

\begin{lemma}
    There is a pullback square of $\infty$-categories 
    \[
    \begin{tikzcd}
        \syncart \ar[r] \ar[d] & \SynSp_{\bT_{\ev}}\times \SynSp_{\bT_{\ev}} \ar[d]\\
        \CycSyn^{Fr} \ar[r] & \SynSp_{\bT_{\ev}}^{\Delta^1},
    \end{tikzcd}
    \]
     where 
    \begin{enumerate}
        \item[{\em (i)}] $\CycSyn^{Fr}$ is the $\infty$-category given in Definition~\ref{defn: cycsyn with fr lift};
        \item[{\em (ii)}] the left vertical map is forgetting the $V$ map;
        \item[{\em (iii)}] the top horizontal map sends an object $(\F^{\geq \star}X,F,V,\sigma)$ to the pair $(\F^{\geq \star}X_{C_{p,\ev}}, \F^{\geq \star}X^{C_{p,\ev}}/F)$;\footnote{Here $\sigma$ denotes the 2-cell witnessing that $F\circ V\simeq \Nm_p$.}
        \item[{\em (iv)}] the bottom horizontal functor sends an object $(M,F:M\to M^{C_{p,\ev}})$ to the composition $M_{C_{p,\ev}}\xrightarrow{\Nm_p}M^{C_{p,\ev}}\to M^{C_{p,\ev}}/F$;
        \item[{\em (v)}] the right vertical map sends a pair $(X,Y)$ to the zero map $X\xrightarrow{0}Y$.
    \end{enumerate}
\end{lemma}

\begin{proof}
    The proof of~\cite[Lem.~3.7]{an1} applies here.
\end{proof}

As a consequence we get the following two corollaries. The proofs follow those
of~\cite[Prop.~3.8]{an1} and~\cite[Prop.~3.11]{an1}.

\begin{corollary}~\label{cor: fiber sequence for maps in syncart}
    The category $\syncart$ is stable presentable $\infty$-category, and for any two objects
    $(\F^{\geq \star}X,F_X,V_X,\sigma_X)$ and $(\F^{\geq \star}Y,F_Y,V_Y,\sigma_Y)$ in $\syncart$
    there is a fiber sequence of spectra \[\mathrm{Map}_{\syncart}(\F^{\geq \star}X,\F^{\geq
    \star}Y)\to \mathrm{Map}_{\CycSyn^{Fr}}(\F^{\geq \star}X,\F^{\geq \star}Y)\to
    \mathrm{Map}_{\SynSp_{\bT_{\ev}}}(\F^{\geq \star}X_{C_{p,\ev}},\fib(F_Y)).\] Moreover, $\mathrm{Map}_{\SynSp_{\bT_{\ev}}}(\F^{\geq
    \star}X_{C_{p,\ev}},\fib(F_Y))\simeq \mathrm{Map}_{\CycSyn^{Fr}}(\F^{\geq
    \star}X_{C_{p,\ev}},\F^{\geq \star}Y)$, where $\F^{\geq \star}X_{C_{p,\ev}}$ has the zero map as
    its Frobenius lift.
\end{corollary}

\begin{corollary}
    The category $\syncart$ is stable, presentable, and the forgetful functor $\syncart\to \SynSp_{\bT_{\ev}}$ preserves limits and colimits. 
\end{corollary}

We end this subsection with the construction of a functor $(-)/V:\syncart\to \CycSyn_p$ which we will later show is an equivalence.

\begin{construction}~\label{const: mod v functor}
    Consider the functor $\syncart\to \SynSp^{\Delta^1}_{\bT_{\ev}}$
    given by the composition \[\syncart\to \SynSp_{\bT_{\ev}}^{\Delta^2}\xrightarrow{p^*}
    \SynSp_{\bT_{\ev}}^{\Delta^1\times
    \Delta^1}\xrightarrow{\mathrm{cofib}}\SynSp_{\bT_{\ev}}^{\Delta^1}\] where the first functor is
    from the definition of $\syncart$, the second is the functor described diagrammatically as
    \[
    \begin{tikzcd}
       \F^{\geq \star}X \ar[dd, "f"] \ar[ddr, "h"] & & &\F^{\geq \star}X \ar[dd, "f"] \ar[r, "\id"]
        &\F^{\geq \star}X \ar[dd, "h"]\\
         & & \mapsto & & \\
         \F^{\geq \star}Y \ar[r, "g"] & \F^{\geq \star}Z & & \F^{\geq \star}Y\ar[r,"g"] & \F^{\geq \star}Z        
    \end{tikzcd}
    \]
    and the last map is taking vertical cofibers. Note that this composition sends a synthetic Cartier module $(X,F,V,\sigma)$ to the map $\F^{\geq \star}X/V\to \F^{\geq \star}X^{\t C_{p,\ev}}$. Postcomposing with the natural tranformation $\F^{\geq \star}X^{\t C_{p,\ev}}\to (\F^{\geq \star}X/V)^{\t C_{p,\ev}}$ then gives a functor $\syncart\to \CycSyn_p$ which sends $(\F^{\geq \star}X,F,V,\sigma)$ to $(\F^{\geq \star}X/V,\F^{\geq \star}X/V\to \F^{\geq \star}X^{\t C_{p,\ev}}\to (\F^{\geq \star}X/V)^{\t C_{p,\ev}})$.
\end{construction}

\subsection{$t$-structures on synthetic Cartier modules}

We will now construct $t$-structures on $\syncart$. Each of these, in turn, will produce a
$t$-structure on synthetic cyclotomic spectra once we have proven the equivalence between these two
categories.


\begin{construction}
    Let $\C$ be a $t$-structure on $\SynSp$ which satisfies $\mathbf{(\star)}$.
    Let $\syncart[,\geq 0]^\C\subseteq \syncart$ be the full subcategory spanned by objects $(X,F,V,\sigma)$
    where the underlying synthetic spectrum $X$ is in $\SynSp_{\geq 0}^\C$.
    Let $\syncart[,\leq 0]$ be the full subcategory of $\syncart$ spanned by objects
    $(X,F,V,\sigma)$ with $Y$ n $\SynSp_{\leq 0}^\C$.
\end{construction}

\begin{lemma}
    Let $\C$ be a $t$-structure on $\SynSp$ which satisfies $\mathbf{(\star)}$.
    \begin{enumerate}
        \item[{\em (a)}] The subcategories $(\syncart[,\geq 0]^\C,\syncart[,\leq 0]^\C)$ define an accessible $t$-structure on $\syncart$.
        \item[{\em (b)}] The forgetful functors $\syncart\to \SynSp_{\bT_{\ev}}\to \SynSp$ are $t$-exact.
        \item[{\em (c)}] The functor $(-)/V\colon\syncart\to \CycSyn_p$ of Construction~\ref{const: mod v functor} is right $t$-exact.
        \item[{\em (d)}] The $t$-structure on $\syncart$ is left complete and compatible with
            filtered colimits.
        \item[{\em (e)}] If $\SynSp^\C_{\geq 0}$ is right complete, then so is $\syncart[,\geq 0]^\C$.
    \end{enumerate}
\end{lemma}

\begin{proof}
    The proof follows that of~\cite[Prop.~3.15]{an1}.
    Following \cite[Proposition 3.15]{an1}, one shows that $\syncart[,\geq 0]$ is presentable and closed under colimits and extensions since
    the same is true for $\SynSp_{\bT_{\ev},\geq 0}$ and the forgetful functor preserves all limits
    and colimits. Thus there is some $\mathcal{C}\subseteq \syncart$ such that $(\syncart[,\geq
    0],\mathcal{C})$ is a $t$-structure. From Corollary~\ref{cor: fiber sequence for maps in
    syncart} we also have that $\syncart[,\leq 0]\subseteq \mathcal{C}$. To prove (a), it is then enough to show
    that for $(M,F_M,V_M, \sigma_M)\in \syncart$ there is a lift of the fiber sequence $\tau_{\geq
    0}M\to M\to \tau_{\leq -1}M$ in $\SynSp_{\bT_{\ev}}$ to $\syncart$. 

    To this end first note that $(\tau_{\geq 0}M)_{C_{p,\ev}}$ is connective since it is the tensor
    product of connective objects and $\tau$ is compatible with the monoidal structure. Thus the
    composition $(\tau_{\geq 0}M)_{C_{p,\ev}}\to M_{C_{p,\ev}}\xrightarrow{V}M$ factors through the
    connective cover $\tau_{\geq 0}M$. Similarly we get a factorization of the Frobenius through a
    map $\tau_{\geq 0}M\to \tau_{\geq 0}M^{C_{p,\ev}}\simeq \tau_{\geq 0}(\tau_{\geq
    0}M)^{C_{p,\ev}}$. The second equivalence comes from the fiber sequence \[(\tau_{\geq
    0}M)^{C_{p,\ev}}:=\mathrm{Map}_{\SynSp_{\bT_{\ev}}}(\rho(p)_*\bT_{\ev}, \tau_{\geq 0}M)\to
    \mathrm{Map}_{\SynSp_{\bT_{\ev}}}(\rho(p)_*\bT_{\ev}, M)\to
    \mathrm{Map}_{\SynSp_{\bT_{\ev}}}(\rho(p)_*\bT_{\ev}, \tau_{\leq -1}M)\] and the last term in
    this fiber sequence vanishes on connective covers since it is a map from a connective object to
    a $(-1)$-truncated object. As in \cite[Proposition 3.15]{an1}, the above maps, together with the map
    $\tau_{\geq 0}M\to M$ and the $C_{p,\ev}$ norm maps, gives a functor $(\partial\Delta^2)\times
    \Delta^1\to \SynSp_{\bT_{\ev}}$, and it is enough for us to show that this fills in to a
    functor $\Delta^2\times\Delta^1\to \SynSp_{\bT_{\ev}}$. This is done inductively, where the
    base case $\partial \Delta^2\times [1]\in \partial\Delta^2\times \Delta^1$ is filled in by
    $\sigma_M$, and the other 2-simplicies are filled in via decomposing $\partial\Delta^2\times
    \Delta^1\cup \Delta^2\times [1]$ as three tetrahedra and using that fact that $\SynSp_{\bT_{\ev}}$
    is an $\infty$-category.

    The functors $\syncart\to \SynSp_{\bT_{\ev}}$ and $\syncart\to \SynSp$ are then $t$-exact by
    definition. We also have that for $(M, F, V, \sigma)\in \syncart[,\geq 0]$ that
    $M_{C_{p,\ev}}\in \SynSp_{\bT_{\ev},\geq 0}$, and so taking the cofiber of two connective
    objects gives that $M/V\in \SynSp_{\bT_{\ev}, \geq 0}$. Finally the fact that $\syncart[,\geq
    0]^\C$ is
    compatible with filtered colimits and is left or right complete if $\tau$ is follows from the fact
    that $\syncart\to \SynSp_{\bT_{\ev}}$ is conservative and commutes with limits and colimits.
\end{proof}

\subsection{Synthetic lifts of the cyclotomic $t$-structure}

Now that we have $t$-structures on $\syncart$, to study the $t$-structures on $\CycSyn$ it is
enough to identify $\CycSyn$ as the full subcategory of $V$-complete objects in $\syncart$. We
begin by lifting a result of Krause and Nikolaus from \cite{krause-nikolaus} to the
synthetic setting. In fact, the proof they give in \cite[Proposition 10.3]{krause-nikolaus} also
works in our setting.

\begin{lemma}~\label{lem: tr is right adjoint for cycsyn with lifted frob}
    The functor $\mathrm{TR}\colon\CycSyn\to \CycSyn^{Fr}$ is right adjoint to the forgetful functor.
\end{lemma}
\begin{proof}
    For $(\F^{\geq \star}Y,F_Y)\in \CycSyn^{Fr}$ by definition we have that maps from $Y$ to
    $\mathrm{TR}(X)$ is given by \[\mathrm{Eq}\left(\mathrm{Map}_{\SynSp_{\bT_{\ev}}}(\F^{\geq
    \star}Y,\mathrm{TR}(\F^{\geq
    \star}X))\substack{\xrightarrow[\hphantom{F_Y^*(-)^{C_{p,\ev}}}]{F_{\mathrm{TR}(\F^{\geq
    \star}X),*}}\\
    \xrightarrow[F_Y^*(-)^{C_{p,\ev}}]{}}
    \mathrm{Map}_{\SynSp_{\bT_{\ev}}}(\F^{\geq \star}Y,\mathrm{TR}(\F^{\geq \star}X)^{C_{p,\ev}})\right)\] and by definition $\mathrm{TR}(\F^{\geq \star}X)$ is given by an equilizer itself. In particular we get equivalences 
    \begin{align*}
        &\mathrm{Map}_{\SynSp_{\bT_{\ev}}}(\F^{\geq \star}Y,\mathrm{TR}(\F^{\geq \star}X))\simeq\\
        & \mathrm{Eq}\left(\prod_{n\geq 0}\mathrm{Map}_{\SynSp_{\bT_{\ev}}}(\F^{\geq \star}Y, \F^{\geq \star}X^{C_{p^n,\ev}})\substack{\xrightarrow[\hphantom{\prod (\varphi_{\F^{\geq \star}X}^{C_{p^n,\ev}})_*}]{\prod can_*}\\ \xrightarrow[\prod (\varphi_{\F^{\geq \star}X}^{C_{p^n,\ev}})_*]{}}\prod_{n\geq 0}\mathrm{Map}_{\SynSp_{\bT_{\ev}}}(\F^{\geq \star}Y, (\F^{\geq \star}X^{\t C_{p,\ev}})^{C_{p^{n-1},\ev}})\right)
    \end{align*}
    and the same equivalence with maps to $\mathrm{TR}(\F^{\geq \star}X)^{C_{p,\ev}}$ with an all the $n$s replaced with $n+1$. 

    Since we are taking the equilizer of equilizers and limits commute, we may first find the
    equilizer \[\mathrm{Eq}\left(\prod_{n\geq 0}\mathrm{Map}_{\SynSp_{\bT_{\ev}}}(\F^{\geq \star}Y,
    \F^{\geq \star}X^{C_{p^n,\ev}})\rightrightarrows\prod_{n\geq
    0}\mathrm{Map}_{\SynSp_{\bT_{\ev}}}(\F^{\geq \star}Y,\F^{\geq
    \star}X^{C_{p^{n+1},\ev}})\right)\] and \[\mathrm{Eq}\left(\prod_{n\geq
    0}\mathrm{Map}_{\SynSp_{\bT_{\ev}}}(\F^{\geq \star}Y, (\F^{\geq \star}X^{\t
    C_{p,\ev}})^{C_{p^n,\ev}})\rightrightarrows\prod_{n\geq
    0}\mathrm{Map}_{\SynSp_{\bT_{\ev}}}(\F^{\geq \star}Y,(\F^{\geq \star}X^{\t
    C_{p,\ev}})^{C_{p^{n+1},\ev}})\right)\] where the top map is given by forgetting the bottom
    factor and the bottom map is $(F_{\F^{\geq \star}Y}(-)^{C_{p,\ev}})_*$. These are given by
    $\mathrm{Map}_{\SynSp_{\bT_{\ev}}}(\F^{\geq \star}Y, \F^{\geq \star}X)$ and
    $\mathrm{Map}_{\SynSp_{\bT_{\ev}}}(\F^{\geq \star}Y,\F^{\geq \star}X^{\t C_{p,\ev}})$,
    respectively. Putting this together then gives
    \begin{align*}
        &\mathrm{Map}_{\CycSyn_p^{Fr}}(\F^{\geq \star}Y, \mathrm{TR}(\F^{\geq \star}X))\\
        &\simeq \mathrm{Eq}\left(\mathrm{Map}_{\SynSp_{\bT_{\ev}}}(\F^{\geq \star}Y, \F^{\geq
        \star}X)\rightrightarrows\mathrm{Map}_{\SynSp_{\bT_{\ev}}}(\F^{\geq \star}Y, \F^{\geq \star}X^{\t C_{p,\ev}})\right)\\
        &\simeq \mathrm{Map}_{\CycSyn}(\F^{\geq \star}Y,\F^{\geq \star}X)
    \end{align*}
    as desired.
\end{proof}

From the lemma, we deduce the following.

\begin{corollary}~\label{cor: TR/V=id}
    The functor $\mathrm{TR}(-)\colon\CycSyn\to \syncart$ is right adjoint to
    $(-)/V\colon\syncart\to \CycSyn$. Furthermore, the counit of the adjunction
    \[\mathrm{TR}(\F^{\geq \star}X)/V\to \F^{\geq \star}X\] is an equivalence when the underlying
    synthetic spectrum $\F^{\geq \star}X$ is $\C$-bounded below with respect to a $t$-structure
    satisfying $\mathbf{(\star)}$.
\end{corollary}
\begin{proof}
    The fact that $\mathrm{TR}\colon\CycSyn\to \syncart$ is right adjoint to $(-)/V\colon\syncart\to \CycSyn$
    follows from the fiber sequence of Corollary~\ref{cor: fiber sequence for maps in syncart} and Lemma~\ref{lem: tr is right adjoint for cycsyn with lifted frob}.
    Note that we can describe the counit of this adjunction explicitly: We have a projection map \[\pi\colon \mathrm{TR}(\F^{\geq \star}X)\to \prod_{n\geq 0}\F^{\geq \star}X^{C_{p^n,\ev}}\to \F^{\geq \star}X\] and by definition the composition $\pi\circ V:\mathrm{TR}(\F^{\geq \star}X)_{C_{p,\ev}}\to \F^{\geq \star}X$ is nullhomotopic. The induced map $\mathrm{TR}(\F^{\geq \star}X)/V\to \F^{\geq \star}X$ is the counit and as in \cite{an1} extending this to the commutative diagram \[
    \begin{tikzcd}
        \mathrm{TR}(\F^{\geq \star}X)_{C_{p,\ev}} \ar[d] \ar[r,"V"] & \mathrm{TR}(\F^{\geq \star}X) \ar[d,"F"] \ar[r, "\pi"] & \F^{\geq \star}X \ar[d,"\varphi_{\F^{\geq \star}X}"]\\
        \mathrm{TR}(\F^{\geq \star}X)_{C_{p,\ev}} \ar[r, "\Nm_p"] & \mathrm{TR}(\F^{\geq \star}X)^{C_{p,\ev}} \ar[r] & \F^{\geq \star}X^{\t C_{p,\ev}}
    \end{tikzcd}
    \]
    shows that the counit is an equivalence if and only if the map $\pi^{\t C_{p,\ev}}\colon\mathrm{TR}(\F^{\geq \star}X)^{\t C_{p,\ev}}\to \F^{\geq \star}X^{\t C_{p,\ev}}$ is an equivalence.

    Now suppose that $\F^{\geq \star}X$ is an object of $\CycSyn$ such that the underlying
    synthetic spectrum is bounded below with respect to $\C$. Note that we can define
    $\mathrm{TR}^n(\F^{\geq \star}X)$ as the $\bT_{\ev}$-equivariant spectrum
    \[\mathrm{TR}^n(\F^{\geq \star}X):= \mathrm{Eq}\left(\prod_{0\leq k\leq n}\F^{\geq
    \star}X^{C_{p^k,\ev}}\rightrightarrows\prod_{0\leq k\leq n}\F^{\geq \star}(X^{\t
    C_{p,\ev}})^{C_{p^{k-1},\ev}}\right)\] and doing so gives us maps
    $R\colon\mathrm{TR}^{n+1}(\F^{\geq \star}X)\to \mathrm{TR}^n(\F^{\geq \star}X)$ given by
    projection onto the first $n$-factors in each of the products. It then follows that
    $\mathrm{TR}(\F^{\geq \star}X)\simeq \lim \mathrm{TR}^n(\F^{\geq \star}X)$, that the map
    $\pi\colon\mathrm{TR}(\F^{\geq \star}X)\to \F^{\geq \star}X$ agrees with the map $\mathrm{TR}(\F^{\geq \star}X)\to \mathrm{TR}^0(\F^{\geq \star}X)$, and that the fiber of the map $\mathrm{TR}^{n+1}(\F^{\geq \star}X)\to \mathrm{TR}^n(\F^{\geq \star}X)$ can be identified with $\F^{\geq \star}X_{C_{p^{n+1},\ev}}$.

    Suppose now that $N$ is such that $\F^{\geq \star}X\simeq \tau_{\geq N}\F^{\geq \star}X$ which
    must exist by assumption on $\F^{\geq \star}X$ being bounded below with respect to $\C$.
    From compatibility with the monoidal structure it follows that $\F^{\geq
    \star}X_{C_{p^{n+1},\ev}}\simeq \tau_{\geq N}\F^{\geq \star}X_{C_{p^{n+1},\ev}}$ for all $n$,
    and so inductively $\mathrm{TR}^n(\F^{\geq \star}X)\simeq \tau_{\geq N} \mathrm{TR}^n(\F^{\geq
    \star}X)$ for all $n$. Applying $(-)^{\t C_{p,\ev}}$ to the fiber sequence $\F^{\geq
    \star}X_{C_{p^{n+1},\ev}}\to \mathrm{TR}^{n+1}(\F^{\geq \star}X)\to \mathrm{TR}^n(\F^{\geq
    \star}X)$ then inductively shows that the map $\mathrm{TR}^n(\F^{\geq \star}X)^{\t
    C_{p,\ev}}\to \F^{\geq \star}X^{\t C_{p,\ev}}$ is an equivalence since $(\F^{\geq
    \star}X_{C_{p^{n+1},\ev}})^{\t C_{p,\ev}}\simeq 0$ by Lemma~\ref{lem: Tate orbit lemma}. The
    result then follows from the fact that $(-)^{\t C_{p,\ev}}$ commutes with limits of uniformly
    $\C$-bounded below objects by Lemma~\ref{lem: orbit/fixedpoints/tate continuity}.
\end{proof}

Since the counit of the adjunction only induces an equivalence for $\C$-bounded below objects we will need to restrict our attention to these objects.

\begin{definition}
    Let $\C$ be a $t$-structure on $\SynSp$ which satisfies $\mathbf{(\star)}$
    and let $\CycSyn^{-}$ denote the full subcategory of $\CycSyn$ spanned by bounded below
    objects in the associated $t$-structure on $\CycSyn$.
    Similarly denote $\syncart^{-}$ the full subcategory of $\syncart$ spanned by bounded below objects.
\end{definition}

On these subcategories we have a nice identification of $\mathrm{TR}(\F^{\geq \star}X)/V^n$.

\begin{lemma}~\label{lem: identification of TR/V^n}
    Let $\F^{\geq \star}X\in \CycSyn^{-}$. Then the natural map $\mathrm{TR}(\F^{\geq \star}X)/V^n\to \mathrm{TR}^n(\F^{\geq \star}X)$ is an equivalence.
\end{lemma}
\begin{proof}
    The case $n=0$ was handled in the proof of Corollary~\ref{cor: TR/V=id}, so suppose that this map is an equivalence for $n$. Then we have a cofiber sequence \[(\F^{\geq \star}X)_{C_{p^n,\ev}}\to \mathrm{TR}(\F^{\geq \star}X)/V^{n+1}\to \mathrm{TR}^n(\F^{\geq \star}X)\] coming from the factorization $V^{n+1}=V^n\circ V_{C_{p^{n},\ev}}$ and using the base case and the inductive hypothesis to identify the first and last terms. This fiber sequence maps to the cofiber sequence \[(\F^{\geq \star}X)_{C_{p^n,\ev}}\to \mathrm{TR}^{n+1}(\F^{\geq \star}X)\to \mathrm{TR}^n(\F^{\geq \star}X)\] and so the result follows.
\end{proof}

\begin{lemma}~\label{lem: TR is t-exact}
    The functor $\mathrm{TR}:\CycSyn^{-}\to \syncart$ is fully faithful and $t$-exact. The essential image is the full subcategory of bounded below $V$-complete objects.
\end{lemma}
\begin{proof}
    The fact that $\mathrm{TR}:\CycSyn^{-}\to \syncart$ is fully faithful follows from Corollary~\ref{cor: TR/V=id}. Similarly $\mathrm{TR}$ is left $t$-exact since $(-)/V$ is right $t$-exact.

    To see that $\mathrm{TR}$ is right $t$-exact, note first that for a connective $\F^{\geq \star}X$, each $\mathrm{TR}^n(\F^{\geq \star}X)$ will inductively be connective.
    Since $\C$ is compatible with countable products we then have that the limit of connective
    objects will be $(-1)$-connective. Furthermore since all terms in the fiber sequences $\F^{\geq
    \star}X_{C_{p^{n+1},\ev}}\to \mathrm{TR}^{n+1}(\F^{\geq \star}X)\to \mathrm{TR}^n(\F^{\geq
    \star}X)$ are connective we have that $p_{n+1}:\pi_0^{\C}(\mathrm{TR}^{n+1}(\F^{\geq
    \star}X))\to \pi_0^{\C}(\mathrm{TR}^{n}(\F^{\geq \star}X))$ is an epimorphism for all $n$. Thus
    $\prod_{n}\pi_0^\C(\mathrm{TR}^{n}(\F^{\geq \star}X))\xrightarrow{id-p_n}\prod_{n}\pi_0^\C(\mathrm{TR}^{n}(\F^{\geq \star}X))$ is an epimorphism.
    From this and the fiber sequence $\mathrm{TR}(\F^{\geq \star}X)\to
    \prod_{n}\mathrm{TR}^{n}(\F^{\geq \star}X)\to \prod_{n}\mathrm{TR}^{n}(\F^{\geq \star}X)$ we
    have that $\mathrm{TR}(\F^{\geq \star}X)$ is connective as desired.

    It remains to identify the essential image. By $t$-exactness we get that $\mathrm{TR}$ lands in
    the full subcategory $\syncart^{-}$. It is now enough to show that for $\F^{\geq \star}M\in
    \syncart^{-}$ the unit map $\F^{\geq \star}M\to \mathrm{TR}(\F^{\geq \star}M/V)$ identifies the
    target as the $V$-completion of the source. We have from Lemma~\ref{lem: identification of
    TR/V^n} that the target is $V$-adically complete, and so it is enough to show that $M/V^n\to
    \mathrm{TR}^n(\F^{\geq \star}M/V)$ is an equivalence for all $n$. This follows inductively by
    looking at the fiber sequences \[(\F^{\geq \star}M/V)_{C_{p^n,\ev}}\to
    \mathrm{TR}^{n+1}(\F^{\geq \star}M/V)\to \mathrm{TR}^n(\F^{\geq \star}M/V)\] and \[(\F^{\geq
    \star}M/V)_{C_{p^n,\ev}}\to \F^{\geq \star}M/V^{n+1}\to \F^{\geq \star}M/V^n\] where the second
    fiber sequence comes from writing $V^{n+1}=V^n\circ V_{C_{p^n,\ev}}$.
\end{proof}

\subsection{Identification of the heart}
We will now restrict our attention to the Postnikov $t$-structure of Construction~\ref{const: Postnikov t-structure on synthetic spectra}. The first part of this section will be devoted to identifying the heart of this $t$-structure in terms of $\eta$-deformed Cartier complexes.

\begin{definition}~\label{def:eta_deformed}
    Let $\mathrm{DCart}^{\eta}$ denote the abelian category of tuples $(M_*, d, F,V)$ where $M_*$
    is a graded $\bZ[\eta]/2\eta$-module with $|\eta|=1$, $d\colon M_* \to M_{* +1}$ a map of
    $\bZ[\eta]/2\eta$-modules such that $d^2=\eta d$, and $F,V\colon M_*\to M_*$ maps such that \[Vd=pdV,\textrm{\hspace{2em}} dF=pFd,\textrm{\hspace{2em}} FdV=d+\eta\hspace{2em} FV=p\] and $F$ and $V$ commute with $\eta$.
\end{definition}

\begin{lemma}
    Let $\syncart[,\geq 0]^\P$ be the $t$-structure on $\syncart$ associated to the Postnikov
    $t$-structure on $\SynSp$. The heart of this $t$-structure is equivalent to $\mathrm{DCart}^\eta$.
\end{lemma}

\begin{proof}
    Since the functor $\syncart\to \SynSp_{\bT_{\ev}}$ is $t$-exact we have that the heart consists
    of $(M_* ,d)$ with extra structure maps $F,V\colon M_* \to M_*$. The maps $F$ and $V$ satisfy the indicated relations by the same argument as in \cite[Lemma 3.33]{an1}. Thus there is a fully faithful functor $\syncart^{\heartsuit}\hookrightarrow \mathrm{DCart}^\eta$ and all that remains is to show essential surjectivity. From example~\ref{ex: discrete syncart objects} for any element in $(M_*,d, F,V)\in \mathrm{DCart}^{\eta, \wedge}_V$ we can build a synthetic Cartier module $(F^{\geq \star}M_*, F, V)$. 
    Essential surjectivity then follows from the fact that the forgetful functor $\syncart\to \SynSp_{\bT_{\ev}}$ is conservative and $t$-exact so we have that $(\F^{\geq \star}M_*, F,V)$ is in the heart as desired. 
\end{proof}

Combining this result with the identification of bounded below synthetic cyclotomic spectra with bounded below $V$-complete synthetic Cartier modules, we get that the heart of the cyclotomic $t$-structure is given by $\mathrm{DCart}^{\eta,\wedge}_V$ of $V$-complete $\eta$-deformed Cartier complexes. This $V$-completion is in terms of the ambient category a priori, not in terms of the $V$-completion of objects in $\mathrm{DCart}^\eta$. Nevertheless it turns out that these two notions agree.

\begin{lemma}
    Let $(\F^{\geq \star}M, F, V)\in \syncart^{\heartsuit}$. Then $\F^{\geq \star}M$ is $V$-complete if and only if it is derived $V$-complete as an element of $\mathrm{DCart}^\eta$.
\end{lemma}
\begin{proof}
    First suppose that $(\F^{\geq \star}M,F,V)$ is $V$-complete as a synthetic Cartier module. Similar to example~\ref{example: norm on pi_0 postnikov} we find that $\pi_0^p(\F^{\geq \star}M_{C_{p^n,\ev}})\simeq \F^{\geq \star}M[d/p^n]$. By assumption \[\lim(\ldots \xrightarrow{V_{C_{p^2,\ev}}}\F^{\geq \star}M_{C_{p^2,\ev}}\xrightarrow{V_{C_{p,\ev}}}\F^{\geq \star}M_{C_{p,\ev}}\xrightarrow{V}\F^{\geq \star}M)\simeq 0\] and so the map $\prod_{n\geq 0}\F^{\geq \star}M_{C_{p^n,\ev}}\to \prod_{n\geq 0}\F^{\geq \star}M_{C_{p^n,\ev}}$ is an equivalence. Applying $\pi_0^p$ then again induces an equivalence and so the system \[\lim (\ldots \to \F^{\geq \star}M[d/p^3]\to \F^{\geq \star}M[d/p^2]\to \F^{\geq \star}M[d/p]\to \F^{\geq \star}M)\simeq 0\] as well. The result follows from the fact that the map $V:\F^{\geq \star}M\to \F^{\geq \star}M$ factors through the map $\F^{\geq \star}M[d/p]$ by definition.

    Conversely suppose that $(\F^{\geq \star}M,F,V)$ is $V$ complete as an $\eta$-deformed Cartier complex. As in the proof of \cite[Lemma 3.25]{an1}, the inverse limit \[\lim(\ldots \F^{\geq \star}M_{C_{p^2,\ev}}\to \F^{\geq \star}M_{C_{p,\ev}}\to \F^{\geq \star}M)\] is given by the diagonal of the diagram 
    \[\begin{tikzcd}
                 & \vdots \arrow[d]                                       & \vdots \arrow[d]                                   & \vdots \arrow[d]           \\
\cdots \arrow[r] & {\F^{\geq \star}M\otimes_{\bT_{\ev}} \rho(p^2)_*\bT_{\ev}} \arrow[d, "{V_{C_{p^2}}}"] \arrow[r] & \F^{\geq \star}M\otimes_{\bT_{\ev}}\rho(p)_*\bT_{\ev} \arrow[d, "{V_{C_{p,\ev}}}"] \arrow[r]    & \F^{\geq \star}M \arrow[d, "V"] \\
\cdots \arrow[r] & {\F^{\geq \star}M\otimes_{\bT_{\ev}}}\rho(p^2)_*\bT_{\ev} \arrow[d, "{V_{C_{p^2,\ev}}}"] \arrow[r] & {\F^{\geq \star}M\otimes_{\bT_{\ev}}\rho(p)_*\bT_{\ev}} \arrow[d, "{V_{C_{p,\ev}}}"] \arrow[r] & \F^{\geq \star}M \arrow[d, "V"] \\
\cdots \arrow[r] & {\F^{\geq \star}M\otimes_{\bT_{\ev}}\rho(p^2)_*\bT_{\ev}} \arrow[r]                         & {\F^{\geq \star}M\otimes_{\bT_{\ev}}\rho(p)_*\bT_{\ev}} \arrow[r]                       & \F^{\geq \star}M               
\end{tikzcd}\] where the maps $\rho(p^n)_*\bT_{\ev}\to \rho(p^{n-1})_*\bT_{\ev}$ are the even filtration applied to the canonical projection $S^1/C_{p^n}\to S^1/C_{p^{n-1}}$. The vertical inverse limits vanish, the rightmost one by assumption and all the others by Lemma~\ref{lem: orbit/fixedpoints/tate continuity}(1b).
\end{proof}

Putting this all together proves the claimed identification of the heart.

\begin{theorem}
    Let $\CycSyn^\P_{\geq 0}$ be the $t$-structure on $\CycSyn$ induced by the Postnikov
    $t$-structure via Theorem~\ref{thm: construction of cyclotomic $t$-structure}. The heart
    of this $t$-structure is equivalent to the abelian category $\CycSyn^{\heartsuit}\simeq
    \mathrm{DCart}^{\eta,\wedge}_V$ of derived $V$-complete $\eta$-deformed Cartier complexes.
\end{theorem}

\subsection{Bounded objects and the Segal conjecture}

We conclude this section with an analogue of \cite[Lemma 2.25]{an1}.
Unlike the previous subsection we will return to working with $C$ a $t$-structure on $\SynSp$ satisfying condition $\mathbf{(\star)}$ from definition~\ref{def:t_star}.
The result \cite[Lemma 2.25]{an1} states that, in the unfiltered case, an object $M\in
\mathrm{CycSp}_{\leq d}$ has cyclotomic Frobenius $\varphi_p\colon M\to M^{\t C_p}$ which is $d$-truncated.

\begin{theorem}~\label{thm: Segal conjecture full strength}
    Let $\F^{\geq \star}M\in \CycSyn_{\leq d}^\C$ for some $d\in \bZ$. Then the fiber of the map
    $\F^{\geq \star}M\to \F^{\geq \star}M^{\t C_{p,\ev}}$ is in $\SynSp_{\leq d}^\C$.
\end{theorem}

\begin{proof}
    Let $\F^{\geq \star}X\in \SynSp_{\geq d+1}^\C$. Consider the object $\F^{\geq
    \star}X\otimes_{\bS_{\ev}}\bT_{\ev}$ as an object of $\CycSyn_{\geq d+1}^\C$ by taking the
    induced $\bT_{\ev}$-module structure and filtered cyclotomic Frobenius given by the zero map
    $\F^{\geq \star}X\otimes_{\bS_{\ev}}\bT_{\ev}\xrightarrow{0} (\F^{\geq
    \star}X\otimes_{\bS_{\ev}}\bT_{\ev})^{\t C_{p,\ev}}$. From \cite[Proposition
    II.1.5(ii)]{nikolaus-scholze} we then have a fiber sequence \[
    \begin{tikzcd}
    \mathrm{Map}_{\CycSyn_p}(\F^{\geq \star}X\otimes_{\bS_{\ev}}\bT_{\ev}, \F^{\geq \star}M)\ar[d] & &\\
    \mathrm{Map}_{\SynSp_{\bT_{\ev}}}(\F^{\geq \star}X\otimes_{\bS_{\MU}}\bT_{\ev}, \F^{\geq \star}M)\ar[rr, "(\varphi_{\F^{\geq \star}M})_*"] & &\mathrm{Map}_{\SynSp_{\bT_{\ev}}}(\F^{\geq \star}X\otimes_{\bS_{\MU}}\bT_{\ev}, \F^{\geq \star}M^{tC_{p,\ev}})
    \end{tikzcd}
    \]
    which identifies $\mathrm{Map}_{\CycSyn_p}(\F^{\geq \star}X\otimes_{\bS_{\ev}}\bT_{\ev},
    \F^{\geq \star}M)\simeq \mathrm{Map}_{\SynSp}(\F^{\geq \star}X, \mathrm{fib}(\F^{\geq
    \star}M\xrightarrow{\varphi_{\F^{\geq \star}M}}\F^{\geq \star}M^{\t C_{p,\ev}}))$. The
    (underlying space of the) left hand term is contractible by assumption, so therefore the right
    hand term is as well. Since this is true for all $\F^{\geq \star}X\in \SynSp_{\geq d+1}^\C$, it
    follows that $\mathrm{fib}(\F^{\geq \star}M\xrightarrow{\varphi_{\F^{\geq
    \star}M}}\F^{\geq*}M^{\t C_{p,\ev}})\in \SynSp_{\leq d}^\C$ as desired.
\end{proof}

\begin{example}
    Consider $A$ a smooth $\bF_p$-algebra of dimension $d$. By~\cite[Cor~1.2]{Darrell_Riggenbach},
    we have that $$\mathrm{TR}(\F^{\geq \star}_\ev \THH(A;\bZ_p))\simeq \tau_{\geq
    \star}\mathrm{TR}(A),$$ so $\F^{\geq \star}\THH(A)\in \CycSyn^\N_{\leq d}$ in the neutral
    $t$-structure. On the other hand it will not be truncated in the Postnikov $t$-structure, since
    any truncated object in the Postnikov $t$-structure is the filtration on the zero object.
\end{example}

\subsection{\texorpdfstring{Lifts of Dieudonn{\'e} modules over smooth $\bF_p$-algebras}{Lifts of Dieudonn{\'e} modules over smooth algebras}}

Recall the following standard lemma.

\begin{lemma}~\label{lem: hearts and modules}
    Let $\mathcal{C}$ be a symmetric monoidal $\infty$-category, $\tau=(\mathcal{C}_{\geq
    0},\mathcal{C}_{\leq 0})$ a $t$-structure compatible with the monoidal structure, $R$ a
    connective $\bE_\infty$ ring, and $M$ and $N$ $\pi_0^\tau(R)$-modules in
    $\mathcal{C}^\heartsuit$. Then, the natural map $\mathrm{Map}_R(M,N)\leftarrow
    \mathrm{Hom}_{\pi_0^\tau(R)}(M,N)$ is an equivalence.
\end{lemma}

Now, let $A$ be a smooth $k$-algebra, $k$ a perfect $\bF_p$-algebra. Note that by
Remark~\ref{rem:bounds} $\F^{\geq \star}_{\ev}\THH(A;\bZ_p)$, and in fact $\F^{\geq
\star}_{\ev}\THH(R)$ for any $R\in \mathrm{CQSyn}$, will be connective in the Postnikov cyclotomic
$t$-structure. In particular by Lemma~\ref{lem: hearts and modules} the functor \[\Mod_{\pi_0^{\cyc,
\P}(\F^{\geq \star}_{\ev}\THH(A:\bZ_p))}(\CycSyn^\heartsuit)\hookrightarrow
\CycSyn_{\F^{\geq \star}_{\ev}\THH(A;\bZ_p)}\] induced by the ring map $\F^{\geq
\star}_\ev\THH(A;\bZ_p)\to \pi_0^{\cyc,\P}(\F^{\geq \star}_\ev\THH(A;\bZ_p))$ is fully faithful.
Thus, in order to prove Theorem~\ref{thm: thmA}, it is enough to show that formal $p$-divisible groups
over $A$ embed fully faithfully into the left-hand category.

In order to do this we will need to identify the left hand category. Before stating the identification we will first recall some notation.

\begin{notation}
    For $A$ an $\bF_p$-algebra, denote by $\W\Omega_A$ the de Rham--Witt cochain complex of Illusie~\cite{illusie-derham-witt}.
    Denote by $\W\Omega^{\geq i}_A$ the $i^{th}$ stage of the Hodge filtration.
\end{notation}

\begin{lemma}
    Let $A$ be a smooth algebra over a perfect $\bF_p$-algebra $k$. Then there is a natural
    identification \[\pi_i^{\cyc,\P}(\F^{\geq \star}_\ev\THH(A;\bZ_p))\cong\W\Omega^{*}_A(-i)\] for all $i\geq 0$ in
    $\mathrm{DCart}^{\eta,\wedge}_V$, where the $F$, $V$, and $d$ maps are the usual ones, and
    $\eta=0$.
\end{lemma}

\begin{proof}
    The identification of $\CycSyn^\heartsuit\cong \mathrm{DCart}^{\eta, \wedge}_V$ goes through
    the $t$-exact functor $\mathrm{TR}\colon\CycSyn\to \syncart$, and so $\pi_i^{\cyc,\P}(\F^{\geq
    \star}_\ev\THH(A;\bZ_p))\cong \pi_i^{\P}(\mathrm{TR}(\F^{\geq \star}_\ev
    \THH(A;\bZ_p)))$. The result then follows from \cite[Corollary 1.2]{Darrell_Riggenbach}.
\end{proof}

\begin{proof}[Proof of Theorem~\ref{thm: thmA}]
    First recall that by \cite[Theorem 3.1]{DeJong_Messing_Deiudonne} the Deiudonn{\'e} functor
    $\mathbf{D}\colon\mathrm{FBT}_A^{op}\to C_A$ is fully faithful, where $C_A$ is the category of
    Deiudonn{\'e} crystals of finite presentation over $A$ together with an $F$ and $V$ operators,
    and lands in those crystals which are $V$-adically complete. Then \cite[Theorem
    1.1]{Bloch_Crystals} produces an embedding from $C_A\hookrightarrow
    \Mod_{\W\Omega^\bullet_A}(\mathrm{DCart}^\eta)$ which sends $V$-complete objects to $V$-complete
    objects. The result follows.
\end{proof}

\begin{conjecture}
    Let $k$ be a perfect $\bF_p$-algebra and let
    $A$ be a smooth $k$-algebra. Let \[\bD_{\CycSyn}:
    \mathrm{FBT}_A^{op}\hookrightarrow \CycSyn_{\F^{\geq \star}_\ev\THH(A;\bZ_p)}\] be the fully
    faithful embedding of Theorem~\ref{thm: thmA}. Note that by taking $\mathrm{gr}^0(-)^{\bT_\ev}$
    of a $\F^{\geq \star}_\ev \THH(A;\bZ_p)$-module in $\CycSyn$ we get a $\mathcal{N}^{\geq
    \star}\prism_{A}$-module with a filtration given by the $v$-adic filtration. Additionally the
    cyclotomic synthetic Frobenius induces a Frobenius map on this module. Then, considered with
    this extra structure, there is an equivalence
    \[\mathrm{gr}^0(\bD_\CycSyn)^{\bT\ev}\simeq \mathrm{R\Gamma}(A; \bD_\prism)\] where $\bD_\prism$ is the prismatic
    Dieudonn{\'e} functor of \cite{anschutz-lebras}.
\end{conjecture}

\section{A filtered Beilinson fiber square}

This section is devoted to proving a filtered version of the Beilinson fiber square of \cite{ammn}
and a filtered version of a forthcoming result of  Devalapurkar and Raksit, which gives a height
$1$ analogue. The utility of our approach is that, after the work we have put in in the previous
sections, we can mimic the proofs of these results given in cyclotomic spectra almost verbatim to
get the filtered result.

\subsection{Filtered $\TC/p$ and colimits}

We begin by proving a filtered version of~\cite[Thm.~G]{CMM}. The proof goes through mostly the
same as in \cite{CMM}, we include a sketch for convenience. Before proving this result we will
establish some preliminary technical results.

\begin{lemma}
    If $\F^{\geq \star}X$ is a cyclotomic synthetic spectrum which is $n$-connective in the
    neutral $t$-structure, then $\mathrm{TC}(\F^{\geq \star}X)$ is $(n-1)$-connective in the
    neutral $t$-structure.
\end{lemma}

\begin{proof}
    By Remark~\ref{rem: tc=trf=1}, there is a functorial fiber sequence \[\mathrm{TC}(\F^{\geq \star}X)\to \mathrm{TR}(\F^{\geq
    \star}X)\to \mathrm{TR}(\F^{\geq \star}X)\] and by Lemma~\ref{lem: TR is t-exact} we have that
    $\mathrm{TR}(\F^{\geq \star}X)$ is $n$-connective.
\end{proof}

We now highlight another interesting consequence of the same proof as the previous lemma.
Compare to~\cite[Thm.~5.1(1)]{ammn} in the discrete $p$-quasisyntomic case.

\begin{proposition}
    Let $R$ be a chromatically quasisyntomic ring. Then $\TC_\ev(R)$ is $(-1)$-connective in the
    Postnikov $t$-structure. In particular, $\bZ_p(i)^{\syn}(R)$ is concentrated in
    cohomological degrees $\leq i+1$.
\end{proposition}

\begin{proof}
    By Remark~\ref{rem:bounds}, $\F^{\geq \star}\THH(R)$ is connective in the Postnikov
    $t$-structure. Then since the Postnikov $t$-structure satisfies condition $\mathbf{(\star)}$ of
    definition~\ref{def:t_star}, we have that $\TR(\F^{\geq \star}_\ev\THH(R))$ is also connective
    in the Postnikov $t$-structure by Lemma~\ref{lem: TR is t-exact}. The result follows from the
    fiber sequence of Remark~\ref{rem: tc=trf=1}.
\end{proof}

The synthetic spectrum $\ins^0\bF_p$, constructed via Example~\ref{ex:inserted_cyclotomic}, will play a particularly large role in the arguments to follow.

\begin{construction}~\label{const: geometric realization of finite type synthetic spectra}
    Let $\F^{\geq \star}X$ be a complete synthetic spectrum such that \[\bigoplus_{n\in \bZ}\pi_i(\mathrm{gr}^nX)\] is finitely generated for all $i\in \bZ$.
    Further assume that $\F^{\geq \star}X$ is connective in the neutral $t$-structure and that $\F^{\geq \star}X$ is constant for $*\leq 0$.
    Then choosing a fixed finite set of these generators, $\F^{\geq \star}X$ can be realized as a
    geometric realization in the following way: take \[\F^{\geq \star}X_0:= \bigoplus_{i\geq 0}\bigoplus_{x \textrm{ a generator of }\pi_0(\mathrm{gr}^iX)}\bS_{\ev}(i)\]
    as the base case of an inductive construction.
    There exists lifts of all the elements of $\pi_0(\mathrm{gr}^iX)$ to elements $\pi_0(\F^{\geq i}X)$ by level-wise connectivity,
    so these elements taken together produce a map $\F^{\geq 0}X_0\to\F^{\geq \star}X$.
    Since $\F^{\geq \star}X$ is complete we have that this map induces a surjection $\pi_0(\F^{\geq
    \star}X_0)\to \pi_0(\F^{\geq i}X)$ for all $i$. 
    
    By the finite generation assumption we have that $\F^{\geq \star}X_0$ is a finite sum, and so
    it also is level-wise connective and has finitely generated associated graded homotopy terms.
    Therefore we also have that $\ker(\bigoplus_{i\geq 0}\pi_0(\F^{\geq i}X_0)\to \bigoplus_{i\geq
    0}\pi_0(\F^{\geq i}X))$ is also finitely generated. We may then inductively define $\F^{\geq
    \star}X_1$ to be the sum of twisted copies of $\bS_{\ev}$, one for each chosen generator of
    $\bigoplus_{n\in \bZ}\pi_1(\mathrm{gr}^nX)$ together with one for each generator of the kernel
    of the previous step, which taken all together is still finite.  This inductively builds a
    system $\F^{\geq \star}X_n$, $n\geq 0$, such that each $\F^{\geq \star}X_n$ is a finite sum of
    spectra of the form $\bS_{\ev}(i)$, $i\geq 0$, and such that $|\F^{\geq \star}X_n|\simeq\F^{\geq
    \star}X$.
\end{construction}

With this we are now able to prove the first step of lifting \cite[Theorem G]{CMM} to the synthetic setting.

\begin{lemma}~\label{lem: TC commutes with tensoring sometimes}
    If $\F^{\geq \star}X\in\CycSyn$ is connective in the neutral $t$-structure, complete, and
    satisfies the finite generation hypothesis of Construction~\ref{const: geometric realization of
    finite type synthetic spectra}, then the natural assembly map \[\mathrm{TC}(-)\otimes
    \F^{\geq \star}X\to \mathrm{TC}(-\otimes \F^{\geq \star}X^{\triv})\] is an
    equivalence.
\end{lemma}

\begin{proof}
    Note that the map in question is an equivalence for $\F^{\geq \star}X\simeq \bS_{\ev}(i)$ for
    some $i\geq 0$ since the same is true for $(-)^{\bT_{\ev}}$ and $(-)^{\t\bT_{\ev}}$. Thus by
    exactness the comparison map is an equivalence for $\F^{\geq \star}X$ a finite sum of synthetic
    spectra which are twists of $\bS_{\ev}$. Finally from the connectivy bounds on $\mathrm{TC}(-)$
    we have that $\mathrm{TC}(-)$ commutes with geometric realizations of uniformly (level-wise)
    bounded below synthetic cyclotomic spectra, which together with the previous construction then
    shows that $\mathrm{TC}(-)\otimes \F^{\geq \star}X\to \mathrm{TC}(-\otimes \F^{\geq
    \star}X^{\triv})$ is an equivalence for $\F^{\geq \star}X$ as in the Lemma statement.
\end{proof}

\begin{theorem}~\label{thm: CMM cocontinuity synthetic}
    The functor $\mathrm{TC}(-)/p:\mathrm{CycSyn}\to \mathrm{FD}(\bS)$ commutes with all colimits
    of synthetic cyclotomic spectra which are connective in the neutral $t$-structure.
\end{theorem}

\begin{proof}
    Note that in order to show that $\mathrm{TC}(-)/p$ commutes with all colimits it is enough to
    show this on each $\F^{\geq i}\mathrm{TC}(-)/p$ since evaluation at $i$ is cocontinuous and
    conservative. Then these are all connective spectra by assumption and the connectivity estimates
    on $\mathrm{TC}(-)$, so instead of checking mod $p$ we can check this equivalence after
    tensoring with $\bF_p$. Lifting this back to the level of synthetic spectra, this amounts to
    showing that $\mathrm{TC}(-)\otimes_{\bS_\ev} \bS_\ev\otimes_{\ins^0\bS}\ins^0\bF_p\simeq
    \mathrm{TC}(-\otimes (\bS_{\ev}\otimes_{\ins^0\bS}\ins^0\bF_p)^{\triv})$ commutes with all
    colimits of level-wise connective synthetic cyclotomic spectra.

    We will in fact show something stronger, that $\mathrm{TC}(-)$ commutes with all colimits as a
    functor from  $\mathrm{Mod}_{\ins^0\bF_p^{\triv}}(\mathrm{CycSyn})$. We have that
    \[\pi_{*,*}(\ins^0\bF_p^{\triv})^{\bT_{\ev}}\simeq \pi_{*,*}(\tau_{\geq
    2*}\mathrm{TC}^-(\bF_p^{\triv}))=\bF_p[x]\] and
    \[\pi_{*,*}(\ins^0\bF_p^{\triv})^{\t\bT_{\ev}}\simeq \pi_{*,*}(\tau_{\geq
    2*}\TP(\bF_p^{\triv}))=\bF_p[x^{\pm 1}]\] where $|x|=(-2, -1)$. The Frobenius is given by the
    double speed Postnikov filtration on the usual Frobenius, and so by \cite[Lemma 2.10]{CMM} it
    will also annihilate $x$. There is then, for any $\F^{\geq \star}X\in
    \mathrm{Mod}_{\ins^0\bF_p^{\triv}}(\mathrm{CycSyn})$, a functorial nullhomotopy of the composit
    \[(\F^{\geq \star}X)^{\bT_{\ev}}[-2](-1)\xrightarrow{x}(\F^{\geq
    \star}X)^{\bT_{\ev}}\xrightarrow{\varphi_{\F^{\geq \star}X}}(\F^{\geq \star}X)^{\t\bT_{\ev}}\] and therefore a functorial fiber sequence \[\mathrm{TC}(\F^{\geq \star}X)\to \F^{\geq \star}X\to (\F^{\geq \star}X)_{\bT_{\ev}}\] As the fiber of functors commuting with all colimits, it follows that $\mathrm{TC}(-)$ commutes with all colimits as well.
\end{proof}

During the proof of Theorem~\ref{thm: CMM cocontinuity synthetic} we produced a formula for
$\mathrm{TC}$ of $\ins^0\bF_p^{\triv}$-modules. This formula is also of interest separately and so we record it here.

\begin{proposition}~\label{prop: SBI in synthetic land}
    Let $\F^{\geq \star}X\in \mathrm{Mod}_{\ins^0\bF_p^{\triv}}(\mathrm{CycSyn})$.
    Then there is a natural fiber sequence \[\TC(\F^{\geq\star}X)\to\F^{\geq \star}X\to (\F^{\geq
    \star}X)_{\bT_{\ev}}.\] 
\end{proposition}

\subsection{The filtered cyclotomic spectra $\bZ_{C_{p,\ev}}$ and $j_{C_{p,\ev}}$}
We will now repeat the arguments of \cite[Section 2]{ammn} in our setting, which will establish a filtered version of their result. The first step in this analysis is to study the analogues of the cyclotomic spectra $\bZ_{hC_{p}}$ and $j_{hC_p}$ in our setting, where $j$ is the connective cover of the $K(1)$-local sphere. 

\begin{lemma}~\label{lem: filtered tate of j}
    If $j$ denotes the connective cover of $L_{K(1)}\bS$ with respect to an odd prime $p$, then
    $\F^{\geq \star}_{\ev, tC_p}j=\tau_{\geq 2\star-1}j^{tC_p}$.
\end{lemma}

\begin{proof}
    By Lemma~\ref{lem: even filtration comparison general} there is an equivalence $\F^{\geq \star}_{\ev, tC_p}j\simeq j_\ev^{tC_{p,\ev}}$, so we will prove the result for this later term instead. From \cite{Devinatz_Morava_COR}, we have that $\mathrm{gr}^ij_\ev$ is concentrated in degrees $[2i,2i-1]$, so by completeness we must have that $j_\ev\simeq \tau_{\geq 2\star -1}j$ which is $(-1)$-connective in the double speed Postnikov $t$-structure. Thus by Lemma~\ref{lem: fixed points are somehow also left t-exact I guess} we get that $j_\ev^{tC_{p,\ev}}$ is also $(-1)$-connective in the double speed Postnikov $t$-structure and there is therefore a map $j_{\ev}^{tC_{p,\ev}}\to \tau_{\geq 2\star-1}j^{tC_p}$. We also have by the definition of the even filtration that $\mathrm{gr}^i j_\ev^{tC_{p,\ev}}$ is $(2i)$-coconnective and these filtrations are exhaustive, so inductively the map $j_\ev^{tC_{p,\ev}}\to \tau_{2\star -1}j^{tC_p}$ must be an equivalence.
\end{proof}

As a consequence we get the following:

\begin{lemma}
    Let $p$ be an odd prime. The natural maps \[\tau_{\geq 0}^{\cyc,\N}\bZ_{p,\ev}^{tC_{p,\ev}}\to
    \F^{\geq \star}_\ev\mathrm{THH}(\bF_p)\] and \[\tau_{\geq 0}^{\cyc,\N} j_\ev^{tC_{p,\ev}}\to \F^{\geq \star}_\ev \mathrm{THH}(\bZ_p)\] are equivalences of synthetic
    cyclotomic spectra where $\tau^{\cyc,\N}_{\geq 0}$ denotes the $0$-connective cover with
    respect to the neutral $t$-structure on cyclotomic synthetic spectra.
\end{lemma}

\begin{proof}
    We have equivalences $\tau_{\geq 0}\bZ_p^{tC_p}\to \mathrm{THH}(\bF_p)$ and $\tau_{\geq 0}j^{tC_p}\simeq \mathrm{THH}(\bZ_p)$ of cyclotomic spectra, the first by \cite{nikolaus-scholze} and the second by work of Sanath Devalapurkar and Arpon Raksit (private communications). Since in all cases the filtrations in question are variants of the Postnikov filtrations the result follows.
\end{proof}

We then have fiber sequences \[(\bZ_{p,\ev})_{C_{p,\ev}}\to \bZ_{p,\ev}\to \F^{\geq
\star}_{\ev}\THH(\bF_p)\] and \[(j_\ev)_{C_{p,\ev}}\to \tau^{\cyc, N}_{\geq 0}j_\ev^{C_{p,\ev}}\to \F^{\geq
\star}_\ev\THH(\bZ_p)\] of $\bT_\ev$-modules. From Lemma~\ref{lem: Tate orbit lemma} we then have
that the natural maps \[(\bZ_{p,\ev})^{\t\bT_{ev}}\simeq (F_{\ev}^{\geq
\star}\THH(\bF_p))^{\t\bT_{\ev}}\] and \[(\tau_{\geq 0}^{\cyc, N}j_\ev^{C_{p,\ev}})^{\t\bT_{\ev}}\simeq
(\F^{\geq \star}_{ev}\THH(\bZ_p))^{\t\bT_{\ev}}\] are equivalences, where $\tau^{\cyc, N}$ is the neutral $t$-structure.

\begin{corollary}
    Let $\F^{\geq \star}X\in \CycSyn$. Then $\mathrm{TC}(\F^{\geq \star}X\otimes (\bZ_{p,\ev})_{C_{p,\ev}})\simeq (\F^{\geq \star}X\otimes (\bZ_{p,\ev})_{C_{p,\ev}})_{\bT_\ev}[1](1)\simeq (\F^{\geq \star}X\otimes (\bZ_{p,\ev})_{C_{p,\ev}})^{\bT_{\ev}}$.
\end{corollary}
\begin{proof}
    This amounts to showing that $(\F^{\geq \star}X\otimes (\bZ_{p,\ev})_{C_{p,\ev}})^{\t\bT_{\ev}}$ vanishes, but this is a module over $(\bZ_{p,\ev}^{C_{p,\ev}})^{\bT_{\ev}}\simeq 0$ and so must vanish as a module over the zero ring. 
\end{proof}

We conclude this subsection with the analogous result for $(\tau_{\geq 0}j_\ev)_{C_{p,\ev}}$.

\begin{construction}
    Let $L_{K(1)}^{\fil}\colon\SynSp\to \SynSp$ denote the functor $L_{K(1)}^{\fil}\F^{\geq \star}X:=(\F^{\geq \star}X\otimes_{\bS_\ev} J_{\ev})^\wedge_p$. Note that this is a filtration on the usual $K(1)$ localization of $\F^{\geq -\infty}X$, hence the name. This same definition also extends to any category which is tensored over $\SynSp$.
    
    Now let $p$ be an odd prime. We may alternatively describe the filtered $K(1)$-localization, at least mod $p$, as follows: let $v_1\in \pi_{2p-2}(\F^{\geq p-1}\bS_\ev/p)$ be a lift of the usual element $v_1\in \pi_{2p-2}(\bS/p)$. Such an element exists in the indicated degree either by directly studying the Adams-Novikov charts or by noting that the map $\bS_\ev\to J_\ev$ is a filtration on the usual map $\bS\to J$, that the element $v_1\in \pi_{2p-2}(J/p)$ lifts to an element $v_1\in \pi_{2p-2}(\F^{\geq p-1}J_\ev/p)$ since the even filtration on $J$ is $\tau_{\geq 2*-1}J$, and so there must exist an element of $\pi_{2p-2}(\F^{\geq p-1}\bS_\ev/p)$ mapping to this element up to higher order filtration degree which we may assume without loss of generality vanishes. We then have that $J_\ev/p\simeq \bS_{\ev}/p[v_1^{-1}]$, and so for a general $\F^{\geq \star}X$ we have that \[L_{K(1)}^{\fil}\F^{\geq \star}X/p\simeq \F^{\geq \star}X/p[v_1^{-1}]\] as in the non-synthetic case.
\end{construction}

\begin{proposition}[Filtered ambidexterity]
    Let $\F^{\geq \star}X\in \mathrm{Mod}_{j_{\ev}}(\SynSp_{\bT_{\ev}})$. Then $L_{K(1)}^{\fil}\F^{\geq \star}X^{tC_{p,\fil}}\simeq 0$. 
\end{proposition}
\begin{proof}
    We have that $L_{K(1)}^{\fil}X^{tC_{p,\ev}}$ is a module over $L_{K(1)}^{\fil}j_{\ev}^{tC_{p,\ev}}$, so it is enough to show that this term vanishes. This is a $p$-complete filtered spectrum and so it is enough to show that this vanishes mod $p$. Thus we reduce to showing that $(j_{\ev})^{tC_{p,\ev}}/p[v_1^{-1}]\simeq 0$ and so we need only show that $v_1$ acts nilpotently on all the filtered homotopy groups of $j_{\ev}^{tC_{p,\ev}}/p$. By Lemma~\ref{lem: filtered tate of j} the filtration on $j_{\ev}^{tC_{p,\ev}}$ will be $\tau_{\geq 2*-1}j^{tC_p}$, so the $v_1$ action on the filtered homotopy groups  $j_{\ev}^{tC_{p,\ev}}$ is exactly the usual action on the homotopy groups of $j^{tC_p}$ which is nilpotent by classical ambidexterity. 
\end{proof}

\begin{corollary}
    Let $\F^{\geq \star}X\in \CycSyn$. Then $L_{K(1)}^{\fil}(\F^{\geq \star}X\otimes (j_{\ev})_{C_{p,\ev}})^{tC_{p,\ev}}\simeq 0$.
\end{corollary}

\begin{lemma}\label{lem:j_vanishing}
    Let $\F^{\geq \star}X\in \SynSp_{\bT_{\ev}}$. Then $(\F^{\geq \star}X\otimes (j_\ev)_{C_{p,\ev}})^{tC_{p,\ev}}\simeq 0$.
\end{lemma}
\begin{proof}
    It is enough to show that $(\F^{\geq \star}X\otimes j_{C_{p,\ev}})^{tC_{p,\ev}}/(p,v_1)\simeq 0$
    since we already have that this vanishes $K(1)$-locally. We then have that this is equivalent
    to $(\F^{\geq \star}X\otimes (j_\ev/(p,v_1))_{C_{p,\ev}})^{tC_{p,\ev}}$. The result then follows
    from the fact that $j_{\ev}/(p,v_1)$ is in the thick subcategory generated by $\ins^0\bF_p$.
\end{proof}

\begin{corollary}~\label{cor: extended tate orbit lemma for j}
    Let $\F^{\geq \star}X\in \SynSp_{\bT_{\ev}}$. Then $\left((\F^{\geq \star}X\otimes
    j_{C_{p,\ev}})^{\t\bT_{\ev}}\right)^\wedge_p\simeq 0$
\end{corollary}
\begin{proof}
    It is enough to show that $(\F^{\geq \star}X\otimes j_{C_{p,\ev}})^{\t\bT_{\ev}}/p\simeq 0$.
    This is then equivalent to \[\left((\F^{\geq \star}X\otimes
    j_{C_{p,\ev}})^{tC_{p,\ev}}/p\right)^{\bT_{\ev}}\] which vanishes by
    Lemma~\ref{lem:j_vanishing}.
\end{proof}

We also have a map $j_{\ev}\to \tau_{\geq 0}^{\cyc,\N} j_{\ev}^{C_{p,\ev}}$ of synthetic cyclotomic spectra. Let $\F^{\geq \star}C$ denote the cofiber. We will also need the following Lemma.

\begin{lemma}
    Let $\F^{\geq \star}X\in \SynSp_{\bT_{\ev}}$. Then $L_{K(1)}^{\fil}(\F^{\geq \star}X\otimes \F^{\geq \star}C)^{tC_{p,\ev}}\simeq 0$. 
\end{lemma}
\begin{proof}
    $\F^{\geq\star}C$ is a $j_{\ev}$-module so we may apply filtered ambidexterity to get the result.
\end{proof}

\begin{corollary}~\label{cor: strong Tate orbit lemma for the fiber of j to thh(zp)}
    Let $\F^{\geq \star}X\in \SynSp_{\bT_{\ev}}$. Then $(\F^{\geq \star}X\otimes \F^{\geq \star}C/p)^{tC_{p,\ev}}\simeq 0$.
\end{corollary}
\begin{proof}
    Since we already have that this holds after applying $L_{K(1)}^{\fil}$ it is then enough to
    show this mod $v_1$. We have that $\F^{\geq \star}C/(p,v_1)$ is in the thick subcategory generated by $\ins^0\bF_p$, hence the result.
\end{proof}

\subsection{Pullback squares when topological periodic homology vanishes}
In this subsection we will use the results from the previous subsection to produce pullback squares relating topological cyclic homology and topological negative cyclic homology in the filtered setting, following the proofs in \cite[Section 2.2]{ammn}.

\begin{lemma}~\label{lem: pullback thh(fp)}
    Let $\F^{\geq \star}X\in \CycSyn$. Then is a pullback square
    \[
    \begin{tikzcd}
        \mathrm{TC}(\F^{\geq \star}X)\otimes \bZ_{p,\ev} \ar[r] \ar[d] & \mathrm{TC}(\F^{\geq
        \star}X\otimes F_{\ev}^{\geq *}\THH(\bF_p))\ar[d]\\
        (\F^{\geq \star}X\otimes \bZ_{p,\ev})^{\bT_{\ev}}\ar[r] & (\F^{\geq \star}X\otimes \F^{\geq \star}_{\ev}\THH(\bF_p))^{\bT_{\ev}}
    \end{tikzcd}
    \]
    where all the maps are the natural ones.
\end{lemma}

\begin{proof}
    We have that $\mathrm{TC}(\F^{\geq \star}X)\otimes \ins^0\bZ_{p,\ev}\simeq \mathrm{TC}(\F^{\geq \star}X\otimes \bZ_{p,\ev}^{\triv})$ by Lemma~\ref{lem: TC commutes with tensoring sometimes}. The top horizontal fiber is then identified with $\mathrm{TC}(\F^{\geq \star}X\otimes (\bZ_{p,\ev})_{C_p,\ev})$. This agrees with the bottom cofibers since the Tate construction on $\F^{\geq \star}X\otimes (\bZ_{p,\ev})_{C_{p,\ev}}$ vanishes.
\end{proof}

\begin{corollary}~\label{cor: filtered ammn}
    Let $R$ be a chromatically quasisyntomic ring. Then there is a commutative square 
    \[
    \begin{tikzcd}
        \F_{\HRW}^{\geq *}\mathrm{TC}(R;\bZ_p) \ar[r] \ar[d] & \F^{\geq \star}_{\HRW}\mathrm{TC}(R\otimes \bF_p;\bZ_p)\ar[d] \\
        \F^{\geq \star}_{\ev}\HC^{-}(R\otimes \bZ_p/\bZ_p)\ar[r] & \F^{\geq \star}_{\HRW}\mathrm{TP}(R\otimes \bF_p;\bZ_p)
    \end{tikzcd}
    \]
    refining the commutative square of \cite[Theorem 1.2]{ammn}. This square is also pullback upon rationalization.
\end{corollary}

\begin{proof}
    We will proceed in three steps starting from Lemma~\ref{lem: pullback thh(fp)}. Let $\F^{\geq \star}X$ be connective in the Postnikov $t$-structure, we will reduce later to the case of $\F^{\geq \star}X=\F^{\geq \star}_{\ev}\THH(R;\bZ_p)$. We may extend the pullback square of Lemma~\ref{lem: pullback thh(fp)} along the norm maps to get a commutative square 
    \[
    \begin{tikzcd}
        \TC(\F^{\geq \star}X)\otimes\bZ_{p,\ev} \ar[r] \ar[d] & \TC(\F^{\geq \star}X\otimes \F^{\geq \star}_{\ev}\THH(\bF_p)) \ar[d]\\
        (\F^{\geq \star}X\otimes \bZ_{p,\ev}^{\triv})^{\bT_{\ev}} \ar[r] & (\F^{\geq \star}X\otimes
        \F^{\geq \star}_{\ev}\THH(\bF_p))^{\t\bT_{\ev}}
    \end{tikzcd}
    \] and we have that $(\F^{\geq \star}X\otimes \F^{\geq \star}_{\ev}\THH(\bF_p))_{\bT_{\ev}}$
    rationally vanishes since inductively \[\pi_i^{\P}(\F^{\geq \star}X\otimes \F^{\geq
    \star}_{\ev}\THH(\bF_p))_{\bT_{\ev}}\] is $p^i$-torsion. We then have that $(\F^{\geq
    \star}X\otimes \bZ_{p,\ev})^{\t\bT_{\ev}}\to (\F^{\geq \star}X\otimes \F^{\geq
    \star}_{\ev}\THH(\bF_p))^{\t\bT_{\ev}}$ is an equivalence and therefore we get a commutative square 
    \[
        \begin{tikzcd}
            \TC(\F^{\geq \star}X)\otimes \bZ_{p,\ev} \ar[r] \ar[d] & \TC(\F^{\geq \star}X\otimes \F^{\geq \star}_{\ev}\THH(\bF_p))\ar[d]\\
            (\F^{\geq \star}X\otimes \bZ_{p,\ev})^{\bT_{\ev}}\ar[r] & (\F^{\geq
            \star}_{\ev}X\otimes \bZ_{p,\ev})^{\t\bT_{\ev}}
        \end{tikzcd}
    \]
    which when $\F^{\geq \star}X$ is connective in the Postnikov $t$-structure is also a rational pullback.

    The third and final step is that for $\F^{\geq \star}X=\F^{\geq \star}_{\ev}\THH(R;\bZ_p)$, we
    have rational equivalences $$\TC(\F^{\geq \star}_{\ev}\THH(R;\bZ_p))\to \TC(\F^{\geq
    \star}_{\ev}\THH(R;\bZ_p))\otimes \bZ_{p,\ev}$$ (which follows from the fact that $\bS_{\ev}\to
    \bZ_p$ is a $\bQ_p$-equivalence), the $\F^{\geq \star}_{\ev}\THH(-;\bZ_p)$ is symmetric
    monoidal (which follows from the same statement on $\THH(-)$ and on prismatic cohomology), and
    that $\F^{\geq \star}_{\ev}\THH(-)\otimes \bZ_{\ev}\simeq \F^{\geq \star}\HH(-)$ (which follows
    from the same statement on $\THH(-)$).
\end{proof}

\begin{example}
    Consider now the case of $R=\bS$ in Corollary~\ref{cor: filtered ammn}. This produces a pullback square 
    \[
    \begin{tikzcd}
        \F^{\geq \star}_{\HRW}\TC(\bS;\bQ_p) \ar[r] \ar[d] & \ins^0(\bQ_p\oplus \bQ_p[-1]) \ar[d]\\
        \tau_{\geq 2\star} \bQ_p[v] \ar[r] & \bQ[v^{\pm 1}]
    \end{tikzcd}
    \] where $|v|=(-2,-1)$. Consequently we find that $\F^{\geq \star}\TC(\bS;\bQ_p)\simeq \bQ_p(0) \oplus \bigoplus_{i=0}^\infty \bQ_p[2i-1](i)$.
\end{example}

We now turn our attention to the case of $\F^{\geq \star}_{\ev}\THH(\bZ_p)$.

\begin{lemma}~\label{lem: pullback thh(zp)}
    Let $\F^{\geq \star}X\in \CycSyn$ and let $p$ be an odd prime. Then there is a pullback square
    \[
    \begin{tikzcd}
        \TC(\F^{\geq \star}X;\bZ_p)\otimes j_{\ev} \ar[r] \ar[d] & \TC(\F^{\geq \star}X\otimes \F^{\geq \star}_{\ev}\THH(\bZ_p);\bZ_p) \ar[d]\\
        (\F^{\geq \star}X\otimes j_{\ev})^{\bT_{\ev}} \ar[r] & (\F^{\geq \star}X\otimes \F^{\geq \star}_{\ev}\THH(\bZ_p))^{\bT_{\ev}}
    \end{tikzcd}
    \]
    where all the maps are the usual ones. 
\end{lemma}
\begin{proof}
    The proof of this result follows in the exact same way as Lemma~\ref{lem: pullback thh(fp)}, using Corollary~\ref{cor: extended tate orbit lemma for j} and Corollary~\ref{cor: strong Tate orbit lemma for the fiber of j to thh(zp)} to identify the horizontal fibers.
\end{proof}

In~\cite{devalapurkar-raksit}, Devalapurkar and Raksit prove that $\tau_{\geq 0}j^{\t
C_p}\we\THH(\bZ_p;\bZ_p)$ as cyclotomic spectra. We obtain a synthetic version of their result.

\begin{corollary}~\label{cor: filtered DR}
    Let $\F^{\geq \star}X\in \CycSyn_p$ where $p$ is an odd prime. Then there is a fiber sequence \[L_{K(1)}^\fil \F^{\geq \star}X_{\bT_\ev}(1)[1] \to L_{K(1)}^\fil \TC(\F^{\geq \star}X)\to L_{K(1)}^\fil\TC(\F^{\geq \star}X\otimes \F^{\geq \star}_\ev\THH(\bZ_p))\] of synthetic spectra.
\end{corollary}
\begin{proof}
    The pullback square of Lemma~\ref{lem: pullback thh(zp)} induces a fiber sequence of the form $(\F^{\geq \star}X\otimes \F^{\geq \star}Y)^{\bT_\ev}\to \TC(\F^{\geq \star}X)\otimes j_\ev\to \TC(\F^{\geq \star}X\otimes\THH_{\ev}(\bZ_p);\bZ_p)$ where $\F^{\geq \star}Y:= \mathrm{fib}(j_\ev\to \THH_\ev(\bZ_p;\bZ_p))$. The result then follows if we can show that $L_{K(1)}^\fil (\F^{\geq \star}X\otimes \F^{\geq \star}Y)^{\bT_\ev}\simeq \F^{\geq \star}X_{\bT_{\ev}}[1](1)$.

    We will prove this in two steps. The first is that $L_{K(1)}^\fil(\F^{\geq \star}X\otimes
    \F^{\geq \star}Y)^{\t\bT_\ev}\simeq 0$. This follows from Corollary~\ref{cor: extended tate orbit lemma for j} and Corollary~\ref{cor: strong Tate orbit lemma for the fiber of j to thh(zp)}. Thus the norm map \[L_{K(1)}^\fil(\F^{\geq \star}X\otimes \F^{\geq \star}Y)_{\bT_{\ev}}[1](1)\to L_{K(1)}^\fil (\F^{\geq \star}X\otimes \F^{\geq \star}Y)^{\bT_{\ev}}\] is an equivalence. 

    Note that since $L_{K(1)}^\fil$ is given by a $p$-complete base-change, it will commute with $(-)_{\bT_\ev}$. Thus the result will follow if we can show that the maps 
    \[
    \begin{tikzcd}
         & \F^{\geq \star}X\otimes \bS_\ev \ar[d]\\
         \F^{\geq \star}X\otimes \F^{\geq \star}Y \ar[r] & \F^{\geq \star}X\otimes j_\ev
    \end{tikzcd}
    \]
    are equivalences after applying $L_{K(1)}^\fil(-)$. The bottom map is an equivalence since the cofiber is $\F^{\geq \star}X\otimes \THH_\ev(\bZ_p;\bZ_p)$ which after applying $L_{K(1)}^\fil(-)$ will be a $L_{K(1)}^\fil \THH_\ev(\bZ_p;\bZ_p)\simeq 0$-module. The vertical map is an equivalence since we may check this modulo $p$ where this becomes the statement that $\bS_{\ev}/p[v_1^{-1}]\simeq (\tau_{\geq 0}^{\cyc, N}\bS_\ev/p[v_1^{-1}])[v_1^{-1}]$ which follows from the fact that inverting an element in positive homotopical degree and weight only depends on the connective cover in the neutral $t$-structure.
\end{proof}

\begin{example}
    Consider the case of $\F^{\geq \star}X=\F^{\geq \star}_{\ev}\THH(\bS)\simeq \bS_{\ev}$. Then Corollary~\ref{cor: filtered DR} gives a fiber sequence \[L_{K(1)}^{\fil}\bS_\ev[\B\bT_\ev][1](1)\to L_{K(1)}^\fil\TC(\bS_\ev)\to L_{K(1)}^\fil\TC(\THH_\ev(\mathbb{Z}_p))\] where $\B\bT_\ev:= \bS_\ev\otimes_{\bT_\ev}\bS_\ev$. 
\end{example}

\small
\bibliographystyle{amsplain}
\bibliography{cycsyn}

\medskip
\noindent
\textsc{Department of Mathematics, Northwestern University}\\
\textsc{Max-Planck-Institut für Mathematik Bonn}\\
{\ttfamily antieau@northwestern.edu}

\medskip
\noindent
\textsc{Department of Mathematics, Northwestern University}\\
{\ttfamily noah.riggenbach@northwestern.edu}

\end{document}